\documentclass{mpaper}

\author{Giovanni Sold\`a \and Manlio Valenti}
\date{}

\include{symbols}
\include{thref}

\newcommand{\printauthor}{{
\footnotesize

Giovanni Sold\`a, \textsc{Department of Mathematics: Analysis, Logic and Discrete Mathematics, \newline 
Ghent University, \newline 
Krijgslaan 281 S8, 9000 Ghent, BE}\par\nopagebreak
\textit{E-mail address}: \url{giovanni.a.solda@gmail.com}

\medskip

Manlio Valenti, \textsc{Department of Mathematics, Computer Science and Physics\newline
University of Udine\newline
Udine, UD 33100, IT}\par\nopagebreak
\textsc{Current address: Department of Mathematics\newline University of Wisconsin – Madison\newline
Madison (WI), 53706, USA}\par\nopagebreak\textit{E-mail address}: \url{manliovalenti@gmail.com}
}}

\title{Algebraic properties of the first-order part of a problem}

\begin{document}

\maketitle

\begin{abstract}
	In this paper we study the notion of first-order part of a computational problem, first introduced in \cite{DSYFirstOrder}, which captures the ``strongest computational problem with codomain $\mathbb{N}$ that is Weihrauch reducible to $f$". This operator is very useful to prove separation results, especially at the higher levels of the Weihrauch lattice. We explore the first-order part in relation with several other operators already known in the literature. We also introduce a new operator, called \emph{unbounded finite parallelization}, which plays an important role in characterizing the first-order part of parallelizable problems. We show how the obtained results can be used to explicitly characterize the first-order part of several known problems.
\end{abstract}


\tableofcontents

\section{Introduction}

When working with reducibilities in computability theory, a standard strategy to show that $a \not \le b$ is to explicitly present an element $c$ such that $c\le a$ and $c\not\le b$. This is not specific to any particular reducibility notion $\le$, as it only exploits the transitivity of the order. Of course, there is no canonical way to choose such a $c$ in general, and this is very dependent on the nature of the problems under analysis. 

In the context of computable analysis and Type-$2$ Theory of Effectivity, we usually work with multi-valued functions on represented spaces and use Weihrauch reducibility to compare their uniform computational strength. In the attempt to provide ``simple" witnesses to a non-reduction $f\not\weireducible g$, it is natural to look for a multi-valued function $h$ with codomain $\mathbb{N}$ that belongs to the lower cone of $f$ but not to the lower cone of $g$. 
This is very common in the literature, but all the proofs require ad-hoc strategies and it is hard to collect them within the same framework. Recently, Dzhafarov, Solomon, and Yokoyama \cite{DSYFirstOrder} suggested studying the lower cone of a computational problem from a more algebraic point of view. To this end, they introduced the notion of \emph{first-order part} of a problem $f$, capturing the ``strongest multi-valued function with codomain $\mathbb{N}$ that is Weihrauch reducible to $f$". 

In this paper, we explore the algebraic properties of the operator that maps a problem to its first-order part. After introducing the relevant background notions on Weihrauch reducibility (Section~\ref{sec:background}), we formally introduce the first-order part operator (Section~\ref{sec:fop}) and highlight some examples where the first-order part of a problem has been (implicitly) used in the literature (Section~\ref{sec:literature}).

The literature on Weihrauch reducibility enjoys a wide variety of different operators on multi-valued functions, and the algebraic structure of the Weihrauch degrees is well-studied. In Section~\ref{sec:alg_fop}, we study the relation of the first-order part with the most common operators and highlight the algebraic connections between them.

In order to characterize the connections between the first-order part and the parallelization operators, in Section~\ref{sec:u*} we introduce a new operator $\ustar{(\cdot)}$, which intuitively captures the idea of using a problem a finite number of times in parallel, but without having to commit in advance to the exact number of instances to use. In particular, using the unbounded finite parallelization, we characterize the first-order part of the problems that are Weihrauch-equivalent to the parallelization of functions with codomain $\mathbb{N}$. Moreover, in the same spirit of Section~\ref{sec:alg_fop}, in Section \ref{sec:alg_u*}, we study the relation between $\ustar{(\cdot)}$ and the most common operators on multi-valued functions. 

Section~\ref{sec:fop_diamond} will be devoted to the connections between the first-order part operator, the unbounded finite parallelization operator, and the diamond operator, which roughly corresponds to closure under compositional product. 

Finally, in Section~\ref{sec:applications} we show how our results can be applied to characterize the first-order part of several well-known problems, including in particular $\mflim^{(n)}$ and $\WKL^{(n)}$. Moreover, we provide some bounds for the first-order part of $\RT{2}{2}$, in particular answering a question raised by Brattka and Rakotoniaina \cite{BRramsey17}.

\subsection*{Acknowledgments}
We would like to thank Vittorio Cipriani, Damir Dzhafarov, Alberto Marcone, Arno Pauly, Paul Shafer, and Keita Yokoyama for useful discussions and many valuable suggestions during the preparation of the draft. We would also like to thank the anonymous reviewer for his/her very careful reading of the paper and the many valuable suggestions.  

Soldà's research was partially supported by a London Mathematical Society Early Career Fellowship (year 2021) and by the grant FWO Odysseus type II, "Recursion, reflection, and second-order arithmetic", G0F8421N. Valenti's research was partially supported by the Italian PRIN 2017 Grant ``Mathematical Logic: models, sets, computability".

\section{Background}
\label{sec:background}

We briefly recall the main notions in Type-2 Theory of Effectivity and Weihrauch reducibility that will be needed in this paper. For a more thorough presentation, the reader is referred to \cite{BGP17, Weihrauch00}.

We write $\Baire$ (resp.\ $\Cantor$) for the Baire space (resp.\ Cantor space) endowed with the product topology. We also write $\baire$ (resp.\ $\cantor$) for the set of finite sequences of natural numbers (resp.\ binary sequences). In particular, we write $\str{n_0, n_1,\hdots, n_k }$ for the string $\sigma:= i \mapsto n_i$ (we write $\str{}$ to denote the empty string). Similarly, we denote an infinite string by $\str{n_0, n_1,\hdots}$, when it is clear from the context how to continue the sequence. We write $\length{\sigma}$ for the length of $\sigma$ and $\sigma\concat \tau$ for the concatenation of the strings $\sigma$ and $\tau$. We use the symbol $\prefix$ to denote the prefix relation, and for every string $x$ and every $n\in\mathbb{N}$, we denote with $x[n]$ the prefix of $x$ of length $n$. We will use the symbol $\coding{\cdot}$ to denote a fixed computable bijection $\function*{\baire}{\mathbb{N}}$ with computable inverse. It is often convenient to write $\coding{n_0,\hdots, n_k}$ in place of $\coding{\str{n_0,\hdots, n_k}}$. We assume that the bijection enjoys all the usual computability properties, such as $\sigma \mapsto \length{\sigma}$ being computable. In the literature, the symbol $\coding{\cdot}$ is often used to denote also the \textdef{join} between two (finite or infinite) strings with the same length. With a small abuse of notation, if $x_0,\hdots,x_{k-1}$ are $k$ strings of the same length we will write $\coding{x_0,\hdots, x_{k-1}}(j):= \coding{x_0(j),\hdots, x_{k-1}(j)}$. Moreover, if $\sequence{x_i}{i\in\mathbb{N}}$ is a sequence of infinite strings we define $\coding{x_0,x_1,\hdots}(\coding{i,j}) := x_i(j)$. It is sometimes convenient to abbreviate $\coding{x_0,\hdots, x_{k-1}}$ with $\pairing{x_i}_{i<k}$ (analogously we write $\pairing{x_i}_{i\in\mathbb{N}}$ for the join of infinitely many strings).

We fix a computable enumeration of $\sequence{\Phi_e}{e\in\mathbb{N}}$ of partial computable functionals from $\Baire$ to $\Baire$. In the following, we use $\UTM$ to denote a fixed universal Type-2 functional, i.e.\  $\UTM\pfunction{\Baire\times \Baire}{\Baire}$ is s.t.\ for every continuous function $F\pfunction{\Baire}{\Baire}$ there is $w\in\Baire$ s.t.\ for every $x\in\Baire$, $\UTM(w,x) = F(x)$. In this case we say that $w$ is an index for $F$. Since continuous functionals are computable relatively to an oracle, if $w= \str{e}\concat p$ it is convenient to think of the $w$-th continuous functional as the map $q\mapsto \Phi_e(\coding{p,q})$. In particular, this induces a listing $\sequence{\Phi_p}{p\in\Baire}$ of the partial continuous functionals\footnote{To be precise, it is a listing of the continuous functionals with $G_\delta$ domain.} $\Baire\to\Baire$. We hope that the context will dispel any ambiguity between $\Phi_e$ and $\Phi_p$. 

A \textdef{represented space} is a pair $(X,\repmap{X})$ where $X$ is a set and $\repmap{X}\pfunction{\Baire}{X}$ is a surjective function called \textdef{representation map}. For every $x\in X$ we say that $\repmap{X}^{-1}(x)$ is the set of \textdef{$\repmap{X}$-names} or \textdef{$\repmap{X}$-codes} for $x$. We will avoid mentioning explicitly the representation map whenever it is clear from the context. If $(X,\repmap{X})$ and $(Y,\repmap{Y})$ are represented spaces, a \textdef{realizer} for $f\pmfunction{X}{Y}$, denoted $F \vdash f$, is a function $F\pfunction{\Baire}{\Baire}$ s.t.\ $(\forall p \in \dom(f\circ \repmap{X}))(\repmap{Y}(F(p))\in f(\repmap{X}(p)))$. Realizers allow us to transfer properties of functions on the Baire space (such as computability or continuity) to multi-valued functions on represented spaces. In particular, we say that a multi-valued function between represented spaces is computable if it has a computable realizer.

Computational problems are formalized via multi-valued functions on represented spaces, and, throughout the paper, we will use the words (computational) problem and multi-valued function interchangeably. Weihrauch reducibility is a notion of reducibility that calibrates the uniform computational strength of computational problems. If $f\pmfunction{X}{Y}$ and $g\pmfunction{Z}{W}$ are multi-valued functions, we say that $f$ is \textdef{Weihrauch reducible} to $g$, and write $f\weireducible g$, if there are two computable functionals $\Phi,\Psi\pfunction{\Baire}{\Baire}$ s.t.\ 
\[ (\forall G \vdash g)( (p\mapsto \Psi(p, G\Phi(p))) \vdash f).   \]
We say that $f$ is \textdef{strongly Weihrauch reducible} to $g$, and write $f\strongweireducible g$, if $\Psi$ does not have direct access to $p$. In symbols, $f\strongweireducible g$ if there are two computable functionals $\Phi$ and $\Psi$ s.t.\ 
\[ (\forall G \vdash g)( \Psi G \Phi \vdash f).   \]
The map $\Phi$ computing names of inputs of $g$ from names of inputs of $f$ is called \textdef{forward functional}, while $\Psi$ is called \textdef{backward functional}. Unless otherwise mentioned, we will denote the forward functional with $\Phi$, and use $\Psi$ for the backward functional.

Both Weihrauch reducibility and strong Weihrauch reducibility are quasi-orders, and hence they induce two degree structures on the class of computational problems. There are several natural operations on multi-valued functions, and most of them lift to the Weihrauch degrees and the strong Weihrauch degrees. Below we formally introduce the ones that we need in this paper.

The Weihrauch degrees form a distributive lattice, with join $\sqcup$ and meet $\sqcap$ defined as
\begin{itemize}
	\item $(f\sqcup g)(i,x):= \{i\}\times f(x)$ if $i=0$ and $(f\sqcup g)(i,x):= \{i\}\times g(x)$ if $i=1$;
	\item $(f\sqcap g)(x,z):= \{0\} \times f(x) \cup \{1\}\times g(z)$.
\end{itemize}

We define the \textdef{parallel product} $f \times g$ as $(f \times g)(x,y) = f(x) \times g(y)$. The parallel product captures the idea of using $f$ and $g$ in parallel. We will write $f^n$ to denote the parallel product of $n$ copies of $f$ (i.e.\ $f^1:=f$, $f^2:=f\times f$, and so on). The \textdef{finite parallelization} allows us to apply a problem a finite number of times in parallel, and is defined as $f^*(\sequence{x_i}{i<n}):=\{ \sequence{y_i}{i<n} \st (\forall i<n)(y_i \in f(x_i)) \}$. Its infinite generalization is called \textdef{parallelization}, and is formally defined as the problem $\parallelization{f}:=\sequence{x_n}{n\in\mathbb{N}}\mapsto \{ \sequence{y_i}{i\in\mathbb{N}} \st (\forall i\in\mathbb{N})(y_i \in f(x_i)) \}$. In other words, given a countable sequence of $f$-instances, $\parallelization{f}$ asks for a $f$-solution for each given $f$-instance. In Section~\ref{sec:u*}, we will introduce a new operator, called \textdef{unbounded finite parallelization}, where we ask for a finite number of instances of $f$ in parallel, but without having to commit in advance to the exact number of instances.

Two multi-valued functions $\varphi\pmfunction{X}{Y}$ and $\psi\pmfunction{Z}{X}$ can be composed letting $(\varphi \circ \psi)(x) := \varphi(\psi(z)) = \bigcup_{x\in \psi(z)} \varphi(x)$ with $\dom(\varphi \circ \psi) := \{z \in \dom(\psi) \st \psi(z) \subseteq \dom(\varphi)\}$. However, the composition does not lift to Weihrauch degrees, as it requires a precise matching between the codomain of $\psi$ and the domain of $\varphi$. In general, if $f\pmfunction{X}{Y}$ and $g\pmfunction{Z}{W}$ are two arbitrary multi-valued functions, to capture the idea of applying $f$ and $g$ in series, we consider instead the \textdef{compositional product} $f\compproduct g$, defined as 
\[ f\compproduct g := \max_{\weireducible} \{ f_0\circ g_0 \st f_0 \weireducible f \text{ and } g_0 \weireducible g\}.  \]
The compositional product was first introduced in \cite{BolWei11}, and proven to be well-defined in \cite{BP16}. While $f\compproduct g$ is (formally speaking) a Weihrauch degree, it is convenient to identify it with its representative whose domain is 
\[ \{ (p,z)\in \Baire \times Z \st z \in \dom(g) \text{ and }(\forall q \in \repmap{W}^{-1}(g(z)))(\repmap{X}\Phi_p(q)\in \dom(f)) \} \] 
and whose output is any pair $(w,y)$ with $w\in g(z)$ and $y \in f(\repmap{X}\Phi_p\repmap{W}^{-1}(w))$ (see \cite{Westrick20diamond} for a short proof of the fact that this problem is a representative of the compositional product). For every $f$, we denote with $f^{[n]}$ the $n$-fold iteration of the compositional product of $f$ with itself, i.e., $f^{[1]} = f$, $f^{[2]} = f \compproduct f$, and so on. By unfolding the definition, it is straightforward to see that we can equivalently think of $f^{[n]}$ as the problem that takes in input $(\sequence{p_i}{i<n-1},x)$, where $x\in \dom(f)$ and $p_i\in \Baire$, and produces $\sequence{y_i}{i<n}$ with $y_0\in f(x)$ and $y_{i+1} \in f(\repmap{X}\Phi_{p_i}\repmap{Y}^{-1}(y_i) )$. In Section~\ref{sec:fop_diamond}, we will formally introduce the \emph{diamond} operator, which roughly captures the closure under compositional product.

The operators $\sqcup,\sqcap,\times,{}^*,$ and $~ \parallelization{} ~$ lift to the Weihrauch and strong Weihrauch degrees. We say that $f$ is a \textdef{cylinder} if $f \strongweiequiv f \times \id$.
If $f$ is a cylinder, then $g \weireducible f$ if and only if $g \strongweireducible f$ (\cite[Cor.\ 3.6]{BG09}). 

We now mention a couple of additional non-degree-theoretic operators which are nevertheless very important in the literature. 

The \textdef{jump} of $f \pmfunction{X}{Y}$ is the problem $f' \pmfunction{X'}{Y}$ defined as $f'(x) := f(x)$, where $X'$ is the represented space $(X,\repmap{X'})$ and $\repmap{X'}$ maps a convergent sequence $\sequence{p_n}{n\in\mathbb{N}}$ in $\Baire$ to $\repmap{X}(\lim_{n\to\infty} p_n)$. In other words, $f'$ takes in input a sequence that converges to a name of an $f$-instance and produces an $f$-solution to that instance. We use $f^{(n)}$ to denote the $n$-th jump of a problem. Observe that, if we define $\mflim \pfunction{(\Baire)^\mathbb{N}}{\Baire}$ by $\mflim((p_n)_{n \in \mathbb{N}}) := \lim_{n \to \infty} p_n$ then we have $f^{(n)}\weireducible f\compproduct \mflim^{[n]}$. The converse reduction does not hold in general (take e.g.\ a function $f$ that only has computable outputs). The equivalence $f^{(n)}\weiequiv f\compproduct \mflim^{(n)}$ holds if $f$ is a cylinder, hence in particular $\mflim^{[n+1]}\weiequiv \mflim^{(n)}$.

The \textdef{total continuation} or \textdef{totalization} of $f\pmfunction{X}{Y}$ is the total multi-valued function $\totalization{f}\mfunction{X}{Y}$ defined as 
\[ \totalization{f}(x):=\begin{cases}
	f(x)	& \text{if } x\in\dom(f)\\
	Y		& \text{otherwise.}
\end{cases} \]

Notice that the definition of $\mathsf{T}f$ is sensitive to the particular definition of $f$ as a multi-valued function between represented spaces. We also mention the \emph{completion} of a problem, which is an operator that maps a multi-valued function to a total one on different represented spaces. The precise definition will be stated in Section~\ref{sec:fop_diamond}. For a more detailed exposition we refer to \cite{BGCompOfChoice19}.

Next, we introduce some problems which are milestones in the Weihrauch lattice. We already mentioned $\mflim$ and its jumps. Of utmost importance is the family of \emph{choice} problems: every represented space $(X,\repmap{X})$ can be endowed with the final topology induced by the map $\repmap{X}$. The space of open subsets of $X$ can be equipped with a representation map using the fact that a subset $U\subset X$ is open iff its characteristic function $\chi_U\function{X}{\Sier}$ has a continuous realizer, where $\Sier=\{0,1\}$ is equipped with the Sierpi\'nski topology. In turn, we represent a closed set (in the final topology on $X$) via a name for its complement. This intuitively corresponds to having a c.e.\ procedure (relative to the code of the set) to decide membership in the complement. With a similar strategy, we can define a representation map for every level of the Borel hierarchy on $X$, as well as for the families $\boldfaceSigma^1_1(X)$, $\boldfacePi^1_1(X)$, $\boldfaceDelta^1_1(X)$  (see \cite[Sec.\ 4]{Pauly16}).

For every level $\boldfaceGamma$ of the Borel hierarchy, or $\boldfaceGamma\in \{ \boldfaceSigma^1_1,\boldfacePi^1_1, \boldfaceDelta^1_1 \}$ we define $\codedChoice{\boldfaceGamma}{}{X}\pmfunction{\boldfaceGamma(X)}{X}$ as the problem of choosing an element from a non-empty set $A\in \boldfaceGamma(X)$. If $\boldfaceGamma=\boldfacePi^0_1$ we simply write $\Choice{X}$. Despite the choice problems can be defined in such generality, we will only mention the choice problems on the sets $k\in \mathbb{N}$, $\mathbb{N}$, $\mathbb{R}$, $\Cantor$, and $\Baire$. In particular, $\CCantor\strongweiequiv \WKL$, where $\WKL$ stands for Weak K\"onig's Lemma and is the problem of producing a path through an ill-founded subtree of $\cantor$. It is known that, for every $n>0$, $\codedChoice{\boldfacePi^0_{n+1}}{}{\mathbb{N}}\strongweiequiv \codedChoice{\boldfacePi^0_n}{}{\mathbb{N}}{}'$. 

There are several variants of the choice problems where we restrict our attention to (non-empty) sets with additional properties. In particular, we write $\codedUChoice{\boldfaceGamma}{}{X}$ if the choice is restricted to singletons. Other useful ones are the \textdef{cofinite choice} principles, where the choice is restricted to cofinite sets. It is easy to see that the cofinite choice problem for $\boldfaceGamma$ subsets of $\mathbb{N}$ is equivalent to the problem $\CodedBound{\check{\boldfaceGamma}}$, where $\check{\boldfaceGamma}$ is the dual class, consisting of finding a bound for a finite subset of $\mathbb{N}$. The problems $\CNatural, \CCantor, \CBaire,$ and $\UCBaire$ are closed under compositional product (see \cite[Cor.\ 7.12]{BGP17}). We also mention the following well-known fact that will be useful in the rest of the paper:

\begin{theorem}
    \thlabel{thm:cn_finite_codomain}
    Fix $p\in \Baire$. Let $\weireducible^p$ denote the relativized version of Weihrauch reducibility where the forward and backward functionals $\Phi$, $\Psi$ are required to be $p$-computable.
    For every $p\in\Baire$ and every problem $f$ with finite codomain, $\CNatural\not\weireducible^p f$.
\end{theorem}
\begin{proof}
    Let $Y = \{y_0,\hdots, y_n\}$ be the codomain of $f$ and let $q_i$ be a name for $y_i$. Assume towards a contradiction that there are two $p$-computable functionals $\Phi, \Psi$ witnessing $\CNatural\weireducible^p f$. Let $0^\mathbb{N}$ be a name for $\mathbb{N}\in\dom(\CNatural)$. By continuity of $\Psi$, there is $k_0$ s.t.\ for every $i\le n$, if $\Psi(0^\mathbb{N}, q_i)(0)\downarrow$ then it does in $k_0$ steps. Consider now the input $0^{k_0}\concat \str{x_0,\hdots, x_n}\concat 0^\mathbb{N}$, where $x_i := \Psi(0^\mathbb{N}, q_i)(0)+1$ if the computation converges in $k_0$ steps, and $0$ otherwise. By iterating this reasoning we can diagonalize against every possible output of $f$, obtaining an input for which the pair $\Phi, \Psi$ fails to compute a valid solution.
\end{proof}

Another important family of problems comes from Ramsey's theorem for $n$-tuples and $k$-colors: for every $A\subset \mathbb{N}$, let $[A]^n:=\{ B\subset A\st |B|=n\}$ be the set of subsets of $A$ with cardinality $n$. For $k\ge 2$, a \textdef{$k$-coloring} of $[\mathbb{N}]^n$ is a function $c\function{[\mathbb{N}]^n}{k}$. An infinite set $H$ s.t.\ $c([H]^n)=\{i\}$ for some $i<k$ is said to be \textdef{$c$-homogeneous for color $i$}. A $k$-coloring $c$ of $[\mathbb{N}]^n$ can be represented by a string $p\in\Baire$ s.t.\ for each $(i_0,\hdots,i_{n-1})\in[\mathbb{N}]^n$, $p(\coding{i_0,\hdots,i_{n-1}})=c(i_0,\hdots,i_{n-1})$. We denote the represented space of $k$-colorings of $n$-tuples with $\mathcal{C}_{n,k}$.

We define $\RT{n}{k}\mfunction{\mathcal{C}_{n,k}}{\Cantor}$ as the total multi-valued function that maps a coloring $c$ to the set of all $c$-homogeneous sets. Similarly we define $\RT{n}{\mathbb{N}}\mfunction{\bigcup_{k\ge 1}\mathcal{C}_{n,k} }{\Cantor}$ as $\RT{n}{\mathbb{N}}(c):=\RT{n}{k}(c)$, where $k-1$ is the maximum of the range of $c$. Note that the input for $\RT{n}{\mathbb{N}}$ does not include information on which colors appears in the range of the coloring. 

A coloring $c\function{[\mathbb{N}]^n}{k}$, with $n\geq 1$ is called \textdef{stable} if, for every $x\in [\mathbb{N}]^{n-1}$, $\lim_{y\to\infty} c(x\cup \{y\})$ exists. We denote with $\SRT{n}{k}$ (resp.\ $\SRT{n}{\mathbb{N}}$) the restriction of $\RT{n}{k}$ (resp.\ $\RT{n}{\mathbb{N}}$) to stable colorings.

The uniform computational content of Ramsey's theorems is well-studied (see e.g.~\cite{BRramsey17,DDHMS16, DGHPP18,Patey2016}). 

We finally mention the following two problems:

\begin{itemize}
	\item $\LPO\function{\Cantor}{\{0,1\}}$ is defined as $\LPO(p):=1$ iff $(\exists n)(p(n)=1)$. It is often convenient to think of $\LPO$ as the problem of finding a yes/no answer to a $\Sigma^{0,p}_1$ or $\Pi^{0,p}_1$ question. 
	\item $\chiPi\function{\Cantor}{\{ 0,1\}}$ is the characteristic function of (the set of characteristic functions for) well-founded trees.
\end{itemize} 
It is known that, for every $n$, $\parallelization{\LPO^{(n)}}\strongweiequiv \mflim^{(n)}$. Moreover, $\parallelization{\chiPi}\strongweiequiv \PiCA$, where $\PiCA$ is the problem of producing the characteristic function of $A\subset \mathbb{N}$ given a $\boldfacePi^1_1$-code for it (see also \cite{HirstLM}).

A very important result in classical computability theory is Kleene's fixed point theorem, also called recursion theorem (see e.g.\ \cite[Sec.\ 2.2]{Soare16}). An important observation is that the recursion theorem relativizes to continuous functionals (by essentially the same proof).

\begin{theorem}
	\thlabel{thm:continuous_recursion_theorem}
	For every total continuous map $F\function{\Baire}{\Baire}$ there is $w\in \Baire$ s.t.\ $\Phi_{F(w)} = \Phi_w$. Moreover, $w$ can be found uniformly from $F$.
\end{theorem}

We conclude this section with the following lemma, which will be useful in the rest of the paper.

\begin{lemma}
	\thlabel{thm:countable_sparse_splitting}
	There are two sequences $\sequence{A_n}{n\in\mathbb{N}}$ and $\sequence{B_n}{n\in\mathbb{N}}$ of subsets of $\mathbb{N}$ s.t.\ 
	\begin{enumerate}
		\item for every $n$, $1^n0 \prefix B_n$;
		\item for every $n$, $\emptyset' \not\turingreducible A_n$, $\emptyset' \not\turingreducible B_n$, but $\emptyset' \turingreducible A_n \oplus B_n$;
		\item for every computable sequence $\sequence{e_i}{i\in\mathbb{N}}\subset \mathbb{N}$ and every computable functional $\Psi$ there is $x\in\mathbb{N}$ s.t.\ if for every $i$, $\{e_i\}^{B_i}(x)\downarrow = \gamma_i \prefix B_i$ then 
		\[ \emptyset'(x) \neq \Psi(x, \pairing{\gamma_i}_{i\in\mathbb{N}}, \pairing{A_j}_{j\in\mathbb{N}} ). \]
	\end{enumerate}
\end{lemma}
\begin{proof}
	The construction of the sets $\sequence{A_n}{n\in\mathbb{N}}$ and $\sequence{B_n}{n\in\mathbb{N}}$ is done using a finite extension argument. At each stage $s$ we will have a finite prefix $A^{(s)}_n$ (resp.\ $B^{(s)}_n$) of $A_n$ (resp.\ $B_n$). For the sake of readability, we write $\mathbf{A}^{(s)}=\bigoplus_{n\in\mathbb{N}} A^{(s)}_n$ and $\mathbf{B}^{(s)}=\bigoplus_{n\in\mathbb{N}} B^{(s)}_n$. Moreover, we say that $\mathbf{A}^{(s+1)}$ extends $\mathbf{A}^{(s)}$ if, for each $n$, $A^{(s)}_n \prefix A^{(s+1)}_n$.
	
	Let $X,Y\subset \mathbb{N}$ be Turing-incomparable sets with $X\oplus Y \turingequiv \emptyset'$. Intuitively, we will build $A_n$ and $B_n$ as follows: the first bit of $A_n$ is $X(0)$ while the first bits of $B_n$ code $n$. Then the following bits of $A_n$ will code the position in $B_n$ where to find $Y(0)$ (which will possibly be some large value $k$). Then, the bits after $B_n(k)$ will code the position in $A_n$ where to find $X(1)$, and so on. In other words, we ``stretch" $X$ and $Y$ into $A_n$ and $B_n$, so that $\emptyset'\turingreducible A_n\oplus B_n$ but we cannot $B_n$-computably map $x$ to the prefix of $B_n$ needed to compute $\emptyset'(x)$. 
	
	We start by defining $A^{(0)}_n:= \str{X(0)}$ and $B^{(0)}_n:=1^n 0$ for every $n$. We also define $c_0:=1$ (this is just an index to keep track of what is the next element of $X$ and $Y$ we need to code). 
	
	At stage $s=\coding{e,h}$ if $\{e\}$ is not total then there is nothing to do: we simply let $\mathbf{A}^{(s+1)}:=\mathbf{A}^{(s)}$, $\mathbf{B}^{(s+1)}:=\mathbf{B}^{(s)}$, $c_{s+1}:=c_s$ and go to the next stage. Assume $\{e\}$ is total and let $e_i:=\{e\}(i)$. We look for some computable extension $\mathbf{A}^{(s+1)}$ of $\mathbf{A}^{(s)}$ and $\mathbf{B}^{(s+1)}$ of $\mathbf{B}^{(s)}$ s.t.\ for every $n$, $A^{(s+1)}_n$ and $B^{(s+1)}_n$ are of the form respectively
	\begin{gather*}
		A^{(s)}_n \concat 1^{\length{B^{(s)}_n}+ \length{\sigma_n}}0 \concat \tau_n \concat \str{ X(c_s) }\\
		B^{(s)}_n\concat \sigma_n \concat \str{Y(c_s)}\concat 1^{\length{A^{(s+1)}_n}-1}0
	\end{gather*}
	for some $\sigma_n, \tau_n \in \cantor$, and such that one of the following conditions hold:
	\begin{itemize}
		\item there are $i,n,x\in\mathbb{N}$ s.t.\ $\{e_i\}^{B^{(s+1)}_n}(x)\downarrow=\gamma_i$, $\length{\gamma_i}\le \length{B^{(s+1)}_n}$ and $\gamma_i \not \prefix B^{(s+1)}_n$;
		\item there are $i,n,x\in\mathbb{N}$ s.t.\ for every extension $\bar B_n$ of $B^{(s+1)}_n$, $\{e_i\}^{\bar B_n}(x)\uparrow$;
		\item there is $x\in\mathbb{N}$ s.t., letting $\gamma_i(x):=\{e_i\}^{B^{(s+1)}_i}(x)$ for every $i$, and defining $\boldsymbol{\gamma}(x):= \bigoplus_{i\in\mathbb{N}} \gamma_i(x)$, either $\{ h \}^{\boldsymbol{\gamma}(x)\oplus \mathbf{A}^{(s+1)}}(x)\downarrow\neq \emptyset'(x)$ or, for every extension $\bar{\mathbf{A}}$ of $\mathbf{A}^{(s+1)}$, $\{h\}^{\boldsymbol{\gamma}(x)\oplus \bar{\mathbf{A}} }(x)\uparrow$.
	\end{itemize}
	We then define $c_{s+1}:= c_s +1$ and go to the next stage.
	
	Observe that a choice of $\mathbf{A}^{(s+1)}$ and $\mathbf{B}^{(s+1)}$ as described is always possible: if this is not the case then, for some $s=\coding{e,h}$, for every computable extensions $\mathbf{A}^{(s+1)}$ and $\mathbf{B}^{(s+1)}$ as above we would have:
	\begin{itemize}
		\item for every $i,x$, $\{e_i\}^{B^{(s+1)}_i}(x)\downarrow=:\gamma_i(x)\prefix B^{(s+1)}_i$
		\item for every $x$, $\{ h \}^{\boldsymbol{\gamma}(x)\oplus \mathbf{A}^{(s+1)}}(x) = \emptyset'(x)$, where $\boldsymbol{\gamma}(x):= \bigoplus_{i\in\mathbb{N}} \gamma_i(x)$.
	\end{itemize}
	However, since both $\mathbf{A}^{(s+1)}$ and $\mathbf{B}^{(s+1)}$ are computable (they were obtained by uniformly computably extending $\mathbf{A}^{(0)}$ and $\mathbf{B}^{(0)}$ finitely many times), this would imply that $\emptyset'$ is computable, which is obviously a contradiction. 
	
	For every $n$, we define $A_n := \lim_{s\to\infty} A^{(s)}_n$ and $B_n := \lim_{s\to\infty} B^{(s)}_n$. Observe that both $\mathbf{A}^{(s)}$ and $\mathbf{B}^{(s)}$ are extended infinitely many times, hence for every $n$, $A_n, B_n \in\Cantor$. It is trivial to see that the obtained sets satisfy the conditions $(1),(3)$ above. To show that condition $(2)$ is satisfied notice that, for every $n$, we can (uniformly) compute the sets $X$ and $Y$ (fixed at the beginning), hence $\emptyset' \turingreducible A_n \oplus B_n$. The conditions $\emptyset'\not\turingreducible A_n$ and $\emptyset'\not\turingreducible B_n$ are satisfied by construction. Indeed, to show that $\emptyset'\not\turingreducible A_n$ pick $e$ s.t., for every $i$, $\{e\}(i)= e_0$ and $\{e_0\}$ is the function that returns the first bit of the oracle. For every $s$ and every function $\varphi$ using $A_n$ as oracle, we can compute an index $h$ for a function s.t., if $\eta_i(x):=\{e_i\}^{B^{(s+1)}_i}(0)$ for every $i$ and  $\boldsymbol{\eta}:= \bigoplus_{i\in\mathbb{N}} \eta_i$, $\{h\}^{\boldsymbol{\eta} \oplus \mathbf{A}} = \varphi^{A_n}$. In particular, at stage $\coding{e,h}$ we diagonalized against any such $\varphi$. With a similar reasoning, notice that if $\emptyset'\turingreducible B_n$ via a computable function $\psi$, then we can choose $e$ and $h$ appropriately so that $\{e\}(n)$ is the index of a function mapping $x$ to the prefix of $B_n$ necessary to $\psi$ to compute $\emptyset'(x)$, and $h$ simulates $\psi$ on the corresponding column of its oracle. 
\end{proof}

\section{The first-order part of a problem}
\label{sec:fop}

Recall that, if $w=\str{e}\concat p$, $\Phi_w(x)$ simulates the $e$-th computable Turing functional with oracle $p$ and input $x$. We conventionally represent the natural numbers via the map $p\mapsto p(0)$. Let us now define formally the first-order part of a problem: 

\begin{definition}
	\thlabel{def:first_order_part}
	We say that a computational problem $f\pmfunction{X}{Y}$ is \textdef{first-order}, and write $f\in\mathcal{F}$, if there is a computable injection $\function*{Y}{\mathbb{N}}$ with computable inverse. For every problem $f\pmfunction{X}{Y}$, the \textdef{first-order part of $f$} is the multi-valued function $\firstOrderPart{f}\pmfunction{\Baire\times X}{\mathbb{N}}$ defined as follows:
	\begin{itemize}
		\item instances are pairs $(w,x)$ s.t.\ $x\in \dom(f)$ and for every $y\in f(x)$ and every name $p_y$ for $y$, $\Phi_w(p_y)(0)\downarrow$;
		\item a solution for $(w,x)$ is any $n$ s.t.\ there is a name $p_y$ for a solution $y\in f(x)$ with $\Phi_w(p_y)(0)\downarrow=n$.
	\end{itemize}
\end{definition}

Observe that any first-order problem is strongly Weihrauch equivalent to some problem with codomain $\mathbb{N}$.
Intuitively, the first-order part of $f$ behaves ``just like $f$, but stops at the first bit''. The motivation for this notion comes from the following fact:

\begin{proposition}[{\cite{DSYFirstOrder}}]
    \thlabel{thm:fop_max}
	For every problem $f$, $\firstOrderPart{f} \weiequiv \max_{\weireducible}\{ g\in \mathcal{F} \st g \weireducible f\}$.
\end{proposition}

We briefly give the idea of the proof, as it can guide the intuition when working with the first-order part of a problem. Observe that $\firstOrderPart{f}$ is a first-order problem (by definition) and $\firstOrderPart{f}\weireducible f$, hence $\firstOrderPart{f}$ belongs to $\{g\in \mathcal{F} \st g \weireducible f \}$. Let $g$ be a first-order problem that reduces to $f$ via the functionals $\Phi$, $\Psi$. Assume without loss of generality that $g\pmfunction{\Baire}{\mathbb{N}}$ and $f\pmfunction{\Baire}{\Baire}$ (this makes the presentation easier, as we do not have to keep track of the representation maps). By definition of Weihrauch reducibility, for every input $p\in\dom(g)$, $\Phi(p)\in\dom(f)$ and, for every solution $q\in f\Phi (p)$, $\Psi(p,q)(0)\in g(p)$. We can computably find a string $r\in\Baire$ s.t.\ for every $t\in\Baire$,
\[\Phi_r(t) = \Psi(p,t). \]
It is straightforward to check that $(r,\Phi(p))$ is a valid input for $\firstOrderPart{f}$, and that $\firstOrderPart{f}(r,\Phi(p))= \Psi(p,q)(0)$, for some solution $q\in f\Phi (p)$, i.e.\ it uniformly computes $g$.

Equivalently\footnotemark{}, we can define the first order part of $f$ as the partial multi-valued function s.t.\
\begin{itemize}
	\item instances are triples $(p,e,i)\in \Baire \times\mathbb{N}\times\mathbb{N}$ s.t.\ $\repmap{X}\Phi_e(p)=:x\in \dom(f)$ and for every $p_y\in \repmap{Y}^{-1}(f(x))$, $\Phi_i(p,p_y)(0)\downarrow$;
	\item a solution for $(p,e,i)$ is any $n$ s.t.\ $\Phi_i(p,p_y)(0)\downarrow=n$, for some name $p_y$ of a solution for $f\repmap{Y}\Phi_e(p)$.
\end{itemize}
\footnotetext{This is actually the original definition used by Dzhafarov, Solomon, and Yokoyama. }%
The two definitions yield two strongly Weihrauch equivalent problems. The difference lies in the fact that, in the first case, we need to consider an input $w\in\Baire$ that (intuitively) codes also the original input for the function we are reducing to $f$, while in the second case we need to specify two integer indexes, as the input will be automatically accessible (as part of the definition of Weihrauch reducibility). 

\begin{remark}
	Observe that the first-order part $\firstOrderPart{f}$ is Weihrauch-equivalent to the problem of producing ``sufficiently long'' prefixes of $f$-solutions. More precisely, $\firstOrderPart{f}$ is Weihrauch-equivalent to the problem that takes in input a pair $(w,x)\in\dom(\firstOrderPart{f})$ and produces a prefix $\sigma$ for some solution $y\in f(x)$ s.t.\ $\Phi_w(\sigma)(0)\downarrow$. Indeed, a solution for $(w,x)\in\dom(\firstOrderPart{f})$ can be uniformly computed from $(w,x)$ and a sufficiently long prefix $\sigma$ for a solution of $f(x)$. The converse reduction follows e.g.\ by \thref{thm:fop_max}. In other words, it can be useful to think of $\firstOrderPart{f}$ as the problem of mapping $(w,x)$ to a prefix of $y\in f(x)$ that satisfies the c.e.\ condition named by $w$.
\end{remark}

The first-order part is a degree-theoretic interior operator, hence a common strategy to characterize the first-order part of a problem $f$ is to show that a first-order function $f_0$ reduces to $f$ and that, for every first-order $g$, if $g\weireducible f$ then $g\weireducible f_0$.

\subsection{Previous appearances in the literature}
\label{sec:literature}

As already mentioned, while Dzhafarov, Solomon, and Yokoyama were the first to consider the operator $\firstOrderPart{(\cdot)}$ on multi-valued functions, there are several results in the literature on Weihrauch degrees that implicitly characterize (or provide a bound for) the first-order part of a problem. Observe e.g.\ that the non-reduction $\Choice{\mathbb{R}}\not\weireducible \CCantor$ \cite{BdBPLow12} can be proved knowing that $\CNatural \weireducible \Choice{\mathbb{R}}$ but $\CNatural \not\weireducible \CCantor$, i.e.\ $\firstOrderPart{\Choice{\mathbb{R}}}\not\weireducible \firstOrderPart{\CCantor}$. In fact, we have $\firstOrderPart{\WWKL}\weiequiv \firstOrderPart{\CCantor}\weiequiv \Choice{2}^*$ \cite{DSYFirstOrder} while $\firstOrderPart{\Choice{\mathbb{R}}}\weiequiv\CNatural$ (\thref{thm:fop_CR}). Besides, \cite[Thm.\ 8.2]{BGP17} lists a number of computational problems that are Weihrauch equivalent to $\CNatural$, immediately characterizing their first-order part.

The first-order part of $\mflim$ was characterized in \cite[Prop.\ 13.10]{BolWei11}, where the authors show that $\firstOrderPart{\mflim}\weiequiv \CNatural$. We generalize this result in \thref{thm:fop_lim}.

In \cite[Prop.\ 3.4]{BRramsey17}, the authors characterize the first-order part of the problems $\RT{1}{k}$ and $\SRT{1}{k}$ for $k\ge 2$ or $k=\mathbb{N}$. More results on the first-order part of principles related to Ramsey's theorem will be obtained in Section~\ref{sec:applications}.

Let $(X,\repmap{X})$ and $(Y, \repmap{Y})$ be two represented spaces. A multi-valued function $f$ is called \textdef{densely realized} if for every $p\in \dom(f\circ \repmap{X})$, $\repmap{Y}^{-1}\circ f \circ \repmap{X}$ is dense in $\dom(\repmap{Y})$. By \cite[Prop.\ 4.13]{BGP17}, if $f\pmfunction{X}{Y}$ is densely realized and $\repmap{Y}$ is total then $\firstOrderPart{f}\weireducible\id$. In particular, the problems $\MLR$, $\mathsf{NON}$, and $\mathsf{NHA}$, defined as ``given $p\in \Baire$, produce some $q\in\Baire$ s.t.\ $q$ is Martin-L\"of random (resp.\ non-computable, non-hyperarithmetic) relatively to $p$'' are all densely realized, and therefore their first-order part is uniformly computable. The same applies to the problem $\COH$ (see e.g.\ \cite[Def.\ 12.1]{BHK2017} and \cite[Prop.\ 8.16]{BGP17}).

The use of the first-order part is particularly helpful when studying principles at the higher levels of the Weihrauch lattice. In \cite{KMP20}, the authors show that $\CBaire\strictlyweireducible \TCBaire$ \cite[Prop.\ 8.2.1]{KMP20} using the fact that $\chiPi \not\weireducible \LPO\compproduct \CBaire \weiequiv \CBaire$, while $\chiPi \weireducible \LPO\compproduct \TCBaire$. The first-order part of $\CBaire$ was explicitly characterized in \cite[Prop.\ 2.4]{GPVDescSeq}, showing that $\firstOrderPart{\CBaire} \weiequiv \codedChoice{\boldfaceSigma^1_1}{}{\mathbb{N}}$. The same argument shows that $\firstOrderPart{\UCBaire} \weiequiv \codedUChoice{\boldfaceSigma^1_1}{}{\mathbb{N}}$ and $\firstOrderPart{\TCBaire} \weiequiv \totalization{(\codedChoice{\boldfaceSigma^1_1}{}{\mathbb{N}})}$.

In \cite[Lem.\ 4.7]{KMP20} the authors (implicitly\footnote{Since $\SigmaWKL$ is parallelizable, $\PiBound\weireducible \SigmaWKL$ iff $\parallelization{\PiBound}\weireducible\SigmaWKL$.}) use the fact that $\PiBound\not\weireducible\SigmaWKL$ to separate $\SigmaWKL$ from $\parallelization{\codedChoice{\boldfaceSigma^1_1}{}{\mathbb{N}}}$,  where $\SigmaWKL$ is the generalization of $\WKL$ to binary trees presented via a $\boldfaceSigma^1_1$ code (see \cite[Def.\ 4.2]{KMP20}). In fact, \cite[Cor.\ 3.19]{KiharaADauriacChoice} strengthen this observation, showing that $\parallelization{\PiBound} \weiincomparable \SigmaWKL \weiequiv \parallelization{\codedChoice{\boldfaceSigma^1_1}{}{2}}$. The principle $\PiBound$ was used in the context of the Weihrauch degrees of some infinite-dimensional generalizations of Ramsey theorems \cite[Thm.\ 4.11]{MVRamsey}, and played an important role in \cite{GPVDescSeq}: in particular, the authors explicitly prove that $\PiBound \weiequiv \firstOrderPart{\DS}$, where $\DS$ is the problem of producing an infinite descending sequence through an ill-founded linear order.

The characterization of the first-order part of problems proved very helpful also in the analysis of the uniform computational strength of Cantor-Bendixson theorem \cite{CMVCantorBendixson}.

\section{Some algebraic rules}
\label{sec:alg_fop}

We now prove some results describing the interaction between the first-order part operator and some well-known operators on multi-valued functions.

\begin{proposition} 
	\thlabel{thm:algebraic_rules_fop}
	For every multi-valued functions $f$ and $g$,
	\begin{enumerate}
		\item $\firstOrderPart{(f \sqcup g)} \weiequiv \firstOrderPart{f} \sqcup \firstOrderPart{g}$;
		\item $\firstOrderPart{(f \sqcap g)} \weiequiv \firstOrderPart{f} \sqcap \firstOrderPart{g}$;
		\item $\firstOrderPart{f} \times \firstOrderPart{g} \weireducible \firstOrderPart{(f\times g)}$;
		\item $\firstOrderPart{f} \compproduct \firstOrderPart{g} \weireducible \firstOrderPart{(f\compproduct g)} \weireducible \firstOrderPart{f} \compproduct g$.
	\end{enumerate}
	None of the above reductions can be reversed. In fact, 	
	\begin{enumerate}
		\setcounter{enumi}{4}
		\item \label{itm:fop_alg_counterexample1} there are $f,g$ s.t.\ $\firstOrderPart{f} \compproduct \firstOrderPart{g}\not\weireducible \firstOrderPart{(f \times g)}$; 
		\item \label{itm:fop_alg_counterexample2} there are $f,g$ s.t.\ $\firstOrderPart{(f \times g)} \not \weireducible \firstOrderPart{f} \compproduct \firstOrderPart{g}$.
	\end{enumerate}
	
\end{proposition}
\begin{proof}
	Without loss of generality, we assume that $f$ and $g$ are multi-valued functions on the Baire space.
	\begin{enumerate}
		\item This is straightforward from the definitions. Recall that an input for $f \sqcup g$ is of the type $(i,x)$ where $x\in \dom(f)$ if $i=0$ and $x\in \dom(g)$ if $i=1$. 
		
		To prove the left-to-right reduction it suffices to map $(w,(i,x))$ to $(i,(w,x))$. To prove the right-to-left reduction it suffices to consider the inverse map $(i,(w,x)) \mapsto (w,(i,x))$.
		\item To prove the left-to-right reduction notice that, by the monotonicity of $\firstOrderPart{(\cdot)}$, $\firstOrderPart{(f \sqcap g)} \weireducible \firstOrderPart{f}$ and $\firstOrderPart{(f \sqcap g)} \weireducible \firstOrderPart{g}$. Since $\sqcap$ is the meet in the Weihrauch lattice, we obtain $\firstOrderPart{(f \sqcap g)} \weireducible \firstOrderPart{f} \sqcap \firstOrderPart{g}$. To prove the right-to-left reduction, recall that, by definition,  
		\[ 		(\firstOrderPart{f} \sqcap \firstOrderPart{g})((w,x),(v,z)) =\firstOrderPart{f}(w,x) \sqcup  \firstOrderPart{g}(v,z). \]
		We can uniformly compute $r\in\Baire$ s.t.\ 
		\[ \Phi_r(\coding{i,t}) = \begin{cases}
			\coding{0, \Phi_w(t)} & \text{if } i=0, \\
			\coding{1, \Phi_v(t)} & \text{if } i=1.
		\end{cases} \]
		It follows that every solution for $\firstOrderPart{(f \sqcap g)}(r,(x,z))$ is a solution for $(\firstOrderPart{f} \sqcap \firstOrderPart{g})((w,x),(v,z))$.
		
		\item To show that the reduction holds it suffices to consider the map $((w,x),(v,z)) \mapsto (r,(x,z))$ where $\Phi_r(\coding{y,t}) = \coding{\Phi_w(y), \Phi_v(t)}$. The fact that the converse reduction does not hold (in general) follows from point (\ref{itm:fop_alg_counterexample2}).

		\item For the first reduction, notice that, by the monotonicity of $\firstOrderPart{(\cdot)}$, $\firstOrderPart{f}\compproduct\firstOrderPart{g} \weireducible f\compproduct g$. Since $\firstOrderPart{f}\compproduct\firstOrderPart{g}$ is a first-order problem we have $\firstOrderPart{f}\compproduct\firstOrderPart{g} \weireducible \firstOrderPart{(f\compproduct g)}$. 
		
		To prove the second reduction, observe that a solution for $\firstOrderPart{(f\compproduct g)}(w, (v,z))$ is $\Phi_{w}(t,y)(0)$ for some $t\in g(z)$ and $y\in f(\Phi_v(t))$. To show that $\firstOrderPart{(f\compproduct g)} \weireducible\firstOrderPart{f} \compproduct g$ it is enough to consider the map $(w,(v,z))\mapsto (r,z)$, where $r\in\Baire$ is s.t.\ 
		\begin{gather*}
			\Phi_r(t) = \pairing{p,\Phi_v(t)};\\
			\Phi_p(y) = \Phi_w(t, y).
		\end{gather*}

		A counterexample for the reduction $\firstOrderPart{(f\compproduct g)}\weireducible \firstOrderPart{f} \compproduct\firstOrderPart{g}$ is given by $f=g=\mflim$. Indeed, we will show $\firstOrderPart{(\mflim^{(n)})} \weiequiv \CNatural^{(n)}$ (\thref{thm:fop_lim}), hence a reduction would yield 
		\[ \CNatural' \weiequiv \firstOrderPart{(\mflim \compproduct \mflim)} \weireducible \firstOrderPart{\mflim} \compproduct \firstOrderPart{\mflim} \weiequiv \CNatural \compproduct \CNatural \weiequiv \CNatural, \]
		which is a contradiction. 
		
		To show that the reduction $\firstOrderPart{f} \compproduct g \weireducible \firstOrderPart{(f\compproduct g)}$ can fail it is enough to notice that the right-hand side always has computable solutions. An explicit counterexample can therefore be obtained by choosing again $f=g=\mflim$.
		
		\item The non-reduction is witnessed by $f=g=\LPO$ \cite[Proof of Prop.\ 4.8(6)]{BP16}.
		
		\item To show that the reduction $\firstOrderPart{(f\times g)}\weireducible \firstOrderPart{f} \compproduct\firstOrderPart{g}$ can fail, let $\sequence{A_n}{n\in\mathbb{N}}$ and $\sequence{B_n}{n\in\mathbb{N}}$ be as in \thref{thm:countable_sparse_splitting} and let $A:=A_0$, $B:=B_0$. In particular we have 
		\begin{itemize}
			\item $\emptyset'\not\turingreducible A$, $\emptyset'\not\turingreducible B$, $\emptyset'\turingreducible A \oplus B$
			\item for every $e$, if $\{e\}^{A\oplus B} = \emptyset'$ then the map sending $x$ to the prefix of $B$ used in the computation $\{e\}^{A\oplus B}(x)$ is not $B$-computable.
		\end{itemize}
		
		Let $f$ (resp.\ $g$) be the constant map returning the characteristic function of $A$ (resp.\ $B$). Clearly, $\charfun{\emptyset'}\weireducible \firstOrderPart{(f\times g)}$. On the other hand $\charfun{\emptyset'}\not\weireducible\firstOrderPart{f}\compproduct \firstOrderPart{g}$. Notice indeed that $\firstOrderPart{f}$ is equivalent to the map that, given $w$ produces a sufficiently long prefix $A[n]$ of $A$ s.t.\ $\Phi_w(A[n])(0)\downarrow$. Analogously, we can think of $\firstOrderPart{g}$ as producing a sufficiently long prefix of $B$. 
		
		Assume there is a reduction $\charfun{\emptyset'}\weireducible \firstOrderPart{f}\compproduct\firstOrderPart{g}$ as witnessed by $\Phi$,$\Psi$. For every $x$, $\Phi(x)=\coding{\Phi_1(x),\Phi_2(x)}$ is s.t.\ $\Phi_2(x)$ is an input for $\firstOrderPart{g}$ and $\Phi_1(x)$ is the index of a functional that, given $A$ and the prefix of $B$ produced by $\firstOrderPart{g}$, computes $\charfun{\emptyset'}$. Since all computations are done uniformly, this corresponds to the existence of two indexes $e$ and $i$ for computable functions s.t.\ $\{e\}^{B}(n)= B[m]$ and $\{i\}^{A\oplus B[m]}(n)=\charfun{\emptyset'}(n)$. Hence, the existence of a Weihrauch reduction contradicts the properties of $A$ and $B$.\qedhere
	\end{enumerate}
\end{proof}

The counterexample used to show that $\firstOrderPart{(f \times g)} \not \weireducible \firstOrderPart{f} \compproduct \firstOrderPart{g}$ suggests also how to show that, in general, $\firstOrderPart{(f \times f)} \not \weireducible \firstOrderPart{f} \compproduct \firstOrderPart{f}$. Indeed, letting $A,B$ as above, we can define $f$ so that $f(0):=\chi_A$ and $f(1):=\chi_B$. It is not hard to adapt the proof of \thref{thm:algebraic_rules_fop}(\ref{itm:fop_alg_counterexample2}) to show that such $f$ satisfies the claim.

We can also explore the connections between the first-order part and the jump in the Weihrauch lattice. To this end, we introduce the following notion:

\begin{definition}
	We say that a first-order problem $f$ is a \textdef{first-order cylinder} if, for every first-order $g$,
	\[ g\weireducible f \Rightarrow g\strongweireducible f. \]
\end{definition}
Notice that no first-order problem can be a (classical) cylinder: indeed, we already noticed that every first-order problem is strongly Weihrauch equivalent to a problem with codomain $\mathbb{N}$, and therefore only has computable outputs (hence, in particular, it cannot strongly compute $\id$).

\begin{proposition}
	\thlabel{thm:fo_cylinder}
	If $f$ is a cylinder then $\firstOrderPart{f}$ is a first-order cylinder.
\end{proposition}
\begin{proof}
	Let $g$ be a first-order problem and assume $g\weireducible \firstOrderPart{f}$ via $\Phi,\Psi$.  Without loss of generality we can assume that $g$ has codomain $\mathbb{N}$. Let also $\Phi_S, \Psi_S$ be two computable functionals witnessing the reduction $\id \times f \strongweireducible f$. Let $v\in\Baire$ be s.t.\ 
	
	\[ \Phi_v(p) = \Psi(p_0, \Phi_{p_1}(p_2)), \]
	where $\coding{\coding{p_0,p_1},p_2}=\Psi_S(p)$.
	
	Let $p_t$ be a name for an input $t$ of $g$, and let $\Phi(p_t)$ be a name for an input $(w_t,x_t)$ of $\firstOrderPart{f}$. We claim that the maps $p_t\mapsto (v, \Phi_S(\coding{p_t,w_t},x_t))$ and $\id$ witness $g \strongweireducible \firstOrderPart{f}$. Indeed, $\Phi_S(\coding{p_t,w_t},x_t)$ is a name for some $y\in\dom(f)$ s.t.\ every name $p_z$ for some $z\in f(y)$ uniformly computes (via $\Psi_S$) a name for a pair $(\coding{p_t,w_t} ,p_r)$, where $p_r$ is a name for some $r\in f(x_t)$. In particular, we obtain $\Phi_v(p_z)(0)= \Psi(p_t, \Phi_{w_t}(p_r))(0)$, which is a valid solution for $g(t)$ (as $g$ has codomain $\mathbb{N}$).
\end{proof}

As a trivial consequence, if $f$ is a cylinder, $g$ is a first-order cylinder and $g\weiequiv \firstOrderPart{f}$, then $g\strongweiequiv \firstOrderPart{f}$ and hence $g' \strongweiequiv (\firstOrderPart{f})'$ (as the jump lifts to the strong Weihrauch degrees).

\begin{proposition}
	\thlabel{thm:fop_jumps}
	For every multi-valued function $f$, $\firstOrderPart{(f')} \strongweireducible (\firstOrderPart{f})'$.	Moreover, if $f$ is a cylinder then $\firstOrderPart{(f')}\strongweiequiv(\firstOrderPart{f})'$.
\end{proposition}
\begin{proof}
	The first statement is a trivial consequence of the definitions. Indeed, given an input $(w,\seq{x})$ for $\firstOrderPart{(f')}$, where $\seq{x}=\sequence{x_n}{n\in\mathbb{N}}$ is a sequence that converges to $x\in\dom(f)$, it is enough to consider the input $(\seq{w},\seq{x})$ for $(\firstOrderPart{f})'$, where $\seq{w}=\sequence{w_n}{n\in\mathbb{N}}$ is the constant sequence $w_n:=w$. Clearly 
	\[ (\firstOrderPart{f})'(\seq{w},\seq{x}) = \firstOrderPart{(f')}(w,\seq{x}).\]
	
	Assume now that $f$ is a cylinder. In particular, $f'$ is a cylinder and $f\compproduct \mflim \strongweiequiv f'$. This implies that
	\[ (\firstOrderPart{f})'\strongweireducible \firstOrderPart{f}\compproduct \mflim \strongweireducible f\compproduct \mflim \strongweiequiv f' \]
	Since $(\firstOrderPart{f})'$ is first-order, the maximality of the first-order part implies that $(\firstOrderPart{f})'\weireducible \firstOrderPart{(f')}$. By \thref{thm:fo_cylinder}, $\firstOrderPart{(f')}$ is a first-order cylinder, hence $(\firstOrderPart{f})'\strongweireducible \firstOrderPart{(f')}$.
\end{proof}

The reduction $(\firstOrderPart{f})'\weireducible \firstOrderPart{(f')}$ can fail if $f$ is not a cylinder. To see this, one can simply notice that $(\firstOrderPart{f})'$ takes in input a sequence $(\sequence{w_n}{n\in\mathbb{N}},\sequence{x_n}{n\in\mathbb{N}})$ that converges to $(w,x)$ and produces $\firstOrderPart{f}(w,x)$, whereas the input for $\firstOrderPart{(f')}$ is of the type $(v, \sequence{z_n}{n\in\mathbb{N}})$, where $\sequence{z_n}{n\in\mathbb{N}}$ converges to an input for $f$. The forward functional of the reduction $(\firstOrderPart{f})' \weireducible \firstOrderPart{(f')}$ would have to commit to some $v(0)$ (i.e.\ to some index for a Turing machine) in a finite number of steps. One can diagonalize against the reduction by changing the sequence $\sequence{w_n}{n\in\mathbb{N}}$ after that stage, so that $\lim_{n\to\infty} w_n(0)=w(0)$ is the index of a different computable function. This procedure does not work if $f$ is a cylinder, as in that case $f' \weiequiv f\compproduct\mflim$, hence one can use $\mflim$ to get the correct $(w,x)$.

\section{The unbounded finite parallelization}
\label{sec:u*}

Let us introduce the following ``unbounded-${}^*$'' operator. Intuitively it generalizes the finite parallelization ${}^*$, by relaxing the requirement that the number of instances of the problem is specified as part of the input.
\begin{definition}
	For every $f\pmfunction{X}{Y}$ define $\ustar{f}\pmfunction{\Baire\times \infStrings{X}}{\finStrings{(\Baire)}}$ as follows:
	\begin{itemize}
		\item instances are pairs $(w,\sequence{x_n}{n\in\mathbb{N}})$ s.t.\ $\sequence{x_n}{n\in\mathbb{N}}\in \dom(\parallelization{f})$ and for each sequence $\sequence{q_n}{n\in\mathbb{N}}$, with $\repmap{Y}(q_n) \in f(x_n)$, there is $k\in \mathbb{N}$ s.t $\Phi_w(\pairing{q_i}_{i<k})(0)\downarrow$ in $k$ steps;
		
		\item a solution for $(w,\sequence{x_n}{n\in\mathbb{N}})$ is every finite sequence $\sequence{q_n}{n<k}$ s.t.\ for every $n$, $\repmap{Y}(q_n) \in f(x_n)$ and $\Phi_w( \pairing{q_i}_{i<k})(0)\downarrow$ in $k$ steps.
	\end{itemize}		
\end{definition}

Observe that, given how the join of finitely many strings is defined, $k$ is uniformly computable from $\pairing{q_i}_{i<k}$. The definition needs to be adapted in case a different definition for the join is used. Similarly, the bound on the number of steps guarantees that we can see sufficiently long prefixes of each $q_i$. A larger bound can be chosen without affecting the development of the theory, as long as it can be uniformly computed from $k$. We underline that choosing a computable bound is important for the proof of the following \thref{thm:game-def-ustar}.

One can think of $\ustar{f}$ as the problem that, given an input for $\parallelization{f}$ (i.e.\ a sequence of inputs for $f$), produces ``sufficiently many'' names for solutions, where the ``sufficiently many'' is determined precisely by the convergence of a given Turing functional. In other words, we require to have a c.e.\ procedure (relatively to the input) that tells us how many columns of the output we need to take a look at.

\begin{proposition}
	\thlabel{thm:u*<=parallelization}
	For every problem $f\pmfunction{X}{Y}$, $f\strongweireducible f^* \strongweireducible \ustar{f}\weireducible \parallelization{f}$.
\end{proposition}
\begin{proof}
	The first reductions is obvious and the second one is straightforward by definition. To prove the last reduction, fix $(w,\sequence{x_n}{n\in\mathbb{N}}) \in \dom(\ustar{f})$ and, for each $n$, let $p_n$ be a $\repmap{X}$-name for $x_n$. Let also $\sequence{y_n}{n\in\mathbb{N}} \in \parallelization{f}(\sequence{x_n}{n\in\mathbb{N}})$ and let $q_n$ be a $\repmap{Y}$-name for $y_n$. 
	
	The backward functional of the reduction is given by the map $\Psi$ such that, when executed with input $(w,\sequence{p_n}{n\in\mathbb{N}},\sequence{q_n}{n\in\mathbb{N}})$, it returns $\pairing{q_n}_{n<k}$ where $k$ is minimum s.t.\ $\Phi_w(\pairing{q_n}_{n<k})(0)$ converges in $k$ steps. The existence of such $k$ is guaranteed by the definition of $\ustar{f}$.
\end{proof}

Similarly to the case of the first-order part, the definition of $\ustar{f}$ can be equivalently given by replacing $w\in\Baire$ with a couple of indexes $e,i\in\mathbb{N}$ for Turing functionals (see the comments after \thref{def:first_order_part}).

We can characterize the Weihrauch degree of $\ustar{f}$ using a reduction game. 

\begin{definition}
	Let $f\pmfunction{X}{Y}$, $g\pmfunction{A}{B}$ be two partial multi-valued functions. We define the reduction game $U(f\to g)$ as the following two-player game: Player 1 starts by playing a $\repmap{A}$-name $p_a$ for some $a\in \dom(g)$ and Player 2 answers with an index $e\in\mathbb{N}$ s.t.\ $\Phi_e(p_a)=\pairing{p_i}_{i\in\mathbb{N}}$ where, for every $i$, $p_i$ is a $\repmap{X}$-name for $x_i\in\dom(f)$. If Player 2 does not have a valid move then Player 1 wins, otherwise the game continues.
	
	On the $(n+1)$-th move Player $1$ produces a $\repmap{Y}$-name $q_n$ for some $y_{n}\in f(x_{n})$, and Player 2 either declares victory and plays a $\pairing{p_a,q_0,\hdots, q_{n}}$-computable $\repmap{B}$-name for $b\in g(a)$, or passes his turn. 
	
	Player 1 wins if Player 2 never declares victory.
\end{definition}

\begin{proposition}
	\thlabel{thm:game-def-ustar}
	For every $f\pmfunction{X}{Y}$, $g\pmfunction{A}{B}$, 
	\[ g\weireducible \ustar{f} \iff \text{Player }2\text{ has a computable winning strategy for }U(f\to g). \]
\end{proposition}
\begin{proof}
	Assume $g\weireducible \ustar{f}$ via $\Phi,\Psi$. For every $\repmap{A}$-name $p_a$ for a $g$-instance $a$, $\Phi(p_a)$ produces the index $w\in\Baire$ for a continuous functional, and a name $\pairing{p_i}_{i\in\mathbb{N}}$ for an instance of $\parallelization{f}$. In particular, letting $\pi_2:= \pairing{p,q}\mapsto q$, the index of $\pi_2\circ \Phi$ is a valid first move for Player $2$.
	On the $k$-th move, Player $2$ checks whether 
	$\Phi_w( \pairing{q_i}_{i<k})(0)\downarrow$ in $k$ steps. If yes then Player 2 declares victory and returns 
	$\Psi(p_a,\pairing{q_i}_{i<k})$, otherwise he passes. The fact that $\Phi,\Psi$ witness a Weihrauch reduction implies that there is $k$ s.t.\ Player 2 declares victory at stage $k$, and therefore he has a computable winning strategy for $U(f\to g)$.
	
	Assume now that Player 2 has a computable winning strategy $\sigma$ for $U(f\to g)$. We claim that the reduction $g\weireducible \ustar{f}$ is witnessed by the maps $\Gamma$ (forward) and $\Delta$ (backward), defined as follows: the functional $\Gamma$ maps a name $p_a$ for $a\in\dom(g)$ to $(w,\Phi_e(p_a))$, where $e\in\mathbb{N}$ is the first move played by Player 2 (i.e.\ $\Phi_e(p_a)$ is a name for a valid input of $\parallelization{f}$) and $w\in\Baire$ is an index for the continuous functional that, upon input $\pairing{q_i}_{i<k}$, checks whether Player 2 declares victory on his $k$-th move (when playing following $\sigma$). The backward functional $\Delta$ is to the map that sends $p_a, q_0,\hdots, q_{k-1}$ to a name for $b\in g(a)$. 
	
	Since $\sigma$ is computable then so are $\Gamma$ and $\Delta$. Notice that $w$ need not be computable, as the strategy $\sigma$ also has access to the original input $p_a$. Moreover, since $\sigma$ is a winning strategy for Player 2, there is a finite stage $k$ in which Player 2 declares victory, i.e.\ $(w,\Phi_e(p_a))$ is a valid input for $\ustar{f}$ and $\Delta$ computes a (name for a) valid solution of $g(a)$.
\end{proof}

We now show that $\ustar{(\cdot)}$ respects both Weihrauch and strong Weihrauch reductions.

\begin{theorem}
	For every $f\pmfunction{X}{Y}$, $g\pmfunction{A}{B}$, 
	\begin{enumerate}
		\item $f\weireducible g \Rightarrow \ustar{f}\weireducible \ustar{g}$;
		\item $\ustar{(\ustar{f})} \weiequiv \ustar{f}$.
	\end{enumerate}
	In particular, $\ustar{(\cdot)}$ is a closure operator. The above properties hold also if we replace $\weireducible$ with $\strongweireducible$.
\end{theorem}
\begin{proof}
	\begin{enumerate}
		\item Assume $f\weireducible g$ via $\Phi, \Psi$. Let $(w,\sequence{x_n}{n\in\mathbb{N}})$ be a valid input for $\ustar{f}$ and, for each $n$, let $p_n$ be a name for $x_n$. We can uniformly compute the pair $(r, \sequence{\Phi(p_n)}{n\in\mathbb{N}})$, where $r\in\Baire$ is s.t. for every $k\in\mathbb{N}$
		\[ \Phi_r( \pairing{q_n}_{n<k}) = \Phi_w (\pairing{\Psi(p_n, q_n)}_{n<k}  ).  \]
		Notice that, since $\Phi$ and $\Psi$ witness the reduction $f\weireducible g$, the pair $(r, \sequence{\Phi(p_n)}{n\in\mathbb{N}})$ is a valid input for $\ustar{g}$. Moreover, if $\sequence{q_n}{n<k}$ is a solution of $\ustar{g}(r, \sequence{\Phi(p_n)}{n\in\mathbb{N}})$, then $\sequence{\Psi(p_n,q_n)}{n<k}$ is a valid solution for $\ustar{f}(w,\sequence{x_n}{n\in\mathbb{N}})$. The same argument shows that $f\strongweireducible g$ implies $\ustar{f}\strongweireducible\ustar{g}$, the only difference is that, instead of $r$, we consider $\bar r$ s.t.\ $\Phi_{\bar r}( \pairing{q_n}_{n<k}) = \Phi_w (\pairing{\Psi(q_n)}_{n<k})$.
		
		\item The proof is essentially a definition-chasing exercise. Let $(w,\sequence{\mathbf{t}_i}{i\in\mathbb{N}})$ be valid input for $\ustar{(\ustar{f})}$, where, for every $i$, $\mathbf{t}_i= (w_i, \sequence{\repmap{X}(p^i_j)}{j\in\mathbb{N}})$ is a valid input for $\ustar{f}$. A solution for $\ustar{(\ustar{f})}(w,\sequence{\mathbf{t}_i}{i\in\mathbb{N}})$ is a finite sequence $\sequence{\mathbf{z}_i}{i<h}$, where $\mathbf{z}_i = \sequence{q^i_j}{j<k_i}$ is a valid output for $\ustar{f}(w_i, \sequence{\repmap{X}(p^i_j)}{j\in\mathbb{N}})$.
		
		We can uniformly compute an index $r\in\Baire$ s.t.\ $\Phi_r$ behaves as follows: given in input $\pairing{v_{\pairing{i,j}}}_{i,j<N}$, let $h:=\max \{ n \st \pairing{n,0}<N \}$ and, for each $i<h$, let $k_i:=\max \{ m \st \pairing{i,m}<N \}$. For every $i<h$ and $j<k_i$, define $q^i_j:=v_{\pairing{i,j}}$. The functional $\Phi_r$ simulates $\Phi_w(\pairing{ \pairing{q^i_j}_{j<k_i} }_{i<h})$ for $h$ steps and $\Phi_{w_i}( \pairing{q^i_j}_{j<k_i})$ for $k_i$ steps. If all of the simulations converge then $\Phi_r$ converges (with dummy output), otherwise it diverges.
		
		We then consider the pair $(r, \sequence{\repmap{X}(\bar p_{\pairing{i,j}})}{i,j\in\mathbb{N}})$, where $\bar p_{\pairing{i,j}}:=p^i_j$. It is straightforward to show that $(r, \sequence{\repmap{X}(\bar p_{\pairing{i,j}})}{i,j\in\mathbb{N}})$ is a valid input for $\ustar{f}$. Indeed, for each $i\in\mathbb{N}$, there is a finite sequence $\sequence{q^i_j}{j<k_i}$ s.t.\ for every $j<k_i$, $\repmap{Y}(q^i_j)\in f(\repmap{X}(p^i_j))$ and $\Phi_{w_i}( \pairing{q^i_j}_{j<k_i})$ converges in $k_i$ steps. Similarly, there is $h$ s.t.\ $\Phi_w(\pairing{\pairing{q^i_j}_{j<k_i} }_{i<h})$
		converges in $h$ steps. Moreover, every solution $\sequence{v_{\pairing{i,j}}}{\pairing{i,j}<N}$ for $\ustar{f}(r, \sequence{\repmap{X}(\bar p_{\pairing{i,j}})}{i,j\in\mathbb{N}})$ can be unpacked (as explained above) into the solution $\sequence{\mathbf{z}_i}{i<h}$ for $\ustar{(\ustar{f})}(w,\sequence{\mathbf{t}_i}{i\in\mathbb{N}})$. 
		
		This also proves that $\ustar{(\ustar{f})} \strongweiequiv \ustar{f}$ (as the reduction is, in fact, a strong Weihrauch reduction).\qedhere
	\end{enumerate}
\end{proof}

Notice that, whenever $f$ is a first-order problem, $\ustar{f}$ is strongly Weihrauch equivalent to a problem with range $\baire$ (it has to produce finitely many natural numbers). With a small abuse of notation, it is convenient to identify $\ustar{f}$ with a strong Weihrauch equivalent first-order problem. 

The operator $\ustar{(\cdot)}$ allows us to characterize the degree of the first-order part of a parallelizable problem.

\begin{proposition}
	\thlabel{thm:FOP(parallelization)=FOP(u*)}
	For every multi-valued function $f$, $\firstOrderPart{(\parallelization{f})} \weiequiv \firstOrderPart{(\ustar{f})}.$
\end{proposition}
\begin{proof}
	The right-to-left reduction follows trivially from $\ustar{f}\weireducible \parallelization{f}$ (\thref{thm:u*<=parallelization}) and the fact that $\firstOrderPart{(\cdot)}$ is degree-theoretic. To prove the left-to-right reduction, let $(w,\seq{x})$, with $\seq{x}:=\sequence{x_i}{i\in\mathbb{N}}$, be an input for $\firstOrderPart{\parallelization{f}}$. If $\pairing{q_i}_{i\in\mathbb{N}}$ is a name for a solution of $\parallelization{f}(\seq{x})$, then $\Phi_w(\pairing{q_i}_{i\in\mathbb{N}})(0)\downarrow$. Let $r\in\Baire$ be an index for the functional that, upon input $\pairing{p_i}_{i<k}$ simulates 
	\[\Phi_w(\pairing{p_0,\hdots,p_{k-1},0^\mathbb{N},0^\mathbb{N},\hdots})\]
	for $k$ steps, and enters an endless loop if $\Phi_w$ does not converge on $0$.
	It is straightforward to check that $(r,(r,\seq{x}))$ is a valid input for $\firstOrderPart{(\ustar{f})}$ and that
	\[ \firstOrderPart{(\ustar{f})}(r,(r,\seq{x})) \subset   \firstOrderPart{(\parallelization{f})}(w,\seq{x}).\qedhere \]
\end{proof} 

\begin{theorem}
	\thlabel{thm:FOP(parallelization)=u*(FOP)}
	For every first-order $f$,
	\[ \firstOrderPart{(\parallelization{f})} \weiequiv \ustar{f} \weiequiv \ustar{(\firstOrderPart{f})}. \]	
\end{theorem}
\begin{proof}
	This follows directly from \thref{thm:FOP(parallelization)=FOP(u*)}, as if $f$ is first-order then so is $\ustar{f}$.
\end{proof}

This result characterizes the first-order part of all the problems that are Weihrauch-equivalent to the parallelization of some first-order problem. Several applications of this result will be made explicit in Section~\ref{sec:applications}. 

Notice that not every parallelizable problem is the parallelization of some first-order problem. As a counterexample, consider the problem $\CBaire$. As mentioned, it is known that $\firstOrderPart{\CBaire}\weiequiv \codedChoice{\boldfaceSigma^1_1}{}{\mathbb{N}}$. If $\CBaire\weiequiv \parallelization{f}$ for some first-order $f$ then $f \weireducible \codedChoice{\boldfaceSigma^1_1}{}{\mathbb{N}}$ and hence $\CBaire\weireducible \parallelization{\codedChoice{\boldfaceSigma^1_1}{}{\mathbb{N}}}$, against \cite[Thm.\ 3.30]{KiharaADauriacChoice}.

\medskip

\thref{thm:FOP(parallelization)=u*(FOP)} can be restated saying that $\ustar{(\cdot)}$ and $\firstOrderPart{(\cdot)}$ commute for first-order problems. The same statement, however, holds for a much wider class of problems. Observe for example that the unbounded finite parallelization and the first-order part commute also for every $f\pfunction{\Baire}{\Baire}$.
Indeed, for every $j\in\mathbb{N}$, let $w_j$ be s.t.\ $\Phi_{w_j}(p)$ returns the prefix of length $j$ of $p$. The idea is that, since $f$ is single-valued, given the input $(w,\sequence{x_i}{i\in\mathbb{N}})$ for $\firstOrderPart{(\parallelization{f})}$, for every $i$ we can compute longer and longer prefixes of $f(x_i)$ by considering the input $\sequence{w_j,x_{\pairing{i,j}}}{\pairing{i,j}\in\mathbb{N}}$ for $\parallelization{\firstOrderPart{f}}$. To obtain the reduction $ \firstOrderPart{(\parallelization{f})} \weireducible \ustar{(\firstOrderPart{f})}$, it is therefore enough to map $(w,\sequence{x_i}{i\in\mathbb{N}})$ to the input $(r,\sequence{w_j,x_{\pairing{i,j}}}{\pairing{i,j}\in\mathbb{N}})$ for $\ustar{(\firstOrderPart{f})}$, where $r$ is s.t.\ $\Phi_r(\pairing{y_{i,j}}_{\pairing{i,j}<k})=\Phi_w(\pairing{y_{i,k_i}}_{\pairing{i,k_i}<k})$ and $k_i$ is the greatest $j$ s.t.\ $\pairing{i,j}<k$. Intuitively: $\Phi_r$ simulates $\Phi_w$ on the longest prefixes of solutions available. This argument can be adapted to show that, whenever $f\pmfunction{\Baire}{\Baire}$ is finitely valued (for every $p\in\dom(f)$, $|f(p)|<\aleph_0$), we have $\firstOrderPart{(\parallelization{f})} \weiequiv \ustar{(\firstOrderPart{f})}$.

However, the following counterexample shows that $\ustar{(\cdot)}$ and $\firstOrderPart{(\cdot)}$ do not commute in general.

\begin{proposition}
	There is $f$ with $\firstOrderPart{(\parallelization{f})} \not \weireducible \ustar{(\firstOrderPart{f})}$.
\end{proposition}
\begin{proof}
	Let $\sequence{A_n}{n\in\mathbb{N}}$ and $\sequence{B_n}{n\in\mathbb{N}}$ be as in \thref{thm:countable_sparse_splitting}, namely:
	\begin{enumerate}
		\item for every $n$, $1^n0 \prefix B_n$;
		\item for every $n$, $\emptyset' \not\turingreducible A_n$, $\emptyset' \not\turingreducible B_n$, but $\emptyset' \turingreducible A_n \oplus B_n$;
		\item for every computable sequence $\sequence{e_i}{i\in\mathbb{N}}\subset \mathbb{N}$ and every computable functional $\Psi$ there is $x\in\mathbb{N}$ s.t.\ if for every $i$, $\{e_i\}^{B_i}(x)\downarrow = \gamma_i \prefix B_i$ then 
		\[ \emptyset'(x) \neq \Psi(x, \pairing{\gamma_i}_{i\in\mathbb{N}}, \pairing{A_j}_{j\in\mathbb{N}} ). \]
	\end{enumerate}
	Let $f\mfunction{\mathbb{N}}{\Cantor}$ be defined as $f(0):=\{ B_n \st n\in\mathbb{N}\}$ and $f(n+1):=\{A_n\}$. We show that $\emptyset' \weireducible \firstOrderPart{(\parallelization{f})}$, while $\emptyset'\not\weireducible \ustar{(\firstOrderPart{f})}$, where we are identifying $\emptyset'$ with its characteristic function. The reduction is straightforward: given $n$, consider the input $\sequence{i}{i\in\mathbb{N}}$ for $\parallelization{f}$. The result of $f(0)$ is $B_j$ for some $j\in\mathbb{N}$ (uniformly computable from $B_j$ by checking the first bits), hence using $B_j$ and $A_j=f(j+1)$ we can compute $\emptyset'$.
	
	Recall that, by definition, an input for $\ustar{(\firstOrderPart{f})}$ is a pair $(w, \mathbf{z})$, where
	\begin{itemize}
		\item $\mathbf{z}=\sequence{w_i,x_i}{i\in\mathbb{N}}$, with $w_i\in\Baire$ and $x_i\in\mathbb{N}$,
		\item for every $i$ and every $y_i\in f(x_i)$, $\Phi_{w_i}(y_i)(0)\downarrow=:\gamma_{x_i}$,
		\item for some $k\in\mathbb{N}$, $\Phi_w(\pairing{\gamma_{x_i}}_{i<k})(0)\downarrow$.
	\end{itemize}
	Observe that $ \ustar{(\firstOrderPart{f})} \weireducible F$, where $F$ is the problem that takes in input a sequence $\sequence{v_j}{j\in\mathbb{N}}$ of indexes for continuous functionals and returns both $\pairing{A_n}_{n\in\mathbb{N}}$ and a sequence $\sequence{\eta_i}{i\in\mathbb{N}}$ of finite strings where, for every $i$, $\eta_i$ is a sufficiently long prefix of $B_i$ such that $\Phi_{v_i}(\eta_i)(0)\downarrow$. Intuitively, this can be shown by letting $\sequence{v_j}{j\in\mathbb{N}}$ be the subsequence $\sequence{w_{i_n}}{n\in\mathbb{N}}$ of $\sequence{w_i}{i\in\mathbb{N}}$ s.t.\ $x_{i_n}=0$. Since the output of $F$ contains all the $A_n$'s, it is clear that we can uniformly compute a correct solution of $\ustar{(\firstOrderPart{f})}$ by properly rearranging the columns.
	
	Assume $\emptyset'\weireducible F$ via $\Phi, \Psi$. In particular, for every $n$, $\Phi(n)$ produces a sequence $\sequence{\Phi_i(n)}{i\in\mathbb{N}}$ of (indexes for) computable functionals. Let $\sequence{e_i}{i\in\mathbb{N}}$ be the computable sequence defined as $\Phi_{e_i}(X)(n):=\Phi_{\Phi_i(n)}(X)(0)$. Let also $x\in\mathbb{N}$ be a witness of the property $(3)$ of the sets $A_n$, $B_n$ for the sequence $\sequence{e_i}{i\in\mathbb{N}}$ and the functional $\Psi$.
	
	A valid solution for $F(\Phi(x))$ is the pair $(\pairing{A_j}_{j\in\mathbb{N}},\pairing{\gamma_i}_{i\in\mathbb{N}})$, where $\gamma_i$ is a prefix of $B_i$ s.t.\ $\Phi_{\Phi_i(x)}(\gamma_i)(0)\downarrow$. We have therefore reached a contradiction, as 
	\[ \emptyset'(x) \neq \Psi(x, \pairing{\gamma_i}_{i\in\mathbb{N}}, \pairing{A_j}_{j\in\mathbb{N}} ), \]
    against the definition of Weihrauch reducibility.
\end{proof}

Notice that this implies that the reduction $\firstOrderPart{(\parallelization{f})} \weireducible \parallelization{\firstOrderPart{f}}$ fails in general (by \thref{thm:FOP(parallelization)=u*(FOP)}).

\begin{open}
	\label{q:commutativity}
	Can we characterize exactly the class of problems for which $\firstOrderPart{(\cdot)}$ and $\ustar{(\cdot)}$ commute?
\end{open}

\subsection{Algebraic properties of $\ustar{(\cdot)}$ }
\label{sec:alg_u*}

\begin{theorem}\thlabel{thm:alg_u*}
	${}$
	\begin{enumerate}
		\item $\ustar{f} \sqcup \ustar{g} \weireducible \ustar{(f \sqcup g)} $;
		\item $ \ustar{(f \sqcap g)}\weiequiv \ustar{f} \sqcap \ustar{g} $;
		\item $\ustar{f} \times \ustar{g} \weireducible \ustar{(f \times g)} $;
		\item $\ustar{(f^*)}\weiequiv (\ustar{f})^* \weiequiv \ustar{f}$;
		\item $\ustar{(\parallelization{f})}\weiequiv \parallelization{(\ustar{f})} \weiequiv \parallelization{f}$;
		\item $\ustar{(f\compproduct g)} \weireducible \ustar{f} \compproduct \parallelization{g}$.
	\end{enumerate}
	None of the above reduction can be reversed. In fact, 	
	\begin{enumerate}
		\setcounter{enumi}{6}
		\item \label{itm:counterex1} there are $f,g$ s.t.\ $\ustar{f} \compproduct \ustar{g}\not\weireducible \ustar{(f \times g)}$; 
		\item \label{itm:counterex2} there are $f,g$ s.t.\ $\ustar{(f \times g)} \not \weireducible \ustar{f} \compproduct \ustar{g}$.
	\end{enumerate}
	
\end{theorem}
\begin{proof}
	\begin{enumerate}
		\item The reduction is straightforward from the fact that $\ustar{f}\weireducible\ustar{(f \sqcup g)}$ and $\ustar{g}\weireducible\ustar{(f \sqcup g)}$, as $\sqcup$ is the join in the Weihrauch lattice. An explicit reduction can be obtained by considering the map $(i,(w,\seq{x}))\mapsto (w,(i,\seq{x}))$ for every $\seq{x}:=\sequence{x_n}{n\in\mathbb{N}}$.
		
		A counterexample for the right-to-left reduction can be obtained by choosing any $f$ and $g$ s.t.\ $f\not\weireducible \parallelization{g}$ and $g\not\weireducible \parallelization{f}$. This condition is equivalent to  $\parallelization{f} \weiincomparable \parallelization{g}$, and, as example, we can choose $f=\WKL$ and $g=\COH$. Any forward functional for the reduction, upon input $(w, \sequence{i_n,x_n}{n\in\mathbb{N}})$, should commit, after finite time, to some $b\in\{0,1\}$ specifying whether it is producing an input for $\ustar{f}$ or $\ustar{g}$. Assume without loss of generality that $b=0$ (the proof in case $b=1$ is obtained by swapping the roles of $f$ and $g$). By changing 
		the input so that $i_n = 1-b = 1$ for every sufficiently large $n$, we can exploit the reduction $\ustar{(f \sqcup g)} \weireducible \ustar{f} \sqcup \ustar{g}$ to obtain a reduction $\ustar{g}\weireducible \ustar{f} \weireducible \parallelization{f}$, against $g \not\weireducible \parallelization{f}$.
		
		\item The reduction follows from $\ustar{(f \sqcap g)}\weireducible \ustar{f}$ and $\ustar{(f \sqcap g)}\weireducible \ustar{g}$, as $\sqcap$ is the meet in the Weihrauch lattice. An explicit reduction can be obtained by considering the map $(w,\sequence{x_n,z_n}{n\in\mathbb{N}}))\mapsto ((w,\sequence{x_n}{n\in\mathbb{N}}),(w,\sequence{z_n}{n\in\mathbb{N}}))$. 
		
        To prove the right-to-left reduction, let $((w_f,\sequence{x_n}{n\in\mathbb{N}}),(w_g,\sequence{z_n}{n\in\mathbb{N}}))$ be a valid input for $\ustar{f}\sqcap \ustar{g}$. Observe that, for every solution $\sequence{b_{\pairing{i,j}},y_{\pairing{i,j}}}{\pairing{i,j}\in\mathbb{N}}$ for $\parallelization{f\sqcap g}(\sequence{x_i,z_j}{\pairing{i,j}\in\mathbb{N}})$, either $(\forall i)(\exists j)(b_{\pairing{i,j}}=0)$, or $(\forall j)(\exists i)(b_{\pairing{i,j}}=1)$. We consider the input $(w,\sequence{x_i,z_j}{\pairing{i,j}\in\mathbb{N}})$, where $w\in\Baire$ is an index for the functional that, upon input $\sequence{b_{\pairing{i,j}},y_{\pairing{i,j}}}{\pairing{i,j}<k}$, works as follows: by exhaustive search, it produces two sequences $\seq{s}=\sequence{s_n}{n}$ and $\seq{t}=\sequence{t_m}{m}$ of maximal length s.t.\ $s_n=y_{\pairing{n,j}}$ for some $j$ s.t.\ $b_{\pairing{n,j}}=0$, and symmetrically, $t_m=y_{\pairing{i,m}}$ for some $i$ s.t.\ $b_{\pairing{i,m}}=1$. It then simulates $\Phi_{w_f}(\seq{s})$ and $\Phi_{w_g}(\seq{t})$ for $k$ steps, and halts iff one of the two halts. The definition of $\ustar{(\cdot)}$ guarantees that $\Phi_w$ halts whenever the input (and hence one between $\seq{s}$ and $\seq{t}$) is sufficiently long.
				
		\item The left-to-right reduction is straightforward: given $(w_f, \sequence{x_i}{i\in\mathbb{N}})$ and $(w_g, \sequence{z_i}{i\in\mathbb{N}})$ we can uniformly compute the input $(w, \sequence{x_i,z_i}{i\in\mathbb{N}})$ for $\ustar{(f\times g)}$, where $w$ is s.t.\ $\Phi_w(\pairing{(a_i,b_i)}_{i<k})$ simulates $\Phi_{w_f}(\pairing{a_i}_{i<k})$ and $\Phi_{w_g}(\pairing{b_i}_{i<k})$. The fact that the right-to-left reduction can fail follows from point (\ref{itm:counterex2}).
		
		\item It follows from the fact that $f^* \weireducible \ustar{f}$ and that $\ustar{(\cdot)}$ is a closure operator. 
		
		\item It follows from the fact that $\ustar{f} \weireducible \parallelization{f}$ and that the parallelization is a closure operator. 
		
		\item This is essentially a definition-chasing exercise. Let $(w,\sequence{w_i,z_i}{i\in\mathbb{N}})$ be an input for $\ustar{(f\compproduct g)}$. By definition, for every name $\pairing{p_i}_{i\in\mathbb{N}}$ of a solution of $\parallelization{f\compproduct g}(\sequence{w_i,z_i}{i\in\mathbb{N}})$ there is $k\in\mathbb{N}$ s.t.\ $\Phi_w(\pairing{p_i}_{i<k})(0)\downarrow$. For every $i$, $p_i$ is a name for a pair $(v_i, u_i)$, where $v_i\in g(z_i)$ and $u_i\in f(\Phi_{w_i}(v_i))$. Given $(w,\sequence{w_i,z_i}{i\in\mathbb{N}})$, we can uniformly compute $r\in\Baire$ s.t.\
		\[ \Phi_r( \pairing{t_i}_{i\in\mathbb{N}} )= \pairing{s, \pairing{\Phi_{w_i}(t_i)}_{i\in\mathbb{N}}}, \]
		where $s\in \Baire$ is s.t.
		\[ \Phi_s(\pairing{y_i}_{i<k})=\Phi_w(\pairing{ t_i, y_i}_{i<k}). \]
		Intuitively: we do not lose any computational power if the sequence $\sequence{v_i}{i\in\mathbb{N}}$ of $g$-solutions is passed as input to $\ustar{f}$. It is routine to check that $(r,\sequence{z_i}{i\in\mathbb{N}})$ is a valid input for $\ustar{f}\compproduct \parallelization{g}$ and that a solution for $\ustar{(f\compproduct g)})(w,\sequence{w_i,z_i}{i\in\mathbb{N}}$ can be uniformly computed from a solution $(\ustar{f}\compproduct \parallelization{g})(r,\sequence{z_i}{i\in\mathbb{N}})$ by truncating the sequence of solutions for $\parallelization{g}(\sequence{z_i}{i\in\mathbb{N}}$.

		A counterexample for the right-to-left reduction is readily obtained choosing $f=g=\CNatural$. Indeed, $\CNatural$ is closed under composition and under $\ustar{(\cdot)}$ (see e.g.\ \thref{thm:fop_lim}), hence $\ustar{(\CNatural\compproduct\CNatural)}\weiequiv \CNatural$, while $\mflim \weireducible \CNatural \compproduct\parallelization{\CNatural}$.
		
		\item As a counterexample, it is enough to take $f=g=\mflim$.  
		
		\item The proof follows the one of \thref{thm:algebraic_rules_fop}(\ref{itm:fop_alg_counterexample2}). Let $\sequence{A_n}{n\in\mathbb{N}}$ and $\sequence{B_n}{n\in\mathbb{N}}$ be as in \thref{thm:countable_sparse_splitting} and let $A:=A_0$, $B:=B_0$. In particular we have 
		\begin{itemize}
			\item $\emptyset'\not\turingreducible A$, $\emptyset'\not\turingreducible B$, $\emptyset'\turingreducible A \oplus B$
			\item for every $e$, if $\{e\}^{A\oplus B} = \emptyset'$ then the map sending $x$ to the prefix of $B$ used in the computation $\{e\}^{A\oplus B}(x)$ is not $B$-computable.
		\end{itemize}
		
		If we identify each set with its characteristic function then clearly $\emptyset'\weireducible \ustar{(A\times B)}$. On the other hand $\emptyset'\not\weireducible\ustar{A}\compproduct \ustar{B}$. Notice indeed that, by definition, an input for $\ustar{A}$ is a pair $(w,\mathbf{x})$, with $\mathbf{x}=\sequence{x_i}{i\in\mathbb{N}}$ being a sequence of natural numbers. Without loss of generality, we can assume that $x_i=i$ and that $\Phi_w$ is just computing a (finite) prefix of $A$. In other words, we can assume that the input of $\ustar{A}$ is the index of a continuous function that, given $A$, produces a finite prefix of $A$ (the same argument applies also to $B$).
		
		Assume there is a reduction $\emptyset'\weireducible \ustar{A}\compproduct\ustar{B}$ as witnessed by $\Phi$,$\Psi$. The input for $\ustar{A}\compproduct \ustar{B}$ is a pair $(e, w)$ where $w$ is an input for $\ustar{B}$ and $\Phi_e(\ustar{B}(w))$ is an input for $\ustar{A}$. 
		
		For every $n$, $\Phi(n)=\coding{\Phi_1(n),\Phi_2(n)}$ is s.t.\ $\Phi_2(n)$ is the index for a computable functional that, given $n$ and $B$, produces a prefix $B[m]$ of $B$ and, $\Phi_{\Phi_1(n)}(B[m])$ is the index of a computable functional that, given $A$, computes $\emptyset'$. Since all computations are done uniformly, this corresponds to the existence of two indexes $e$ and $i$ for computable functions s.t.\ $\{e\}^{B}(n)= B[m]$ and $\{i\}^{A\oplus B[m]}(n)=\emptyset'(n)$. Hence, the existence of a Weihrauch reduction contradicts the properties of $A$ and $B$.\qedhere
	\end{enumerate}
\end{proof}

\begin{theorem}
	\thlabel{thm:ustar_jumps}
	For every multi-valued function $f$, $\ustar{(f')} \strongweireducible (\ustar{f})'$.	Moreover, if $f$ is a cylinder then $\ustar{(f')}\strongweiequiv(\ustar{f})'$.
\end{theorem}
\begin{proof}
	The first statement follows from the definitions: notice indeed that the domains of $\ustar{(f')}$ and $(\ustar{f})'$ share the same underlying set ($\Baire\times \infStrings{X}$), the only difference being in the respective representations. Let $(w,\sequence{x_n}{n\in\mathbb{N}})$ be an input for $\ustar{(f')}$. Given a double sequence $\sequence{p_{n,m}}{n,m\in\mathbb{N}}$ where $\lim_{m\to\infty} p_{n,m}=:p_n$ is a name for $x_n$, we can uniformly compute a sequence in $\Baire$ converging to $\pairing{w, \pairing{p_n}_{n\in\mathbb{N}}}$, i.e.\ to a name of the input $(w,\sequence{x_n}{n\in\mathbb{N}})$ for $(\ustar{f})'$.
	
	Assume now that $f$ is a cylinder, so that $f\strongweiequiv \id \times f$ and $f'\strongweiequiv (\id\times f)'$. We show that $(\ustar{f})'\strongweireducible \ustar{((\id \times f)')}$, and the claim follows from the fact that $\ustar{(\cdot)}$ is strong-degree-theoretic. A name for the input $(w,\sequence{x_n}{n\in\mathbb{N}})$ for $(\ustar{f})'$ is (essentially) a sequence $\sequence{w_i}{i\in\mathbb{N}}$ converging to $w\in\Baire$ and a double sequence $\sequence{q_{n,m}}{n,m\in\mathbb{N}}$ where $\lim_{m\to\infty}{q_{n,m}}=: q_n$ is a name for $x_n$. Clearly the sequence $\sequence{\pairing{w_m,q_{n,m}}}{m\in\mathbb{N}}$ converges to a name for $(w,x_n)$. Let $r\in \Baire$ be s.t.\ 
	\[ \Phi_r(\pairing{\pairing{w,t_0},\hdots,\pairing{w,t_{k-1}}}) = \Phi_w(\pairing{t_i}_{i<k}). \]
	It follows that  $\pairing{r, \pairing{w_m,q_{n,m}}_{m\in\mathbb{N}} }$ is a name for a valid input of $\ustar{((\id \times f)')}$ and that, applying $\ustar{((\id \times f)')}$, we obtain a solution for $(\ustar{f})'(w,\sequence{x_n}{n\in\mathbb{N}})= \ustar{f}(w,\sequence{x_n}{n\in\mathbb{N}})$.
\end{proof}

The reduction $(\ustar{f})'\strongweireducible \ustar{(f')}$ can fail if $f$ is not a cylinder, see the comments after \thref{thm:fop_jumps}. An explicit counterexample is the constant function.

We now show that, under the relatively mild assumption that $\parallelization{f}$ is a cylinder, \thref{thm:FOP(parallelization)=u*(FOP)} can be strengthened as follows: 

\begin{theorem}
	\thlabel{thm:FOP(parallelization)=u*(FOP)_strong}
 	For every first-order $f$ and every $n\in\mathbb{N}$, if $\parallelization{f}$ is a cylinder then $\firstOrderPart{((\parallelization{f})^{(n)})}\strongweiequiv (\ustar{f})^{(n)}$.
\end{theorem}
\begin{proof}
	The statement can be proved by induction on $n\in\mathbb{N}$. Fix a first-order problem $f$ such that $\parallelization{f}$ is a cylinder and, for the sake of readability, let $F:=\parallelization{f}$.

	Observe first of all that $\ustar{f}$ is a first-order cylinder. Indeed, let $g$ be a first-order problem s.t.\ $g\weireducible \ustar{f}$. Let $\Phi, \Psi$ be two functionals witnessing the strong reduction $g\strongweireducible F$. For every name $z$ of an input for $g$, let $\seq{x}$ be the input for $F$ named by $\Phi(z)$. For every name $\pairing{y_i}_{i\in\mathbb{N}}$ for a solution of $F(\seq{x})$, there is $k$ s.t.\ $\Psi(\pairing{y_i}_{i<k})(0)\downarrow \in g(z)$. In particular, letting $w$ be an index for $\Psi$, the maps $z\mapsto (w,\seq{x})$ and $\Psi$ witness $g\strongweireducible \ustar{f}$.

    Moreover, $\firstOrderPart{F}$ is a first-order cylinder (\thref{thm:fo_cylinder}), and therefore \thref{thm:FOP(parallelization)=u*(FOP)} implies $\ustar{f}\strongweiequiv \firstOrderPart{F}$. This proves the statement for $n=0$.
    
    Assume now that the claim holds up to $n$. Since $F$ is a cylinder, and the jump of a cylinder is a cylinder (\cite[Prop.\ 6.14]{BGP17}), applying \thref{thm:fop_jumps} we have $\firstOrderPart{(F^{(n+1)})}\strongweiequiv (\firstOrderPart{(F^{(n)})})'$. By the inductive step
	\[ \firstOrderPart{(F^{(n+1)})} \strongweiequiv ((\ustar{f})^{(n)} )' \strongweiequiv (\ustar{f})^{(n+1)}. \qedhere\]
\end{proof}

\begin{corollary}
    \thlabel{thm:FOP(parallelization)=u*(FOP)_strong_cor}
	For every first-order $f$ and every $n\in\mathbb{N}$, if $\parallelization{f}$ is a cylinder then
	\[ \firstOrderPart{((\parallelization{f})^{(n)})}\strongweiequiv (\ustar{f})^{(n)}\strongweiequiv \ustar{(f^{(n)})} .\]     
\end{corollary}
\begin{proof}
    The first equivalence is \thref{thm:FOP(parallelization)=u*(FOP)_strong}. Since $(\parallelization{f})' \strongweiequiv\parallelization{(f')}$ \cite[Prop.\ 5.7(3)]{BolWei11}, the second equivalence can be obtained by applying \thref{thm:FOP(parallelization)=u*(FOP)_strong} to $g=\parallelization{f^{(n)}}$.
\end{proof}

\section{First-order part and diamond}
\label{sec:fop_diamond}

The diamond operator was introduced in \cite[Def.\ 9]{NP18} using generalized register machines, and intuitively $f^\diamond$ captures the possibility of calling $f$ as oracle an arbitrary but finite number of times during a computation (with the constraint of having a c.e.\ condition that tells you when no more calls will be made). In \cite[Def.\ 4]{Westrick20diamond}, the author gives an alternative definition by means of a ``higher-order'' model of computation, and shows that the diamond operator corresponds to closure under compositional product for pointed problems (\cite[Thm.\ 1]{Westrick20diamond}).

If $f$ is a first-order problem, a clear upper bound for $\firstOrderPart{(\parallelization{f})}$ is given by $f^\diamond$. It is therefore natural to ask what is the relation between the first-order part, the unbounded finite parallelization, and the diamond operators. 

In the following, we will mostly use the game-theoretic definition introduced by \cite[Def.\ 4.1 and def.\ 4.3]{HJ16}. 

\begin{definition}\label{sec:def-diamond}
	Let $f,g\pmfunction{\Baire}{\Baire}$ be two partial multi-valued functions. We define the \textdef{reduction game} $G(f\to g)$ as the following two-player game: on the first move, Player 1 plays $x_0\in \dom(g)$, and Player 2 either plays an $x_0$-computable $y_0\in g(x_0)$ and declares victory, or responds with an $x_0$-computable instance $z_1$ of $f$.
	
	For $n > 1$, on the $n$-th move (if the game has not yet ended), Player 1 plays a solution $x_{n-1}$ to the input $z_{n-1}\in\dom(f)$. Then Player 2 either plays a $\coding{x_0,\hdots,x_{n-1}}$-computable solution to $x_0$ and declares victory, or plays a $\coding{x_0,\hdots,x_{n-1}}$-computable instance $z_n$ of $f$.
	
	If at any point one of the players does not have a legal move, then the game ends with a victory for the other player. Player 2 wins if it ever declares victory (or if Player 1 has no legal move at some point in the game). Otherwise, Player 1 wins.
	
	We say that \textdef{$g$ is Weihrauch reducible to $f$ in the generalized sense}, and write $g \le_{\mathrm{gW}} f$, if Player 2 has a computable winning strategy for the game $G(f\to g)$, i.e.\ there is a Turing functional $\Phi$ s.t.\ Player 2 always plays $\Phi(\coding{x_0,\hdots,x_{n-1}})$, and wins independently of the strategy of Player 1.
	
	We described the game assuming that $f,g$ have domain and codomain $\Baire$. The definition can be extended to arbitrary multi-valued functions, and the moves of the players are \emph{names} for the instances/solutions.
	
	For every $f\pmfunction{X}{Y}$ we define $f^\diamond\pmfunction{\mathbb{N}\times\Baire}{\finStrings{Y}}$ as the following problem:
	\begin{itemize}
		\item $\dom(f^\diamond)$ is the set of pairs $(e,d)$ s.t.\ Player 2 wins the game $G(f\to \id)$ when Player 1 plays $d$ as his first move, and $\Phi_e$ is a winning strategy for Player 2;
		\item a solution is the list of moves of Player $1$ for a run of the game.
	\end{itemize}
\end{definition}

Intuitively, in the reduction game $G(f\to g)$, Player 1 plays the role of the oracle $f$, while Player 2 plays the role of the algorithm trying to compute a solution for $g$, calling the oracle finitely many times. It is easy to see that  $g\weireducible f^\diamond$ if and only if $g \le_{\mathrm{gW}} f$. In particular, this interpretation means that $e$ above can be read as an index of a computation containing a special instruction, namely a call to the oracle $f$.

\begin{remark}
    Observe that if $f$ is first-order then so is $f^\diamond$.
\end{remark}

We can compare the intuition given in the above paragraph with the informal description of $\ustar{f}$: one would expect that, loosely speaking, a problem corresponding to ``an arbitrary number of applications of $f$ in sequence'' always computes ``an arbitrary number of parallel applications of $f$''. The next Proposition shows that this is indeed the case.

\begin{proposition}\thlabel{thm:ustar<=diamond}
	For every multi-valued function $f\pmfunction{X}{Y}$, $\ustar{f}\weireducible f^\diamond$.  
\end{proposition}
\begin{proof}
	It is enough to notice that, for every partial multi-valued functions $f,g$, if $g\weireducible \ustar{f}$, then by \thref{thm:game-def-ustar}, Player 2 has a computable winning strategy for the game $U(f\rightarrow g)$. The same strategy allows him to win the game $G(f\rightarrow g)$ for every initial move of Player 1. 
\end{proof}

\begin{remark}
	We notice that, if $g \strongweireducible \ustar f$, then a solution to any input $a$ of $g$ can be found just from the list of the moves of Player 1 in the game $U(f\rightarrow g)$. In particular, if $g\strongweireducible \ustar f$ then $g\strongweireducible f^\diamond$, i.e.\ $\ustar{f}\strongweireducible f^\diamond$.
\end{remark}

We now prove a generalization of the aforementioned result of Westrick which is useful to present upper bounds for $f^\diamond$.

\begin{theorem}\thlabel{sec:theo-red-diam}
	For every two multi-valued functions $f,g \pmfunction \Baire \Baire$ such that $g$ is pointed, we have that
	\[
	g\compproduct f \weireducible g \Rightarrow f^\diamond \weireducible g.
	\]
\end{theorem}

\begin{proof}	
	The proof is essentially the same as the one of \cite[Thm.\ 1]{Westrick20diamond}: we try to follow the same notation as much as is allowed by the other conventions we have fixed so far in this document.
	
	Let $\Delta$ and $\Gamma$ respectively the forward and backward functionals of  $g\compproduct f\weireducible g$. Fix a universal Turing functional $\UTM$ s.t.\ $\UTM(n,m,p,q):= \Phi_n(m,p,q)$. Consider the function $F\function{\mathbb{N}}{\mathbb{N}}$ s.t.\ $\UTM(F(n),e,d,\pairing{y_i}_{i<k})$ works as follows:
	
	\begin{description}
		\item (1): We think of $(e,d)$ as an input for $f^\diamond$. We simulate the computation $\Phi_e(d)$. If Player 2 declares victory then go to (3). Otherwise, Player 2 makes an oracle call. Let $Q(e,d, \pairing{})$ be the instance of $f$ that we are asking to be solved by $f$. If $k=0$ then go to $(4)$, otherwise go to (2-$0$). 
		
		\item (2-$i$): Use $y_i$ as an answer to the $i$-th oracle call. Carry on with the computation until either the computation halts (i.e.\ Player 2 declares victory), in which case we move to (3), or the $(i+1)$-st request to use $f$ is made. If $i+1< k$, we go to (2-$i+1$), otherwise go to (4).
		
		\item (3): We output a computable element of $\dom(g)$, which we can do by the assumption that $g$ is pointed.
		
		\item (4): Let $w_k\in \Baire$ be s.t.\ $\Phi_{w_k}(y_k)=\UTM(n,e,d,\pairing{y_i}_{i\le k})$. Clearly $w_k$ is uniformly computable from $n,e,d,\pairing{y_i}_{i<k}$. The computation $\UTM(F(n),e,d,\pairing{y_i}_{i<k})$ then returns 
		\[ \Delta(w_k,Q(e,d,\pairing{y_i}_{i<k})), \]
		where $Q(e,d,\pairing{y_i}_{i<k})$ is the instance of $f$ submitted upon the $(k+1)$-st request to use $f$.
	\end{description}
	
	Clearly $F$ is total and computable, hence, by the relativized recursion theorem, for every $e,d$ there is a fixed point $n$ for the function $F$. 
		
	We now define the maps $\Phi,\Psi$ which will witness the reduction $f^\diamond \weireducible g$. Define $\Phi:=(e,d)\mapsto \UTM(n,e,d, \pairing{})$. The functional $\Psi$, when executed with input $z_0$, works as follows: we divide the computation in stages, where, at the beginning of each stage $s$ we have defined $z_s\in \Baire$ and a finite sequence $p_s:=\sequence{y_i}{i<s}$. At stage $0$ we define $p_0:=\pairing{}$. At every stage $s$, we run the computation $\UTM(n,e,d,\pairing{y_i}_{i<s})$. If the computation enters condition (3) above then we return $\sequence{y_i}{i<s}$. If the computation enters condition (4), we let  $(y_s,z_{s+1}):=\Gamma(z_s, (w_s,Q(e,d,\pairing{y_i}_{i<s})))$ where $w_s\in \Baire$ is computed by $\UTM(n,e,d,\pairing{y_i}_{i<s})$ as defined above, and we go to stage $s+1$. 
	
	We now show that the functionals $\Phi$ and $\Psi$ witness the reduction $f^\diamond\weireducible g$. Let us first notice that, if for every $k$ and every $i<k$, $y_i\in f(Q(e,d,\pairing{y_j}_{j<i}) )$ then  
	\[ \UTM(n,e, d, \pairing{y_i}_{i<k})\in \dom(g).\]
	To prove this it is enough to notice that all the $(y_0,\hdots, y_{k-1})$ as above ordered by extension form a well-founded tree (possibly with continuum-sized branching). The key observation is that the sequences $(y_0,\dots,y_{k-1})$ that satisfy the hypotheses above are exactly the (possibly partial) runs of the game $G(f\to\id)$ such that Player 1's first move is $d$ and Player 2 plays according to the strategy $e$. The claim can therefore be proved by induction on the rank of this tree (as is done in \cite[Claim on p.\ 5]{Westrick20diamond}). This implies that $\UTM(n,e,d, \pairing{} )=\Phi(e,d)\in\dom(g)$.

	Let us now show that if $z_0\in g(\UTM(n,e,d,\pairing{}))$ then $\Psi(z_0)$ is a valid solution for $f^\diamond(e,d)$. Assume that, at each stage $s$, the sequence $\sequence{y_i}{i<s}$ as defined by $\Psi(z_0)$ is a valid (partial) run of the game $G(f\to \id)$ and $z_s \in g(\UTM(n,e,d,\pairing{y_i}_{i<s}))$. We show that either $\sequence{y_i}{i<s}$ is a solution of $f^\diamond(e,d)$ or $\sequence{y_i}{i< s+1}$ and $z_{s+1}$ satisfy the inductive hypotheses. This suffices to prove the claim as the case $s=0$ is trivially verified. If the computation $\UTM(n,e,d,\pairing{y_i}_{i<s})$ enters case (3) then by construction Player 2 declares victory, hence we are done. Otherwise the computation enters case (4). By the inductive step, 
	\[ z_s\in g(\UTM(n,e,d,\pairing{y_i}_{i<s})) = g(\Delta(w_s,Q(e,d,\pairing{y_i}_{i<s}) )).\]
	Since $\Gamma$ is the backward functional of the reduction $g\compproduct f \weireducible g$, we obtain that the pair  
	\[ (y_s,z_{s+1})=\Gamma(z_s, (w_s,Q(e,d,\pairing{y_i}_{i<s})))\] 
	is s.t.\ $y_s\in f(Q(e,d,\pairing{y_i}_{i<s}))$ and $z_{s+1}\in g(\Phi_{w_s}(y_s))= g(\UTM(n,e,d,\pairing{y_i}_{i<s+1} ))$.	
\end{proof}
		
Notice that we only need the stronger problem $g$ to be pointed, whereas there is no such assumption on $f$.

It is natural to ask whether the implication $f\compproduct g\weireducible g\Rightarrow f^\diamond \weireducible g$ also holds. Interestingly, although $f\compproduct g\weireducible g$ implies that $f^{[n]}\weireducible g$ for any natural number $n$, this is not enough to conclude that $f^\diamond\weireducible g$.

\begin{proposition}
	There are partial multi-valued functions $f,g \pmfunction\Baire \Baire$ such that $g$ is pointed, $f\compproduct g\weireducible g$, but $f^\diamond\not\weireducible g$.
\end{proposition}
\begin{proof}			
	Let $\{q_i\st i\in\mathbb{N}\}$ be a set of Turing-incomparable sets. We define $f\pmfunction{\Cantor}{\Cantor}$ as $f(0^\mathbb{N}):=\{\str{i}\concat q_i\st i\in\mathbb{N}\}$ and $f(q_{i+1}):=q_i$ for every $i$. We also let $g:=\bigsqcup_{n\in\mathbb{N}}f^{[n]}$.
	
	Let us identify $q_0$ with the constant map $x\mapsto q_0$. Pauly and Yoshimura \cite{Dagstuhl18} observed that $q_0\weireducible f^\diamond$ but for every $n$, $q_0\not\weireducible f^{[n]}$. This, in particular, implies that $q_0\not\weireducible g$, hence $f^\diamond \not \weireducible g$. On the other hand, to show that $f\compproduct g\weireducible g$ it is enough to notice that an input $(p_{n-1},(n,(\sequence{p_i}{i<n-1},x)))$ for $f\compproduct g$ can be uniformly mapped to $(n+1, (\sequence{p_i}{i< n},x))\in \dom(g)$ and that $(f\compproduct g)(p_{n-1},(n,(\sequence{p_i}{i<n-1},x))) = g(n+1, (\sequence{p_i}{i< n},x))$.
\end{proof}
		
\thref{sec:theo-red-diam} will be a useful tool in various proofs concerning the question of when it is the case that $\ustar f\weiequiv f^\diamond$. The first result we give is essentially a direct proof of $\ustar{\CNatural}\weiequiv \CNatural^\diamond$ (which holds by \thref{thm:FOP(parallelization)=u*(FOP)}, see \thref{thm:fop_lim}). The main idea behind the proof is that, in order to compute $\CNatural^\diamond (e,x)$, we can simulate a run of the game $G(\CNatural\to \id) $ (i.e.\ we can simulate a computation with oracle $\CNatural$) and ``guess'' the possible moves of Player 1 (i.e.\ guess the results of the oracle calls). This generates a subtree of $\baire$ corresponding to the possible sequences of moves for Player $1$. Observe that, if the guess is a legal move for Player 1 then we are guaranteed (by the definition of diamond) that Player $2$ will declare victory after finitely many moves. On the other hand, if the guess is not a legal move for Player 1 then we have no information on the behavior of $\Phi_e$, which may fail to produce a valid input for $\CNatural$ (it may produce a name for an empty set, or also fail to produce an infinite string). However, since the graph of $\CNatural$ is closed, it is c.e.\ (relatively to the input) to check whether we are guessing a wrong answer to an oracle call. Since every finite string is extendible to a valid input for $\CNatural$, this guarantees that we can always compute a valid input for $\CNatural$, and therefore obtain a reduction $\CNatural^\diamond\weireducible \ustar{\CNatural}$. 

These ideas can be formally presented as follows.

\begin{proposition}
	\thlabel{thm:total-func-ustar-diam}
	If $f\mfunction{\Baire}{\mathbb{N}}$ is s.t.\ $\{(n,x) \st n\in f(x) \}\in \lightfacePi^0_1$ then $\ustar{f}\weiequiv f^\diamond$.
\end{proposition}
\begin{proof}
	The reduction $\ustar{f}\weireducible f^\diamond$ follows from \thref{thm:ustar<=diamond}. To prove that $f^\diamond\weireducible \ustar{f}$, we show that $\ustar f\compproduct f\weireducible \ustar f$ (\thref{sec:theo-red-diam}). Since $\ustar f\compproduct f$ is a first-order problem, it is enough to show that $\ustar f\compproduct f\weireducible \parallelization f$ (\thref{thm:FOP(parallelization)=u*(FOP)}).
	
	We define the forward functional $\Phi$ as the map that sends an input $(p,x)\in \dom(\ustar f\compproduct f)$ to $\sequence{x_n}{n\in\mathbb{N}\cup \{-1\}}$ defined as follows: first of all we define $x_{-1}:=x$. 
	
	Before the formal definition of $x_n$ with $n\ge 0$, we give an intuition of how the functional works. Intuitively, if $n\in f(x)$, then $\Phi_p(n)$ is (a name of) $(w_n, \sequence{z^n_i}{i\in\mathbb{N}})$. We would like to run $\Phi_p(n)$ for every guess $n$ of a possible solution to $f$, and define $x_{\pairing{n,i}}:=z^n_i$. However, there is one major obstacle that needs to be overcome in order to make this work: there is no guarantee that, upon being given an $n$ such that $n\not\in f(x)$, $\Phi_p(n)$ will actually return a valid input for $\ustar f$. 
	
	To address this issue, we build $\Phi(p,x)$ in stages. Let $\varphi$ be a $\Delta^0_1$ formula such that 
	\[ n\in f(x) \iff (\forall k)(\varphi(n,x,k)),\] 
	which exists by our assumptions on $f$.

	At every stage $s$, for every unmarked $n<s$, we check whether $(\forall k<s)(\varphi(n,x,k))$ holds: if it does, we run $\Phi_p(n)$ for $s$ steps. We interpret its output as a prefix of $\pairing{q^n_i}_{i\in\mathbb{N}}$ and write the prefix of $q^n_{i+1}$ as prefix of $x_{\pairing{n,i}}$. If there is $k<s$ s.t.\ $\varphi(n,x,k)$ does not hold, then $n\notin f(x)$. We label $n$ as marked and we extend $x_{\pairing{n,i}}$ concatenating with $0^\mathbb{N}$. After having checked every $n$ we go to the next stage. 	
	
	We claim that $\Phi(p,x)$ as defined above is an instance of $\parallelization f$: indeed, suppose for a contradiction that this were not the case, then there are minimal $n$ and $i$ such that $x_{\pairing{n,i}} \notin\dom(f)$. Notice that, by the totality of $f$, this means that $x_{\pairing{n,i}}$ is a finite string. This implies that $n\notin f(x)$, otherwise, by definition of compositional product, $x_{\pairing{n,i}}=\pi_{\langle n,i\rangle}(\Phi(p,x))=\pi_{i+1}(\Phi_p(n))$, i.e.\ $x_{\pairing{n,i}}\in \dom(f)$. But if $n\not\in f(x)$, there is a stage $s$ such that $\neg\varphi(n,x,s)$ holds, and hence, by construction, $x_{\pairing{n,i}}$ is an infinite string, which is a contradiction.  
	This proves that $\Phi(p,x)\in\dom\parallelization f$.
	
	We now define the return functional $\Psi$: given $(p,x)$ and $\langle n_j\rangle_{j\in\Nb\cup \{-1\}}\in \parallelization f(\Phi(p,x))$, we run
	\[
	\UTM(\pi_0 \Phi_p(n_{-1}),\langle n_{\langle n_{-1},i\rangle}:i\in\Nb\rangle),
	\]
	where $\UTM$ is a fixed universal Turing functional. We are guaranteed that this computation will converge 
	by definition of unbounded parallelization, and it will do so after finitely many computational steps, say $k$. The map $\Psi$ then outputs the pair $(n_{-1},\langle n_{\langle n_{-1},i\rangle} : i< k \rangle)$, which clearly is in $\ustar f\compproduct f(p,x)$, thus concluding the proof.
\end{proof}

\begin{proposition}
	\thlabel{thm:star=diamond}
	For every $f\mfunction{\Baire}{k}$ s.t.\ $\{(x,n) \st n\in f(x) \}\in \lightfacePi^0_1$ we have  $f^*\weiequiv f^\diamond$. 
\end{proposition}
\begin{proof}
	We refine the proof of \thref{thm:total-func-ustar-diam} and show that $f^*\compproduct f \weireducible f^*$. This follows from the fact, since $f$ has codomain $k$, given $(p,x)\in \dom(f^*\compproduct f)$ we can uniformly compute a bound $b\in \mathbb{N}$ s.t.\ for every $n<k$, either $\Phi_p(n)(0)\downarrow < b$ or we find a witness $<b$ of the fact that $n\notin f(x)$. Letting $C:=\{ n<k \st \Phi_p(n)(0)\downarrow$ in $b$ steps$\}$, we can define
	\[ N:= \sum_{n\in C} \Phi_p(n)(0). \] 
	Since $f$ is total, it is not hard to show that a solution for $(f^*\compproduct f)(p,x)$ can be uniformly computed using $f^N$.
\end{proof}

The hypotheses of \thref{thm:total-func-ustar-diam} are sufficient for $\ustar{f}\weiequiv f^\diamond$, but not necessary. The following theorem is based on the same ideas, but uses weaker hypotheses and therefore is more general.
Before we can state it, we need to recall what a complete problem is.

\begin{definition}
	We adopt the following convention: for every $p\in\Baire$, we denote by $p-1\in \Baire\cup \baire$ the (finite or infinite) string $(p-1)(n):=p(i_n-1)$, where $i_n$ is the $n$-th non-zero element of $p$.
	
	Let $(X,\delta_X)$ be a represented space. The \emph{completion} $(\ol X,\delta_{\ol X})$ of $(X,\delta_X)$ is defined as follows: $\ol X:= X\cup \{\bot\}$ (with the understanding that $\bot\not\in X$), and $\delta_{\ol X}:\Baire \to \ol X$ is given by
	\[ \delta_{\ol X}(p):=
		\begin{cases}
			\delta_X(p-1) & \text{ if } p-1\in \dom(\delta_X)\\
			\bot & \text{ otherwise.}
		\end{cases}\]
	Let $f \pmfunction{X}{Y}$ be a problem. Then, the \emph{completion of $f$} is the problem $\ol f \mfunction{\ol X}{\ol Y}$ defined as
	\[ \ol f(x):= \begin{cases}
			f(x) & \text{ if } x\in \dom(f)\\
			\ol Y & \text{ otherwise.}
		\end{cases} \]
	We say that a problem $f$ is \emph{complete} if it is Weihrauch equivalent to its completion: $f\weiequiv\ol f$.
\end{definition}

\begin{theorem}
	\thlabel{thm:complete_problems_u*_diamond}
	For every complete problem $f\pmfunction{X}{\Nb}$, $\ustar{f}\weiequiv f^\diamond$. 
\end{theorem}
\begin{proof}
	The right-to-left reduction follows from \thref{thm:ustar<=diamond}, hence we now show the left-to-right one. Since the unbounded finite parallelization is degree-theoretic, by \thref{sec:theo-red-diam}, it is enough to show that $\ustar{f}\compproduct f\weireducible \ustar{\ol f}$.

	Intuitively speaking, the proof is similar to the one of  \thref{thm:total-func-ustar-diam}: given a certain instance $(p,x )$ of $\ustar{ f}\compproduct f$, the forward functional $\Phi$ of the reduction uses the first position of the input to $\ustar{f}$ to store $x$, whereas the other positions will contain strings corresponding to a guess of what $f(x)$ may be. The difference is that now it is not c.e.\ to determine whether a guess is wrong. However, we can force the computation to produce valid instances of $\ol f$ even when we are guessing a wrong solution for $f(x)$. Outputs corresponding to wrong guesses will not be used in the computation of a solution for $\ustar{f}\compproduct f$. Without loss of generality, we prove the theorem for $X=\Baire$. 

	Let $\Phi_C,\Psi_C$ be a pair of functionals witnessing the reduction $f\strongweireducible \ol f$. The forward functional $\Phi$ is the map that, given an input $(p,x)$ for $\ustar{f}\compproduct f$, produces a pair $(w, \sequence{x_k}{k\in\mathbb{N}})$ defined in stages as follows: we first let $x_{0}:=\Phi_C(x)$. To define $x_k$ for $k>0$, recall that, for every $n\in f(x)$, $\Phi_p(\str{n}\concat 0^\mathbb{N}) = \pairing{w^n, \pairing{q^n_i}_{i\in\mathbb{N}}}$ for some $(w^n,\sequence{q^n_i}{i\in\mathbb{N}})\in \dom(\ustar{f})$. Observe however that there is a computable functional $\Delta$ s.t., for every $n$, $\Delta(p, \str{n}\concat 0^\mathbb{N}) = \pairing{\ol{w^n}, \pairing{\ol{q^n_i}}_{i\in\mathbb{N}}}$, where $\ol{w^n}, \ol{q^n_i} \in \Baire$ and, if $n\in f(x)$, then for every $i$, $\repmap{\ol \Baire}(\ol{q^n_i}) = q^n_i$ (intuitively, this can be done simulating $\Phi_C(\Phi_p(q))$ and interleaving the output with zeroes to guarantee that we produce an infinite string). We then define $x_{\pairing{n,i}+1}:=\ol {q^{n}_{i}}$.
	
	We also let $w\in\Baire$ be an index for the continuous functional that works as follows: given in input $\pairing{y_i}_{i<k}$, let $m_{0}:=\Psi_C (y_{0})(0)$. We run $\Phi_p(\str{m_{0}}\concat 0^\mathbb{N})$, so to obtain (a name for) a pair $(v,\sequence{z_k}{k\in\mathbb{N}})$. The functional $\Phi_w$ returns $\str{m_{0}}\concat \Phi_v(\pairing{m_{i+1}}_{i<h})$ where $m_{i+1}:=\Psi_C(y_{\pairing{m_{0}, i} +1} )(0)$ and $h$ is largest so that, for every $i<h$, $\pairing{m_0, i}+1<k$. Clearly, the convergence of $\Phi_w(\pairing{y_i}_{i<k})$ depends on the convergence of $\Phi_v(\pairing{m_{i+1}}_{i<h})$.
 	The backward functional $\Psi$ simply outputs $\Psi_C(y_0)$ and $\Psi_C(y_{\pairing{m_0,i}+1})$ for sufficiently many $i$.
	
	It is a definition-chasing exercise to verify that, for every $(p,x)\in \dom(\ustar{f}\compproduct f)$, $\Phi$ produces a valid input for $\ustar{\ol f}$. Indeed, for every $k\in\mathbb{N}$, clearly $x_k\in \dom(\ol f)$. Moreover, the functional $\Phi_w$ essentially just looks at the columns corresponding to $m_0$ which, by construction, is a valid solution to $f(x)$. The claim therefore follows from the fact that $\Phi_p$ produces a valid input for $\ustar{f}$ when executed on valid solution for $f(x)$. 
	
	This also implies that $\Psi$ correctly computes a valid solution for the compositional product, hence concluding the proof.
\end{proof}

Notice that the completion of a first-order problem is not necessarily Weihrauch-equivalent to a first-order problem. In other words, the previous theorem does not imply that, for every first-order problem $f$, $\ustar{\completion{f}}\weiequiv\completion{f}^\diamond$.

\begin{proposition}
	$\ustar{\completion{\CNatural}}\strictlyweireducible \completion{\CNatural}^\diamond$, and therefore $\completion{\CNatural}$ is not Weihrauch-equivalent to any first-order problem.
\end{proposition}
\begin{proof}
	The reduction follows trivially from \thref{thm:ustar<=diamond}. To prove the separation, observe that $\completion{\CNatural}\weireducible\mflim$ (\cite[Prop.\ 8.2]{BGCompOfChoice19}). Hence $\ustar{\completion{\CNatural}} \weireducible \mflim$ (\thref{thm:u*<=parallelization}) and, in particular, $\Choice{2}'\not\weireducible\ustar{\completion{\CNatural}}$. 

	We show that $\Choice{2}'\weireducible \completion{\CNatural}\compproduct\completion{\CNatural} \weireducible \completion{\CNatural}^\diamond$ instead, therefore proving the separation. We think of an input for $\Choice{2}'$ as a sequence $x\in\Cantor$, and $\Choice{2}'(x)=\{ b<2  \st (\forall i)(\exists j>i)(x(j)=b) \}$. 

	For every $x\in \Cantor$, we first define the set 
	\[ A_x:= \{ n\in \mathbb{N} \st (\forall k>n)(x(k)=x(n)) \}. \]
	Applying $\completion{\CNatural}$ to $A_x$ we obtain some $p\in \Baire$ s.t.\ if $A_x\neq \emptyset$ then $p$ is a $\repmap{\completion{\mathbb{N}}}$-name for some $n\in A_x$, otherwise $p$ is just an arbitrary string.

	Since $\LPO\weireducible \completion{\CNatural}$, we can use $\completion{\CNatural}$ to check whether $(\exists i)(p(i)\neq 0)$. Observe that the existence of such $i$ does not guarantee that $p$ is a valid name for some $n\in A_x$ (valid names need to have infinitely many non-$0$ coordinates). If there is no such $i$, we return $0$, otherwise we let $m$ be the least $i$ s.t.\ $p(i)\neq 0$ and return $x(p(m)-1)$.

	We now show that this procedure computes an element of $\Choice{2}'(x)$ using $\completion{\CNatural}$ twice. Observe that if $A_x\neq \emptyset$ then $p$ is a $\repmap{\completion{\mathbb{N}}}$-name for $n\in A_x$. This implies that there is a least $m$ s.t.\ $p(m)-1=n$, and therefore the procedure returns a correct solution. Otherwise, if $A_x=\emptyset$ then $\Choice{2}'(x)=\{0,1\}$, and therefore any answer is correct (so we only need to be sure that the procedure returns some $b<2$, which is easily checked).

	Finally observe that, if $\completion{\CNatural}$ were a first-order problem then, by \thref{thm:complete_problems_u*_diamond}, $\ustar{\completion{\CNatural}} \weiequiv \completion{\CNatural}^\diamond$, against $\ustar{\completion{\CNatural}}\strictlyweireducible \completion{\CNatural}^\diamond$.
\end{proof}

\section{The first-order part of known problems}
\label{sec:applications}

In this section, we apply the previously obtained results to characterize the first-order part of several problems that are well-known in the literature. To this end, we also prove some theorems providing sufficient conditions for proving that $\ustar{f}\not\weireducible f^*$.

We start by characterizing the first-order part of $\WKL$. As already mentioned, the equivalence $\firstOrderPart{\WKL}\weiequiv \Choice{2}^*$ was already proved in \cite{DSYFirstOrder}. However, the technique used is ad-hoc, while the same result follows from a more general argument.

\begin{proposition}
	\thlabel{thm:fop_wkl}
	$\firstOrderPart{\WKL} \strongweiequiv \Choice{2}^\diamond \strongweiequiv \ustar{\Choice{2}} \strongweiequiv \Choice{2}^*$.
\end{proposition}
\begin{proof}
	The equivalence $\firstOrderPart{\WKL} \weiequiv \ustar{\Choice{2}}$ is just an application of \thref{thm:FOP(parallelization)=u*(FOP)}, whereas $\Choice{2}^\diamond \weiequiv \ustar{\Choice{2}} \weiequiv \Choice{2}^*$ follows from \thref{thm:star=diamond} and the fact that $\Choice{2}\weiequiv \totalization{\Choice{2}}$ (\cite[Prop.\ 6.3]{BGCompOfChoice19}).
	To complete the proof it is enough to show that $\Choice{2}^*$ is a first-order cylinder, as $\Choice{2}^* \strongweireducible \ustar{\Choice{2}} \strongweireducible \Choice{2}^\diamond$ is trivial, and if $\Choice{2}^*$ is a first-order cylinder then both the equivalences $\firstOrderPart{\WKL} \weiequiv \Choice{2}^*$ and $\Choice{2}^\diamond \weiequiv \Choice{2}^*$ are strong.
	
	Let $g$ be a first-order problem s.t.\ $g\weireducible \Choice{2}^*$ via $\Phi, \Psi$. For every name $z$ of some input for $g$, we can uniformly compute a bound $b\in\mathbb{N}$ s.t., letting $\sequence{A_i^z}{i<k}$ be the input for $\Choice{2}^*$ named by $\Phi(z)$, for every name $y$ of a solution for $\Choice{2}^*(\sequence{A_i^z}{i<k})$, we have 
	\[ \Psi(z, y)(0) =  \Psi(z[b],y)(0).\]
	This follows from the fact that, for every finite string $\sigma \in 2^k$, either $\Psi(z,\sigma)(0)\downarrow$ (and hence it converges after reading a finite prefix of $z$) or $\Psi(z,\sigma)(0)\downarrow$. In the latter case, there is $i<k$ s.t.\ $\sigma(i)\notin \Choice{2}(A_i)$, which is a c.e.\ condition (relatively to $A_i$). This is the same argument exploited in the proof of \thref{thm:star=diamond}.	Since $\id_{\baire}\strongweireducible\Choice{2}^*$ and $\Choice{2}^*\times \Choice{2}^*\strongweiequiv \Choice{2}^*$ we have $g\strongweireducible \Choice{2}^*$.
\end{proof}
		
\begin{theorem}
	\thlabel{thm:fop_lim}
	For every $n\in\mathbb{N}$, $\firstOrderPart{(\mflim^{(n)}) } \strongweiequiv \CNatural^{(n)} \strongweiequiv (\CNatural^{(n)})^\diamond \strongweiequiv \ustar{(\LPO^{(n)})} \strongweiequiv (\LPO^{(n)})^\diamond. $
\end{theorem}
\begin{proof}
    	The fact that $\mflim^{(n)}\weiequiv \parallelization{\LPO^{(n)}} \weiequiv \parallelization{\CNatural^{(n)}}$ is well-known in the literature, and can easily be proved using $\mflim \strongweiequiv \parallelization{\LPO} \strongweiequiv \parallelization{\CNatural}$ (\cite[Thm.\ 6.7]{BGP17} and \cite[Cor.\ 3.11]{BG11}) and $(\parallelization{f})' \strongweiequiv\parallelization{(f')}$ \cite[Prop.\ 5.7(3)]{BolWei11}. Applying \thref{thm:FOP(parallelization)=u*(FOP)_strong_cor}, we immediately obtain 
    	\[ \firstOrderPart{(\mflim^{(n)}) } \strongweiequiv (\ustar{\CNatural})^{(n)}\strongweiequiv \ustar{(\CNatural^{(n)})} \strongweiequiv (\ustar{\LPO})^{(n)}\strongweiequiv \ustar{(\LPO^{(n)})}.\] 
    	    	
    	We now prove that $(\ustar{\CNatural})^{(n)} \strongweiequiv \CNatural^{(n)}$. Observe first that $\ustar{\CNatural} \weireducible \CNatural$: indeed, a solution for $\ustar{\CNatural}(w,\sequence{A_n}{n\in\mathbb{N}})$ can be easily obtained choosing an element from the set
    	\[ \{ \sigma \in \baire \st (\forall i<\length{\sigma})( \sigma(i) \in A_i) \land \Phi_w(\sigma)(0)\downarrow \text{ in } \length{\sigma} \text{ steps} \}. \]
    	Since $\CNatural$ is a first-order cylinder (\thref{thm:fo_cylinder}), we obtain $\ustar{\CNatural} \strongweireducible \CNatural$. Moreover, given that $\CNatural \strongweireducible \ustar{\CNatural}$ (\thref{thm:u*<=parallelization}) and that the jump is strong-degree theoretic, the claim follows.

        To conclude the proof, observe that the reductions $\CNatural^{(n)} \strongweireducible (\CNatural^{(n)})^\diamond$ and $\ustar{(\LPO^{(n)})} \strongweireducible (\LPO^{(n)})^\diamond$ are straightforward, while the converse reductions $(\CNatural^{(n)})^\diamond \strongweireducible \CNatural^{(n)}$ and $(\LPO^{(n)})^\diamond\strongweireducible \ustar{(\LPO^{(n)})}$ follow from the facts that the diamond of a first-order problem is first-order and that both $\CNatural^{(n)}$ and $\ustar{(\LPO^{(n)})}$ are first-order cylinders (\thref{thm:fo_cylinder}).
\end{proof}

\begin{corollary}
	\thlabel{thm:fop_CR}
	$\firstOrderPart{\Choice{\mathbb{R}}} \weiequiv \CNatural$.
\end{corollary}
\begin{proof}
	Since $\firstOrderPart{\mflim} \weiequiv \CNatural$ (\thref{thm:fop_lim}), the statement follows easily from $\CNatural\weireducible\Choice{\mathbb{R}}\weireducible \mflim$ and the fact that the first-order part is a degree-theoretic operator.
\end{proof}

Observe that, as a corollary of \thref{thm:fop_lim}, we obtain $\CNatural\weiequiv\CNatural^*\weiequiv \ustar{\CNatural}$. This is not the case for $\LPO$: indeed $\LPO\strictlyweireducible \LPO^* \strictlyweireducible \ustar{\LPO}$. The fact that the reduction $ \LPO^* \strictlyweireducible \ustar{\LPO}$ is strict follows from the fact that $\ustar{\LPO}$ can compute the problem ``given $A\in\boldfaceSigma^0_1(\mathbb{N})$, say if $A$ is empty and, if not, produce its minimum''. The same problem cannot be solved by $\LPO^*$. 

More generally, the following \thref{thm:f*<fu*} provides a sufficient condition for $f^* \strictlyweireducible \ustar{f}$. In particular, it can be used to show that, for every $n>0$, $(\LPO^{(n)})^* \strictlyweireducible \ustar{(\LPO^{(n)})}$. 

We first prove the following technical lemma.

\begin{lemma}
	\thlabel{thm:LPO_fixed_point}
	Let $f$ be a multi-valued function s.t.\ $\LPO\weireducible f$. Assume for simplicity that $f$ has codomain $\Baire$. For every $w\in\Baire$ we can uniformly compute $x\in\dom(f)$ and a computable functional $\Psi_\Sigma$ s.t.\ 
	\begin{gather*}
		\Phi_w(x)(0)\downarrow \iff (\forall y \in f(x))(\Psi_\Sigma(y)(0)=1),\\
		\Phi_w(x)(0)\uparrow \iff (\forall y \in f(x))(\Psi_\Sigma(y)(0)=0).
	\end{gather*}
\end{lemma}
\begin{proof}
	Assume $\LPO\weireducible f$ via $\Phi$, $\Psi$ and let $z= \Phi(0^\mathbb{N})$. Define $p\in \Cantor$ as follows: we search for the least $s\in\mathbb{N}$ s.t.\ $\Phi_w(z[s])(0)\downarrow$ in $s$ steps. If we find such $s$, we let $p:=	0^t\concat 1^\mathbb{N}$ for some $t$ sufficiently large so that $z[s]\prefix \Phi(0^t\concat 1^\mathbb{N})$. Otherwise, we let $p:=0^\mathbb{N}$.

	It is clear that $p$ is uniformly computable from $w$: we just need to simulate $\Phi_w(z)$. If the computation never halts we keep concatenating zeros to obtain $0^\mathbb{N}$, otherwise we can compute a sufficiently large $t$ that satisfies the second condition.
	
	Let $x:=\Phi(p)$ and $\Psi_\Sigma:=y\mapsto \Psi(p,y)$. Notice that if $\Phi_w(z)(0)\uparrow$ then $x=z$ and therefore $\Phi_w(x)(0)\uparrow$. Moreover, for every $y\in f(x)$, $\Psi_\Sigma(y)(0)=\Psi(0^\mathbb{N}, y)(0)=0$, as $\Psi$ is the backward functional of a Weihrauch reduction. Similarly, if $\Phi_w(z)(0)\downarrow$ then, by the continuity of $\Phi_w$, $\Phi_w(x)(0)\downarrow$. In this case, for every solution $y$ of $f(x)$ we obtain $\LPO(p)=1=\Psi_\Sigma(y)(0)$. 
\end{proof}

\begin{theorem}
	\thlabel{thm:f*<fu*}
	Let $f\pmfunction{\Baire}{k}$ be s.t.\ $\LPO\weireducible f$. If there is $x\in\dom(f)$ s.t.\ 
	\[ (\forall i\in f(x))(\forall n\in\mathbb{N})(\exists z\in\dom(f))( z[n]\prefix x \land i\notin f(z))  \] 
	then $f^*\strictlyweireducible \ustar{f}$.
\end{theorem}
Before formally proving the theorem, let us sketch the idea of the proof: intuitively, the reason why the reduction fails is that any forward functional witnessing $\ustar{f}\weireducible f^*$ eventually commits to some (finite) number $N$ of instances of $f$ to solve in parallel. Since $f$ has codomain $k$, $N$ instances of $f$ can only result in $k^N$ possible different solutions. Since $\ustar{f}$ does not have such a bound, we can diagonalize against all $k^N$ possible different solutions. However, since the input of $\ustar{f}$ includes the index $w$ of a functional giving the condition on the number of columns in the output of $\ustar{f}$, we need to be sure that we can pick a suitable input for $\ustar{f}$ with ``more columns'' than what the hypothetical reduction $\ustar{f}\weireducible f^*$ would require. This cannot be done, for example, whenever $f$ is parallelizable.
\begin{proof}
	Assume that there is a reduction $\ustar{f}\weireducible f^*$ witnessed by the functionals $\Phi$, $\Psi$. 
	Let also $x\in\dom(f)$ be as in the hypotheses. 
	
	By \thref{thm:LPO_fixed_point}, given $w$ and $x$ we can uniformly compute $x_0\in\dom(f)$ and an index for the computable functional $\Psi_\Sigma$ s.t.
	\begin{gather*}
		\Phi(w,\pairing{x_0,x,\hdots})(0)\downarrow \iff (\forall y \in f(x_0))(\Psi_\Sigma(y)(0)=1),\\
		\Phi(w,\pairing{x_0,x,\hdots})(0)\uparrow \iff (\forall y \in f(x_0))(\Psi_\Sigma(y)(0)=0),
	\end{gather*}
	where $\pairing{x_0,x,\hdots}$ is the join of countably many strings, the first one being $x_0$ and all other ones being $x$.
	Consider the functional $F\function{\Baire}{\Baire}$ defined as $F(w):=v$, where $v\in \Baire$ is s.t.\ $\Phi_v(\pairing{y_i}_{i<n})$ works as follows: it first simulates $\Psi_\Sigma(y_0)$ until it converges in $0$ (if this never happens then $\Phi_v(\pairing{y_i}_{i<n})\uparrow$). Then 
	\begin{itemize}
		\item if $\Psi_\Sigma(y_0)(0)=0$ then it immediately halts and returns $0$;
		\item if $\Psi_\Sigma(y_0)(0)>0$ then it waits until $\Phi(w,\pairing{x_0,x,\hdots})(0)$ commits to some $m$, and halts iff $n > k^m+1$.		
	\end{itemize}
	Clearly $F$ is total and continuous (in fact it is $x$-computable), hence, by the recursion theorem (\thref{thm:continuous_recursion_theorem}), there is $w$ s.t.\ $\Phi_{F(w)}=\Phi_w$. 
	
	Consider the pair $(w,\mathbf{x})$ where $\mathbf{x}:=(x_0,x,\hdots)$, and $x_0$ is obtained applying \thref{thm:LPO_fixed_point} to the fixed point $w$. We first show that it is a valid input for $\ustar{f}$. Clearly $\mathbf{x}\in\dom(\parallelization{f})$. If $\sequence{y_i}{i\in\mathbb{N}}$ is a sequence of $f$-solutions for $\mathbf{x}$, then $\Phi_w(\pairing{y_i}_{i<n})$ works as follows: it first simulates $\Psi_\Sigma(y_0)$. By construction, $\Psi_\Sigma(y_0)(0)\downarrow = b$. If $b=0$ then $\Phi_w(\pairing{y_i}_{i<n})$ halts immediately, otherwise it simulates $\Phi(w,\pairing{x_0,x,\hdots})$. However, again by construction, $b=1$ implies that $\Phi(w,\pairing{x_0,x,\hdots})(0)\downarrow = m$, and therefore $\Phi_w(\pairing{y_i}_{i<n})$ halts whenever $n$ is sufficiently large. This proves that $(w,\mathbf{x})$ is a valid input for $\ustar{f}$. 
	
	Let $s_0$ be sufficiently large so that, upon input $(w,\mathbf{x})$, the forward functional $\Phi$ commits to some $m$ after $s_0$ steps. In particular, since $f$ has codomain $k$, there are only $k^m$ possible answers for $f^*(\Phi(w,\mathbf{x}))$. Recall also that, by construction, a valid answer for $\ustar{f}(w,\mathbf{x})$ is a (finite) sequence $\sequence{y_i}{i<n}$ with $n>k^m+1$. We iteratively diagonalize against every possible outcome of $f^*(\Phi(w,\mathbf{x}))$. We proceed as follows: for the sake of readability, let $\mathbf{z}_0:=\mathbf{x}$. At each stage $t<k^m$, $\mathbf{z}_t$ is of the form $(x_0,z_0,\hdots, z_{t-1},x,\hdots)$. We check if there is an unmarked finite sequence $b_0,\hdots,b_{m-1}$, with $b_i<k$, s.t.\ 
	\[ \Psi(w,\mathbf{z}_t, \pairing{b_i}_{i<m}) \]
	produces a (name for a) valid solution $\sequence{y^t_i}{i<n}$ of $\ustar{f}(w,\mathbf{z}_t)$ (with a small abuse of notation, we are identifying $b_i$ with its name).	If there is none then we are done. Otherwise, by continuity, there is $s\in\mathbb{N}$ s.t.\ $ \Psi(w,\mathbf{z}_t, \pairing{b_i}_{i<m})$ produces such a solution in $s$ steps. 
	
	Let $s_{t+1}:= s_t + s + 1$ and choose $z_t \extends x[s_{t+1}]$ so that $y^t_t\notin f(z_t)$. We then define $\mathbf{z}_{t+1}:=(x_0,z_0,\hdots,z_t, x,\hdots)$, label the sequence $\sequence{b_i}{i<m}$ as marked and go the next stage. Notice that, at each stage, $(w,\mathbf{z}_t)$ is still a valid input for $\ustar{f}$.  
	
	After finitely many stages ($k^m$ in the worst case), we find a valid input for $\ustar{f}$ for which the forward functional $\Phi$ produces (a name for) $m$ instances of $f$ but such that the backward functional cannot compute a correct solution, against the fact that $\Phi$, $\Psi$ witness a Weihrauch reduction.
\end{proof}

\begin{corollary}\thlabel{cor:fop-wkl}
	For every $k\ge 2, n>0$, $(\Choice{k}^{(n)})^* \strictlystrongweireducible \ustar{(\Choice{k}^{(n)})}$, and therefore  
	\[ (\Choice{2}^{(n)})^* \strictlystrongweireducible (\Choice{2}^*)^{(n)}\strongweiequiv \ustar{(\Choice{2}^{(n)})} \strongweiequiv (\ustar{\Choice{2}})^{(n)}  \strongweiequiv\firstOrderPart{(\WKL^{(n)})}.\]
\end{corollary}
\begin{proof}
	Fix $k\ge 2$ and $n>0$. The strict reduction $(\Choice{k}^{(n)})^* \strictlystrongweireducible \ustar{(\Choice{k}^{(n)})}$ is a simple application of \thref{thm:f*<fu*}. Moreover, since $\WKL^{(n)}\strongweiequiv \parallelization{\Choice{2}^{(n)}}$ and $\parallelization{\Choice{2}}$ is a cylinder, applying \thref{thm:FOP(parallelization)=u*(FOP)_strong_cor} we obtain 
	\[ \ustar{(\Choice{2}^{(n)})} \strongweiequiv (\ustar{\Choice{2}})^{(n)}  \strongweiequiv\firstOrderPart{(\WKL^{(n)})}. \]

	To conclude the proof it is enough to show that, for every $n\ge 0$, $(\Choice{2}^*)^{(n)}\strongweiequiv (\ustar{\Choice{2}})^{(n)}$. This can be proved by induction on $n$: the base step was already proved in \thref{thm:fop_wkl}. If the claim holds up to $n$, then 
	\[ (\Choice{2}^*)^{(n+1)} \strongweiequiv  ((\Choice{2}^*)^{(n)})' \strongweiequiv ((\ustar{\Choice{2}})^{(n)})' \strongweiequiv (\ustar{\Choice{2}})^{(n+1)}. \qedhere \]
\end{proof}

\begin{corollary}${}$
	\begin{enumerate}
		\item $(\codedChoice{\boldfaceSigma^1_1}{}{2})^*\strictlyweireducible \ustar{(\codedChoice{\boldfaceSigma^1_1}{}{2} ) }\weiequiv (\codedChoice{\boldfaceSigma^1_1}{}{2})^\diamond \weiequiv \firstOrderPart{\SigmaWKL} $;
		\item $(\codedChoice{\boldfacePi^1_1}{}{2})^* \strictlyweireducible \ustar{(\codedChoice{\boldfacePi^1_1}{}{2})} \weiequiv (\codedChoice{\boldfacePi^1_1}{}{2})^\diamond \weiequiv \firstOrderPart{\UCBaire}$.
	\end{enumerate}
\end{corollary}
\begin{proof}
	\begin{enumerate}
		\item : The first strict reduction is an application of \thref{thm:f*<fu*}, whereas the equivalence $\ustar{(\codedChoice{\boldfaceSigma^1_1}{}{2} ) }\weiequiv \firstOrderPart{\SigmaWKL} $ follows from $\SigmaWKL \weiequiv \parallelization{\codedChoice{\boldfaceSigma^1_1}{}{2}}$ (\cite[Lem.\ 4.6]{KMP20}) using \thref{thm:FOP(parallelization)=u*(FOP)}. Finally, the reduction $(\codedChoice{\boldfaceSigma^1_1}{}{2})^\diamond \weireducible \ustar{(\codedChoice{\boldfaceSigma^1_1}{}{2})}$ follows from the fact that $\SigmaWKL$ is closed under compositional product (\cite[Prop.\ 4.8]{KMP20}).
		\item : Similarly to $(1)$, the first strict reduction follows from \thref{thm:f*<fu*}, while the equivalence $\ustar{(\codedChoice{\boldfacePi^1_1}{}{2})} \weiequiv \firstOrderPart{\UCBaire}$ follows from \thref{thm:FOP(parallelization)=u*(FOP)} and \cite[Thm.\ 3.11]{KMP20} (the principle $\parallelization{\codedChoice{\boldfacePi^1_1}{}{2}}$ corresponds the principle named $\boldfaceSigma^1_1\text{-}\mathsf{Sep}$ in \cite{KMP20}). Finally, the reduction $(\codedChoice{\boldfacePi^1_1}{}{2})^\diamond \weireducible\ustar{(\codedChoice{\boldfacePi^1_1}{}{2})}$ follows from the fact that $\UCBaire$ is closed under compositional product. \qedhere
	\end{enumerate}
\end{proof}
		
\begin{corollary}
	$\chiPi^* \strictlyweireducible \ustar{ \chiPi}  \weiequiv  \chiPi^\diamond \weiequiv \firstOrderPart{\PiCA}$.
\end{corollary}
\begin{proof}
	The fact that $\chiPi^*\strictlyweireducible \ustar{\chiPi }$ is a straightforward application of \thref{thm:f*<fu*}. The equivalence $\ustar{\chiPi }\weiequiv \firstOrderPart{\PiCA}$ follows from \thref{thm:FOP(parallelization)=u*(FOP)}. Finally, the equivalence $\ustar{\chiPi} \weiequiv \chiPi^\diamond$ follows from \thref{thm:complete_problems_u*_diamond}, as $\chiPi$ is complete \cite[Cor.\ 11.3(1)]{BGCompOfChoice19}.
\end{proof}

\thref{thm:f*<fu*} is far from being a characterization of the problems $f$ s.t.\ $f^*\strictlyweireducible \ustar{f}$. Another sufficient condition is given by the following theorem. 

\begin{definition}[{\cite[Def.\ 4.10]{BGP17}}]
	A problem $f\pmfunction{\Baire}{\Baire}$ is called \textdef{fractal} if, for every clopen $A\subset \Baire$ with $\dom(f)\cap A \neq \emptyset$,  $f\restrict{A} \weiequiv f$.
\end{definition}

\begin{theorem}
	\thlabel{thm:*<u*_fractals}
	For every first-order pointed fractal $f$ s.t.\ $\LPO\weireducible f$ and, for every $n\in\mathbb{N}$, $f^{n+1}\not\weireducible f^{n}$, we have $f^*\strictlyweireducible \ustar{f}$.
\end{theorem}
\begin{proof}
	Assume there is a reduction $\ustar{f}\weireducible f^*$ as witnessed by the functionals $\Phi,\Psi$. Fix a computable point $x\in\dom(f)$. The proof is similar to the one of \thref{thm:f*<fu*}, the difference is that now we use the fractality of $f$ to show that a reduction $\ustar{f}\weireducible f^*$ would yield a reduction $f^{n+1}\weireducible f^n$ for some $n$.
	
	Using \thref{thm:LPO_fixed_point} we know that, given $w$ and $x$ we can uniformly compute $x_0\in\dom(f)$ and an index for the computable functional $\Psi_\Sigma$ s.t.
	\begin{gather*}
		\Phi(w,\pairing{x_0,x,\hdots})(0)\downarrow \iff (\forall y \in f(x_0))(\Psi_\Sigma(y)(0)=1),\\
		\Phi(w,\pairing{x_0,x,\hdots})(0)\uparrow \iff (\forall y \in f(x_0))(\Psi_\Sigma(y)(0)=0),
	\end{gather*}
	where $\pairing{x_0,x,\hdots}$ is the join of countably many strings, the first one being $x_0$ and all other ones being $x$. 
	
	Consider the functional $F\function{\Baire}{\Baire}$ defined as $F(w):=v$, where $v\in \Baire$ is s.t.\ $\Phi_v(\pairing{y_i}_{i<k})$ works as follows: it first simulates $\Psi_\Sigma(y_0)$ until it converges in $0$ (if this never happens then $\Phi_v(\pairing{y_i}_{i<k})\uparrow$). Then 
	\begin{itemize}
		\item if $\Psi_\Sigma(y_0)(0)=0$ then it immediately halts and returns $0$;
		\item if $\Psi_\Sigma(y_0)(0)>0$ then wait until $\Phi(w,\pairing{x_0,x,\hdots})(0)$ commits to some $n$, and halts iff $k>n+1$. 
	\end{itemize}
	Clearly $F$ is total and computable, hence, by the recursion theorem (\thref{thm:continuous_recursion_theorem}), there is a computable $w\in\Baire$ s.t.\ $\Phi_{F(w)}=\Phi_w$. 

	Consider the pair $(w,\mathbf{x})$ where $\mathbf{x}:=(x_0,x,\hdots)$, and $x_0$ is obtained applying \thref{thm:LPO_fixed_point} to the fixed point $w$. Is it not hard to see that $(w,\mathbf{x})\in\dom(\ustar{f})$ (again, see the proof of \thref{thm:f*<fu*}). Let $\sigma$ be a sufficiently long prefix of the name of $(w,\mathbf{x})$ s.t.\ $\Phi(\sigma)$ commits to a number $n$ of instances of $f$ needed to solve $\ustar{f}(w,\mathbf{x})$. By the continuity of $\Phi$, every instance of $\ustar{f}$ whose name extends $\sigma$ can be solved with $f^n$. Since $\sigma$ is a finite string, it only commits to a finite prefix on a finite number of columns of the input. More precisely, we think of $\sigma$ as the prefix of an infinite string obtained by joining countably many strings $p_i$. In particular, let $\tau_i$ be the prefix of $p_i$ contained in $\sigma$, and let $h$ be s.t.\ for every $i\ge h$, $\tau_i=\str{}$. 
	
	Since $f$ is a fractal, for every $0<i<h$ we have $f \weiequiv f\restrict{\tau_i}$, where $f\restrict{\tau_i}$ indicates the restriction of $f$ to inputs with names that begin with $\tau_i$. Let $\varphi_i, \psi_i$ be the forward and backward functionals of the reduction $f\weireducible f\restrict{\tau_i}$.
		
	By construction, for every $\sequence{z_i}{i<n+1}\in\dom(f^{n+1})$, a solution for $f^{n+1}(\sequence{z_i}{i<n+1})$ can be uniformly computed from a solution to $\ustar{f}(w,\str{x_0,\varphi_0(z_0),\hdots, \varphi_n(z_n),x,\hdots})$: just take the first $n+2$ columns, ignore the first one, and apply the correspondent $\psi_i$. This shows that a reduction $\ustar{f}\weireducible f^*$ yields a reduction $f^{n+1}\weireducible f^n$, against the hypotheses.
\end{proof}

We can use \thref{thm:*<u*_fractals} to show e.g.\ that $\tcn^*\strictlyweireducible \ustar{\tcn}$. In the following \thref{thm:tcn^u*=cn'}, we give a better characterization of the degree of $\ustar{\tcn}$. We first prove the following lemma.

\begin{lemma}
	\thlabel{thm:tcn^n<tcn^n+1}
	For every $n\in\mathbb{N}$, $\tcn^{n+1} \not \weireducible \tcn^n$. 
\end{lemma}
\begin{proof}
	Fix $n>0$ (the case $n=0$ is trivial) and assume that the reduction is witnessed by the maps $\Phi, \Psi$. In the following, for the sake of readability, we identify a closed subset of $\mathbb{N}$ with its name. 

	We want to define an input $\seq{p}=\sequence{p_i}{i<n+1}$ so to diagonalize against $\Phi,\Psi$. For each stage $s$, we define an input $\seq{p}^s=\sequence{p^s_i}{i<n+1}$ for $\tcn^{n+1}$ and a placeholder $t_s$. At stage 0, we first check if there is a sufficiently large $t_0$ so that 
	\[ \Psi(0^{t_0},\hdots,0^{t_0}, 0,\hdots,0)(0)\downarrow = \pairing{a^{(0)}_i}_{i<n+1}, \]
	for some $a^{(0)}_0,\hdots,a^{(0)}_n\in\mathbb{N}$, where $\Psi$ is executed with $(n+1)$-many $0^{t_0}$ and $n$-many $0$. If there is such a $t_0$ then, for every $i<n+1$, define $p^0_i:=0^{t_0}\concat \str{a^{(0)}_i+1}\concat 0^\mathbb{N}$, otherwise define $p^0_i:=0^\mathbb{N}$. We can already define $p^s_n:=p^0_n$ for every $s$.
	
	At stage $s+1$, let $\seq{q}^{s}=\sequence{q^s_i}{i<n}:=\Phi(\seq{p}^s)$ be an input for $\tcn^n$. Let $\seq{b}^s=\sequence{b^s_i}{i<n}$ be the ``least solution of $\tcn^n(\seq{q}^s)$", i.e.\ for every $i<n$, $b^s_i+1$ is the least positive number not enumerated by $q^s_i$ (if $q^s_i$ enumerates every positive number, we let $b^s_i:=b^{s-1}_i$ if $s> 0$, and $0$ otherwise). If, for every $i<n$, $b^{s-1}_i=b^{s}_i$ (i.e.\ if all the minima remained unchainged), then the input $\seq{p}^s$ witness the fact that $\Phi,\Psi$ are not a valid reduction, hence we can stop the construction (i.e.\ define $p^r_i:=p^s_i$ for every $i<n$ and $r>s$). Otherwise, let $t_{s+1}>t_s$ be sufficiently large so that
	\begin{itemize}
		\item the prefix of $\seq{q}^s$ produced by $\Phi(\seq{p}^s[t_{s+1}])$ is long enough so that, for every $i<n$ and every $m<b^s_i$, there is $j$ s.t.\ $q^s_i(j)=m+1$;
		\item $\Psi(p^s_0[t_{s+1}],\hdots,p^s_{n}[t_{s+1}],b^s_0,\hdots,b^s_{n-1})(0)\downarrow=\pairing{a^{(s+1)}_i}_{i<n+1}$, for some $a^{(s+1)}_0,\hdots,a^{(s+1)}_{n}\in \mathbb{N}$.
	\end{itemize}
	For every $i<n$ define $p^{s+1}_i:=p^{s}_i[t_{s+1}]\concat \str{a^{(s+1)}_i+1}\concat 0^\mathbb{N}$ and go to the next stage. 

	For every $i<n+1$ we define $p_i:=\lim_{s\to\infty}p^s_i$, and let $\seq{q}=\sequence{q_i}{i<n}:=\Phi(\seq{p})$. The fact that the sequence $\sequence{t_s}{s}$ is strictly increasing guarantees that the limits are well-defined. 

	Observe that, if for some $i<n$, $b^s_i$ changes infinitely many times then, by the continuity of $\Phi$, $q_i$ is (a name for) the empty set. Let $J\subset n$ be s.t.\ for every $j\in J$, $b^s_j$ changes finitely many times. If $J\neq \emptyset$, then fix $s$ sufficiently large so that, for every $j\in J$ and every $r>s$, $b^s_j=b^r_j$. At stage $s+1$, we diagonalized against the possible solution $\str{b^s_0,\hdots, b^s_{n-1}}$. Since, for every $i\notin J$, $q_i$ is a name for the empty set, we have that $\str{b^s_0,\hdots, b^s_{n-1}}\in \tcn^n(\seq{q})$, against the fact the $\Phi, \Psi$ witness the reduction. 
	This implies that $J=\emptyset$, and hence, for every $i<n$, $q_i$ is the empty set. This implies that $\str{0,\hdots,0}\in\tcn(\seq{q})$. However, we diagonalized against $\str{0,\hdots,0}$ in stage $0$ (using the $(n+1)$-th instance $p_n$ of $\tcn^{n+1}$), hence we obtained again a contradiction with the fact that $\Phi,\Psi$ witness a Weihrauch reduction.
\end{proof}

Notice that, since $\tcn$ and its parallel products $\tcn^n$ are total fractals, combining \cite[Thm.\ 7.15]{BGP17} with the previous result we obtain, as a corollary, that for every $n$, $\tcn^{n+1}\not\weireducible \tcn^n\compproduct\CNatural$.

\begin{theorem}
	\thlabel{thm:tcn^u*=cn'}
	$\tcn \strictlyweireducible \tcn^* \strictlyweireducible \ustar{\tcn} \weiequiv \tcn^\diamond \weiequiv \CNatural'$.
\end{theorem}
\begin{proof}
	The fact that the first two reductions are strict are corollaries of \thref{thm:tcn^n<tcn^n+1} and \thref{thm:*<u*_fractals} respectively (as $\tcn$ is a pointed fractal). The reduction $\ustar{\tcn} \weireducible \tcn^\diamond$ is trivial (\thref{thm:ustar<=diamond}) and $\tcn^\diamond \weireducible \CNatural'$ follows from $\tcn \weireducible \CNatural'$ \cite[Cor.\ 8.14]{BGCompOfChoice19} and the fact that $\CNatural'$ is closed under diamond (\thref{thm:fop_lim}). To conclude the proof we show that $\CNatural'\weireducible \ustar{\tcn}$. Since $\CNatural'\weiequiv \BWT_\mathbb{N}$ \cite[Fact 2.3(1)]{BRramsey17}, it is enough to show that $\ustar{\tcn}$ suffices to find a number that appears infinitely many often in a given sequence $x\in\Baire$ (provided that such a number exists). We first define a sequence of inputs for $\tcn$. For every $m\in\mathbb{N}$, let $A_m:=\{ a\in\mathbb{N} \st (\forall i>a)(x(i)\neq m) \}$. Moreover, for every $a,m$, define $B_{a,m}:=\{ b\in\mathbb{N} \st b>a \text{ and } x(b) = m\}$. Clearly all the sets $A_m$ and $B_{a,m}$ are closed, hence they can be arranged so to be an instance of $\parallelization{\tcn}$. Given $a_m\in\tcn(A_m)$ and $b_{a,m}\in \tcn(B_{a,m})$, we can compute a cluster point of the sequence as follows: for every $m$, we check if $x(b_{a_m,m})=m$. If yes, then $A_m=\emptyset$ and hence $m$ is a cluster point. If not, then $A_m\neq \emptyset$ and $m$ appears only finitely many times in $x$. The existence of a cluster point guarantees that it can be found by unbounded search using only finitely many columns of the output of $\parallelization{\tcn}$.
\end{proof}

\begin{corollary}
	$\totalization{(\mflim)} \strictlyweireducible \parallelization{\tcn}\weiequiv \mflim'$.
\end{corollary}
\begin{proof}
	It is not hard to check that $\totalization{(\mflim)}\weireducible \parallelization{\tcn}$. Indeed, it suffices to notice that, given a sequence $\sequence{p_n}{n\in\mathbb{N}}$ in $\Baire$, the $i$-th instance of $\tcn$ can be used to find the limit of $\sequence{p_n(i)}{n\in\mathbb{N}}$ if it exists. In other words, $\tcn \weiequiv \totalization{(\mflim_\mathbb{N})}$. The equivalence $\parallelization{\tcn} \weiequiv \mflim'$ follows from \thref{thm:tcn^u*=cn'}, as $\parallelization{\tcn} \weiequiv \parallelization{\CNatural'}$. Finally, $\parallelization{\tcn}\not\weireducible \totalization{(\mflim)}$ because $\mflim'$ has a computable input with solution $\emptyset''$, whereas every computable input for $\totalization{(\mflim)}$ has a solution computable in $\emptyset'$. 
\end{proof}

We now show that $\PiBound^*\strictlyweireducible \ustar{\PiBound}$. Unfortunately, $\PiBound$ does not have finite range, and therefore we cannot apply \thref{thm:f*<fu*} directly. Moreover, it is closed under product (as we show in the proof \thref{thm:pibound*<u*}), hence \thref{thm:*<u*_fractals} cannot be applied either. We will exploit instead the fact that $\PiBound$ is upwards closed (namely if $b\in \PiBound(A)$ then every $b'>b$ is a valid solution as well).

We first prove the following technical lemmas. Without loss of generality, we will assume that an input for $\PiBound$ is an initial segment of $\mathbb{N}$. This follows from the fact that if $X\in\dom(\PiBound)$ then $Y:=\{ n \st (\exists m \ge n)(m\in X)\}$ is a finite $\boldfacePi^1_1$ initial segment of $\mathbb{N}$ and $\PiBound(Y) = \PiBound(X)$. In particular, we will use the fact that if $X\in\dom(\PiBound)$ and $x\notin X$ then $x\in \PiBound(X)$.

\begin{lemma}
	\thlabel{thm:tcn<pibound}
	$\totalization{\CNatural} \weireducible \PiBound$. Moreover, the reduction is witnessed by total functionals. 
\end{lemma}
\begin{proof}
	We think of a name for a closed subset $X$ of $\mathbb{N}$ as a string $q\in\Baire$ s.t.\ $n\notin X$ iff $(\exists i)(q(i)=n+1)$. In particular, every $q\in\Baire$ can be seen as the name for some closed $X\in\boldfacePi^0_1(\mathbb{N})$. Consider the computable function $g\function{\Baire\times \mathbb{N}}{\mathbb{N}}$ that maps $q\in\Baire$ and $s\in\mathbb{N}$ to the smallest $n\in\mathbb{N}$ s.t.\ $(\forall i<s)(q(i)\neq n+1)$, i.e.\ $n$ is the smallest number not enumerated by $q$ by stage $s$. Given a name $q$ for $X\in \boldfacePi^0_1(\mathbb{N})$, we can uniformly compute a $\boldfacePi^1_1$ name for the set 
	\[ A:=\{ n \in\mathbb{N}\st (\exists x\in\mathbb{N})(x+1\notin \ran(q)) \land (\exists i,j>n)( g(q,i)\neq g(q,j)) \}. \]
	Notice that such a set is finite (it can be empty) and that, for every $b$ bounding $A$, $g(q,b)\in \totalization{\CNatural}(X)$. In particular, the maps $q\mapsto A$ and $g$ are total maps witnessing the reduction $\totalization{\CNatural} \weireducible \PiBound$.
\end{proof}
		
\begin{lemma}
	\thlabel{thm:fixed_point_tcn}
	Given a name for a continuous functional $D\function{\Baire}{\Baire}$ we can uniformly compute $p_A\in\Baire$ and (an index for) a computable function $\Gamma\function{\mathbb{N}}{\mathbb{N}}$ s.t.\ $p_A$ is a $\boldfacePi^1_1$ name for a finite set $A\subset \mathbb{N}$ (i.e.\ $A\in\dom(\PiBound)$) and, letting $X$ be the closed subset of $\mathbb{N}$ named by $D(p_A)$ we have 
	\[ X\neq \emptyset \Rightarrow (\forall b\notin A)(\Gamma(b)\in X).\]
\end{lemma}
\begin{proof}
	Let $\Phi_\Pi, \Psi_\Pi$ be two total maps witnessing the reduction $\totalization{\CNatural}\weireducible \PiBound$. At stage $0$, we define $z_0:=0^\mathbb{N}$ and $t_0:=0$.
	
	At stage $i+1$, we search for some sufficiently large $t_{i+1}>t_i$ s.t.
	\[ D(\Phi_\Pi(z_i[t_{i+1}]))(i)\downarrow = n_i.\]
	Clearly such $t_{i+1}$ exists, as every infinite string can be seen as a name for a valid input of $\totalization{\CNatural}$. We define
	\[ z_{i+1}:= z_i[t_{i+1}]\concat \str{n_i} \concat 0^\mathbb{N}.\]
	Let $z:=\lim_{i\to\infty} z_i$. Observe that the sequence $\sequence{z_i}{i\in\mathbb{N}}$ is convergent and its limit is uniformly computable (as for every $i$,  $z_{i+1}[i]=z_i[i]$), hence $z$ is well-defined. Define $p_A:=\Phi_\Pi(z)$ and $\Gamma:=b\mapsto \Psi_\Pi(p_A,b)(0)$. Let also $A\in \boldfacePi^1_1(\mathbb{N})$ be the set named by $p_A$.
	
	Let $X$ be the closed subset of $\mathbb{N}$ named by $D(p_A)$. Observe that $X$ is exactly the set named by $z$. If $n \notin X$ then there is $i$ s.t.\ $D(p_A)(i)=n+1$. By construction, $z(t_{i+1})=z_{i+1}(t_{i+1})=n+1$, hence,  for every $b\notin A$,  $\Gamma(b)=\Psi_\Pi(p_A,b)\neq n$ (because $\Psi_\Pi$ is the backward functional of a Weihrauch reduction).
\end{proof}

\begin{theorem}
	\thlabel{thm:pibound*<u*}
	$\PiBound \weiequiv \PiBound^* \strictlyweireducible \ustar{\PiBound} \weiequiv \firstOrderPart{\parallelization{\PiBound}}$.
\end{theorem}
\begin{proof}
	The equivalence $\PiBound \weiequiv \PiBound^*$ is straightforward: one direction is trivial, while for the other it is enough to notice that, for every $\sequence{A_n}{n<k} \in \dom(\PiBound^*)$ and every $j<k$, we have $\PiBound( \bigcup_{n<k} A_n) \in \PiBound(A_j)$. Moreover, the equivalence $ \ustar{\PiBound} \weiequiv \firstOrderPart{(\parallelization{\PiBound})}$ is an application of \thref{thm:FOP(parallelization)=u*(FOP)}, so we only need to show that $ \PiBound^* \strictlyweireducible \ustar{\PiBound}$.

	The reduction is trivial from \thref{thm:u*<=parallelization}. Assume towards a contradiction that there is a reduction $\ustar{\PiBound} \weireducible \PiBound$ as witnessed by the maps $\Phi,\Psi$. Let $\sequence{A_n}{n>0}$ be a computable input for $\parallelization{\PiBound}$ with no computable solution. The existence of such $\sequence{A_n}{n>0}$ follows e.g.\ from the fact that $\UCBaire \weireducible \parallelization{\PiBound}$ (see the proof of \cite[Prop.\ 5.21]{GPVDescSeq}). We want to define $w\in\Baire$ and $A_0\in\dom(\PiBound)$ so to obtain a valid input $(w,\seq{A}:=\sequence{A_n}{n\in\mathbb{N}})$ for $\ustar{\PiBound}$ that diagonalizes against $\Phi,\Psi$.
	
	Let $F\function{\Baire}{\Baire}$ be the total computable functional that works as follows: given $w\in\Baire$, we can (uniformly) compute an index $u\in\Baire$ for the functional $D_w\function{\Baire}{\Baire}$ that maps $A\in\Baire$ to (a name for) the closed set 
	\[ X_w:= \{ m\in \mathbb{N} \st (\forall b)( \Psi(w,\sequence{A, A_1, A_2, \hdots}{}, b)(0)\downarrow=\pairing{r_i}_{i<h} \Rightarrow h< m)\}. \] 	
	By \thref{thm:fixed_point_tcn}, we can uniformly compute an input $A_w$ for $\PiBound$ and an index for a computable function $\Gamma_w$ s.t.
	\[ X_w\neq\emptyset \Rightarrow (\forall b\notin A_w)(\Gamma_w(b)\in X_w). \]
	The functional $F$ maps $w$ to an index $v\in \Baire$ s.t.\ $\Phi_v(\pairing{b_i}_{i<k})$ halts iff $k> \Gamma_w(b_0)$. 
	
	Clearly $F$ is total and computable, hence, by the recursion theorem (\thref{thm:continuous_recursion_theorem}), it has a computable fixed point. We define $w$ as a fixed point for $F$ (i.e.\ $\Phi_{F(w)}= \Phi_{w}$) and $A_0:=A_w$ as the input for $\PiBound$ obtained applying \thref{thm:fixed_point_tcn} to the functional $D_w$.

	It is straightforward to show that $(w,\seq{A})\in \dom(\ustar{\PiBound})$. Indeed, $\seq{A}\in\dom(\parallelization{\PiBound})$ by definition, and $\Phi_{w}(\pairing{b_i}_{i<k})$ halts whenever $k$ is sufficiently large (recall that $\Gamma_w$ is always defined on solutions of $\PiBound(A_0)$).
	
	Let $B$ be the finite $\boldfacePi^1_1$ subset of $\mathbb{N}$ named by $\Phi(w, \seq{A})$. Observe that, since $(w,\seq{A})$ is a valid input for $\ustar{\PiBound}$, there is $m\in\mathbb{N}$ s.t.\ 
	\begin{equation*}
		(\forall b)( \Psi(w,\seq{A}, b)(0)\downarrow=\pairing{r_i}_{i<h} \Rightarrow h< m).\tag{$\star$}
	\end{equation*}
	Indeed, if $b$ is a bound for $B$ then $\Psi(w,\seq{A}, b)(0)\downarrow=\pairing{r_i}_{i<h}$ and for every $i<h$, $r_i\in A_i$. If $\limsup_{b\to\infty} \length{\Psi(w,\seq{A},b)(0)}= +\infty$ then, simulating $\Psi(w,\seq{A},b)$ for every $b>b_0$ (for some sufficiently large $b_0$), we could obtain a computable solution for $\parallelization{\PiBound}(\seq{A})$, which contradicts the hypotheses.
	
	However, any valid solution for $\ustar{\PiBound}(w,\seq{A})$ is s.t.\ its length is some $m$ satisfying $(\star)$, and therefore, for every $b\in\PiBound(B)$, $\Psi(w,\seq{A},b)$ cannot be a valid solution for $\PiBound(w,\seq{A})$, against the fact that $\Psi$ was the backward functional of a Weihrauch reduction.	
\end{proof}

\subsection{Ramsey's theorem}
		
In this section, we analyze the first-order part of $\SRT{n}{k}$ and $\RT{n}{k}$ (see Section~\ref{sec:background} for their definitions). As already mentioned, the first-order part of the principles $\SRT{1}{k}$ and $\RT{1}{k}$ was already characterized in \cite{BRramsey17}. 

\begin{proposition}
	For every $k\ge 2$ and for $k=\mathbb{N}$ we have $\mflim_k\weiequiv \SRT{1}{k}$ and $\Choice{k}' \weiequiv \RT{1}{k}$.
\end{proposition}
\begin{proof}
	This is essentially \cite[Prop.\ 3.4]{BRramsey17}, as $\Choice{k}' \weiequiv \mathsf{BWT}_k$ \cite[Fact 2.3(1)]{BRramsey17}.
\end{proof}

In particular, although $\SRT{1}{k}$ and $\RT{1}{k}$ are not first-order problems in their standard formulation (i.e., if we require the output to be an infinite homogeneous set), they are Weihrauch-equivalent to a first-order problem, namely the partial multifunction with the same instances but outputting the color of an infinite homogeneous set.

We now turn our attention to Ramsey's theorem for $n$-tuples and $k$-colors, for $n,k\ge 2$. We first show that, for every $n$ and $k\ge 2$, $\Choice{\mathbb{N}}^{(n)}\not\weireducible \RT{n+1}{k}$. This answers negatively an open question raised by Brattka and Rakotoniaina \cite[Question 7.10]{BRramsey17}.

\begin{lemma}
	\thlabel{thm:sigma0n_bound_jump}
	For every $n>0$, $\CodedBound{\boldfaceSigma^0_{n+1}}\weiequiv \CodedBound{\boldfaceSigma^0_1}^{(n)}$.
\end{lemma}
\begin{proof}
	This follows from the fact that, for every $n$, $\CodedBound{\boldfaceSigma^0_{n+1}}\weiequiv \CodedBound{\boldfaceSigma^0_n}'$. 
	Indeed, observe that a name for a $\boldfaceSigma^0_n$ set $B$ can be thought of as a string $p\in\Cantor$ s.t.\ 
	\[ x \in B \iff (\exists i_0)(\forall i_1)\hdots (Q i_{n-1})(p(\coding{x,i_0,\hdots,i_{n-1}})=1), \]
	where $Q=\exists$ if $n$ is odd and $Q=\forall$ otherwise. For every convergent sequence $\sequence{p_i}{i\in\mathbb{N}}$ and every $j\in\mathbb{N}$, $\lim_{i\to\infty} p_i(j)=1$ is a $\Delta^0_2$ property (relatively to the sequence). 
	
	To prove that $\CodedBound{\boldfaceSigma^0_n}' \weireducible \CodedBound{\boldfaceSigma^0_{n+1}}$ notice that, if $\sequence{q_i}{i\in\mathbb{N}}$ is a sequence in $\Cantor$ converging to a name $q$ for a $\boldfaceSigma^0_n$ set $B$, then we can uniformly compute a name for the set 
	\[ \{ x\in\mathbb{N}\st (\exists i_0)(\forall i_1)\hdots (Q i_{n-1})(q(\coding{x,i_0,\hdots,i_{n-1}})=1) \}, \]
	which is clearly $\lightfaceSigma^0_{n+1}$ relatively to $\sequence{q_i}{i\in\mathbb{N}}$.

        We now prove the converse reduction, i.e.\ that $\CodedBound{\boldfaceSigma^0_{n+1}}\weireducible \CodedBound{\boldfaceSigma^0_n}'$. Let $p$ be a name for a $\boldfaceSigma^0_{n+1}$ set $B$, i.e.\ a string $p\in \Cantor$ s.t.\
        \[ x \in B \iff (\exists i_0)(\forall i_1)\hdots (Q i_{n})(p(\coding{x,i_0,\hdots,i_{n}})=1). \]
        For every $m\in\mathbb{N}$, we define the sequence $p_m\in\Cantor$ as follows: for every  $x,i_0,\hdots,i_{n-1}\in \mathbb{N}$, 
        \[ p_{m}(\coding{x,i_0,\hdots,i_{n-1}})=1 \iff (Qi_n<m)(p(\coding{x,i_0,\hdots,i_{n-1},i_n})=1)), \]
        where $Q=\exists$ if $n$ is even and $Q=\forall$ otherwise.

        We claim that the sequence $\sequence{p_m}{m\in\mathbb{N}}$ converges. Let us fix $x,i_0,\hdots,i_{n-1}\in\mathbb{N}$. Suppose first that $Q=\exists$. Then, if there exists $\bar i$ such that $p(\coding{x,i_0,\hdots,i_{n-1},\bar i})=1$, then for every $m>\bar i$, $p_m(\coding{x,i_0,\hdots,i_{n-1}})=1$. Otherwise, for every $m$, $p_m(\coding{x,i_0,\hdots,i_{n-1}})=0$. In either way, we have that $\lim_{m\to\infty}p_m(\coding{x,i_0,\hdots,i_{n-1}})$ exists. A similar verification can be carried out if $Q=\forall$.

        Finally, we show that $p':=\lim_{m\to\infty}p_m$ is a name for the $\boldfaceSigma^0_n$ set $B$. Again, we suppose at first that $Q=\exists$. We then notice that
        \begin{align*}
          x\in B & \iff (\exists i_0)(\forall i_1)\hdots (\exists i_{n})(p(\coding{x,i_0,\hdots,i_{n}})=1) \phantom{\lim_{m\to\infty}} \\
            & \iff (\exists i_0)(\forall i_1)\hdots (\forall i_{n-1})(\exists m)(p_m(\coding{x,i_0,\hdots,i_{n-1}})=1) \phantom{\lim_{m\to\infty}} \\
            & \iff (\exists i_0)(\forall i_1)\hdots (\forall i_{n-1})(\lim_{m\to\infty} p_m(\coding{x,i_0,\hdots,i_{n-1}})=1) \\
            & \iff (\exists i_0)(\forall i_1)\hdots (\forall i_{n-1})(p'(\coding{x,i_0,\hdots,i_{n-1}})=1),  
        \end{align*}
        as we wanted. The verification for $Q=\forall$ is analogous.         
\end{proof}

\begin{lemma}
	\thlabel{thm:LPO^k*WKL}
	For every $k\in\mathbb{N}$,  $\LPO^k \compproduct\WKL \weiequiv \WKL \times   \LPO^k$.
\end{lemma}
\begin{proof}
	The statement is trivial for $k=0$, hence assume $k>0$. The right-to-left reduction is straightforward. To prove the left-to-right one, we first observe that $\WKL \weiequiv \TCCantor$ \cite[Prop.\ 6.1]{BGCompOfChoice19} and that $\WKL$ is closed under parallel product. We show that $\LPO^k \compproduct\WKL \weireducible \WKL\times  \TCCantor^k \times \LPO^k$.
	Fix  $(w,T)\in \dom(\LPO^k \compproduct\WKL)$. By definition, for every $x\in [T]$, $\Phi_w(x)=\pairing{p^x_i}_{i<k}$, where each $p^x_i$ is an input for $\LPO$. For every $i<k$, we can define a tree $S_i$ s.t.\ 
	\[ [S_i] = \{ x \in [T] \st \LPO(p^x_i)=0 \}. \]
	For every $j\in \{ 1,\hdots, k\}$, fix an enumeration $\sequence{\sigma^j_a}{a< \binom{k}{j}}$ of the strings in $2^k$ with exactly $j$ ones. We can define a tree $Z_j$ s.t.\ 
	\[ [Z_j] = \left\{\sigma^j_a\concat y \in \Baire: a< \binom{k}{j} \land  (\forall i\st \sigma^j_a(i)=1 )(y \in [S_i]) \right\}. \]
    Define also $q_j\in\Baire$ s.t.\ $\LPO(q_j)=1$ iff $Z_j$ is well-founded. Notice that all the $Z_j$ and the $q_j$ are uniformly computable from $(w,T)$.

    Let the forward functional $\Phi$ be the map that sends $(w,T)$ to $(T,\sequence{Z_j}{1\leq j \leq k}, \sequence{q_j}{1\leq j \leq k})$. Let also the backward functional $\Psi$ be the map that, upon input $(x,\sequence{y_j}{1\leq j\leq k}, \sequence{b_j}{1\leq j\leq k})$, works as follows: if $b_i=1$ for every $i<k$ then return $(x,1^k)$. Otherwise, let $\bar j$ be the largest $j$ such that $b_j=0$, and let $y_{\bar j} = \sigma^{\bar j}_a \concat x_{\bar j}$ for some $a< \binom{k}{\bar j}$. Then, on input $(x,\sequence{y_j}{1\leq j\leq k}, \sequence{b_j}{1\leq j\leq k})$, $\Psi$ returns $(x_{\bar j},\tau)$ with $\tau:= i \mapsto 1-\sigma^{\bar j}_a(i)$.
    
    We now show that $(x_{\bar j},\tau)$ is a valid solution to the instance $(w,T)$ of $\lpo^k*\WKL$ (the case of $\Psi$ outputting $(x,1^k)$ being trivial). It is immediate to verify that $x_{\bar j}$ is indeed a path through $T$, so we only have to check that $\tau$ is a valid $\lpo^k$-solution to $\Phi_w(x_{\bar j})$. Observe that, by construction, if $b_j = 1$ for some $j$, then $b_l=1$ for every $l\ge j$. By the definition of $Z_{\bar j}$, we know that, for every $i<k$ s.t.\ $\sigma^{\bar j}_a(i)=1$, $\lpo(p^{x_{\bar j}}_i)=0$. Therefore, we only need to show that if $\sigma^{\bar j}_a(i)=0$ then $\lpo(p^{x_{\bar j}}_i)=1$. Suppose towards a contradiction that this is not the case. Fix $i$ s.t.\ $\sigma^{\bar j}_a(i)=0=\lpo(p^{x_{\bar j}}_i)$. In particular, this implies that $\bar j<k$ (as $\sigma^{\bar j}_a$ has exactly $\bar j$ ones). Let $\rho$ be the string defined as $\rho(i)=1$ and, for $h\neq i$, $\rho(h):=\sigma^{\bar j}_a(h)$. Then clearly, $\rho$ has $\bar j +1$ ones, which implies that $Z_{\bar j+1}$ is ill-founded, against the maximality of $\bar j$: if $Z_{\bar j+1}$ is ill-founded then $b_{\bar j+1}=0$. 
\end{proof}

The following proof was obtained jointly with Arno Pauly.

\begin{theorem}\thlabel{thm:cnnjump-rtn1k}
	For every $n$ and $k\ge 2$, $\CNatural^{(n)} \not \weireducible \RT{n+1}{k}$.
\end{theorem}
\begin{proof}
	The claim for $n=0$ follows from \cite[Prop.\ 7.8]{BRramsey17}, hence assume $n>0$. 
	We prove a slightly stronger statement, namely 
	\[ \CodedBound{\boldfaceSigma^0_{n+1}}\not\weireducible \RT{1}{k} \compproduct \WKL^{(n)}. \]
	Observe indeed that $\CodedBound{\boldfaceSigma^0_{n+1}} \weireducible \CNatural^{(n)}$, as $\CNatural^{(n)} \weiequiv  \codedChoice{\boldfacePi^0_{n+1}}{}{\mathbb{N}}$ (hence we can just use $\CNatural^{(n)}$ to choose a bound for a given finite $\boldfaceSigma^0_{n+1}$ set). Moreover, since $\RT{n+1}{k}\weireducible \RT{n}{k}\compproduct \WKL'$ \cite[Cor.\ 4.14]{BRramsey17}, and using the fact that $(\WKL')^{[n]} \weiequiv \WKL^{(n)}$ \cite[Thm.\ 7.53(2)]{BGP17}, we obtain 
	\[ \RT{n+1}{k} \weireducible  \RT{1}{k}\compproduct \WKL^{(n)}. \] 

	We identify $\RT{1}{k}$ with the Weihrauch-equivalent problem of computing the \emph{color} of a homogeneous solution for the given coloring. 

	Assume that $\CodedBound{\boldfaceSigma^0_{n+1}}\weireducible \RT{1}{k} \compproduct \WKL^{(n)}$ via $\Phi,\Psi$. In particular, for every name $p_A$ of some $A\in\dom(\CodedBound{\boldfaceSigma^0_{n+1}})$, $\Phi(p_A)$ is (a name for) a pair $(w,T)\in \dom(\RT{1}{k} \compproduct \WKL^{(n)})$. Observe that, since $\RT{1}{k}$ has codomain $k$, for every path $x\in [T]$, there is $j<k$ s.t.\ $\Psi(p_A,x,j)(0)\downarrow$. Moreover, since $\CodedBound{\boldfaceSigma^0_{n+1}}(A)$ is upwards closed, for every $x \in [T]$ we have 
	\[ \max \{ b\in\mathbb{N} \st (\exists j<k)( \Psi(p_A, x,j)(0)\downarrow = b)\} \in \CodedBound{\boldfaceSigma^0_{n+1}}(A). \]
	In particular, this implies that $\CodedBound{\boldfaceSigma^0_{n+1}}$ can be uniformly computed from $\WKL^{(n)}$ if the backward functional is allowed $k$ mind-changes. This can be made precise writing $\CodedBound{\boldfaceSigma^0_{n+1}}\weireducible \LPO^k\compproduct \WKL^{(n)}$, i.e.
	\[ \LPO^k\to \CodedBound{\boldfaceSigma^0_{n+1}} \weireducible \WKL^{(n)}. \]

	Applying \thref{thm:sigma0n_bound_jump} we obtain
	\[ \LPO^k\to \CodedBound{\boldfaceSigma^0_1}^{(n)} \weireducible \WKL^{(n)}. \]
	Moreover, by \cite[Thm.\ 3.13]{BP16}, a representative of $\LPO^k\to \CodedBound{\boldfaceSigma^0_1}^{(n)}$ is the following problem: given in input $x\in\dom(\CodedBound{\boldfaceSigma^0_1}^{(n)})$, $\LPO^k\to \CodedBound{\boldfaceSigma^0_1}^{(n)}$ produces a pair $(p,\pairing{y_i}_{i<k})\in \Baire \times \Baire$ s.t.\ $\Phi_p(\pairing{\LPO(y_i)}_{i<k})\in \CodedBound{\boldfaceSigma^0_1}^{(n)}(x)$. In particular, 
	\[ \LPO^k\to \CodedBound{\boldfaceSigma^0_1}^{(n)} =(\LPO^k\to \CodedBound{\boldfaceSigma^0_1})^{(n)}.\] 
	By the jump inversion theorem \cite[Thm.\ 11]{BHK2017MonteCarlo}, we have 
	\[ \LPO^k\to \CodedBound{\boldfaceSigma^0_1} \weireducible^{\emptyset^{(n)}} \WKL, \]
	where $\weireducible^{\emptyset^{(n)}}$ denotes Weihrauch-reducibility relatively to the $n$-th Turing jump of $\emptyset$. We therefore have 
	\[  \UCNatural \weiequiv \CodedBound{\boldfaceSigma^0_1} \weireducible^{\emptyset^{(n)}} \LPO^k \compproduct\WKL \weireducible \WKL \compproduct \LPO^k, \]
	where the second reduction follows from \thref{thm:LPO^k*WKL}.

	Observe that the choice elimination principle \cite[Thm.\ 7.25]{BGP17} relativizes to any oracle (see also \cite[Thm.\ 3.9]{GPVDescSeq}), hence we obtain
	\[ \UCNatural \weireducible^{\emptyset^{(n)}} \LPO^k,\]
	which contradicts \thref{thm:cn_finite_codomain} as $\CNatural \weiequiv \UCNatural$.
\end{proof}

We are now ready to prove the following bounds for the first-order parts of $\SRT{n}{k}$ and $\RT{n}{k}$. We underline that, for every $k>2$, $(\Choice{2})^{(n)} \strictlyweireducible (\Choice{k})^{(n)}$. This is well-known in the literature and can be easily proved using e.g.\ \cite[Prop.\ 7.18]{BGP17} and the jump inversion theorem \cite[Thm.\ 11]{BHK2017MonteCarlo}.

\begin{theorem}
	\thlabel{thm:fop_rtnk}
    For every $n,k\geq 2$,
    \[
        (\Choice{k})^{(n)}\strictlyweireducible\firstOrderPart{(\SRT{n}{k})}\weireducible \firstOrderPart{(\rt^n_k)}\strictlyweireducible(\Choice{2}^*)^{(n)}.
    \]
\end{theorem}
\begin{proof}
    The reduction $(\Choice{k})^{(n)}\weireducible\firstOrderPart{(\SRT{n}{k})}$ can be proved observing that the proof of \cite[Theorem 3.5]{BRramsey17} generalizes to any number of colors. To prove that the reduction is strict, recall that $\CNatural$ is not computed by any partial multi-valued function with finite codomain (\thref{thm:cn_finite_codomain}). On the other hand, using \cite[Prop.\ 7.8 and Cor.\ 3.25]{BRramsey17} we have $\CNatural\weireducible\SRT{2}{2}$, and hence clearly $\CNatural\weireducible\SRT{n}{k}$ for every $n,k\geq 2$.
    
    Finally, $\firstOrderPart{(\RT{n}{k})}\weireducible(\Choice{2}^*)^{(n)}$ follows from $\RT{n}{k}\weireducible \WKL^{(n)}$ \cite[Corollary 4.18]{BRramsey17}, $\firstOrderPart{(\WKL^{(n)})}\weiequiv (\Choice{2}^*)^{(n)}$ (\thref{cor:fop-wkl}) and the Weihrauch-degree theoreticity of the first-order part operator. The fact that the reduction is strict follows from  $\CNatural^{(n-1)} \not \weireducible \RT{n}{k}$ (\thref{thm:cnnjump-rtn1k}) and $\CNatural^{(n-1)} \weireducible (\Choice{2}^*)^{(n)}$ (by \thref{cor:fop-wkl}, as $\CNatural^{(n-1)} \strongweireducible \WKL^{(n)}$).
\end{proof}

In the case of Ramsey's theorem for pairs, we can refine the result as follows:

\begin{corollary}
	For every $k\geq 2$,
    \[ (\Choice{k})''\strictlyweireducible\firstOrderPart{(\SRT{2}{k})}\strictlyweireducible \firstOrderPart{(\RT{2}{k})}\strictlyweireducible(\Choice{2}^*)''. \]
\end{corollary}
\begin{proof}
	In light of \thref{thm:fop_rtnk}, we only need to show that $\firstOrderPart{(\RT{2}{k})}\not\weireducible \firstOrderPart{(\SRT{2}{k})}$. This follows from the fact that $\RT{1}{\mathbb{N}}\weireducible\RT{2}{2}$ \cite[Cor.\ 3.25]{BRramsey17}, while $\RT{1}{k+1}\not\weireducible\SRT{2}{k}$ \cite[Prop.\ 3.4(2) and Cor.\ 6.6]{BRramsey17}. 
\end{proof}

We now refine the lower bounds on the degree of $\firstOrderPart{(\srt^2_k)}$. In fact, $\srt^2_2$ computes several first-order problems that are not below $(\Choice{k})''$, for every $k\geq 2$. To be more precise, we will see that these problems are actually solved by a multi-valued function weaker than $\srt^2_2$.
We summarize the results concerning $\firstOrderPart{(\srt^2_k)}$ and $\firstOrderPart{(\rt^2_k)}$ in Figure~\ref{fig:summary}. 

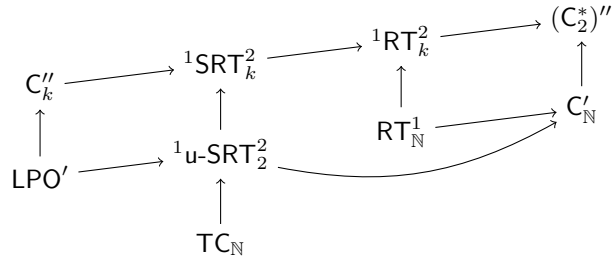
\begin{figure}[ht]
	\begin{center}
		\begin{tikzpicture}[scale=0.6]
			\node (ck) at (-4, -1.5) {$\Choice{k}''$};
			\node (srt) at (0, -1) {$\firstOrderPart{\SRT{2}{k}}$};
			\node (rt2k) at (4, -0.5) {$\firstOrderPart{\RT{2}{k}}$};
			\node (c2*) at (8, 0) {$(\Choice{2}^*)''$};
			\node (lpo) at (-4, -3.5) {$\LPO'$};
			\node (usrt) at (0, -3) {$\firstOrderPart{\usrt^2_2}$};
			\node (rt1n) at (4, -2.5) {$\RT{1}{\mathbb{N}}$};
			\node (cn') at (8, -2) {$\CNatural'$};
			\node (tcn) at (0, -5) {$\tcn$};

			\draw [->] (ck) to (srt);
			\draw [->] (srt) to (rt2k);
			\draw [->] (rt2k) to (c2*);
			\draw [->] (tcn) to (usrt);
			\draw [->] (lpo) to (usrt);
			\draw [->] (usrt) to (srt);
			\draw [->] (lpo) to (ck);
			\draw [->] (rt1n) to (rt2k);
			\draw [->] (cn') to (c2*);
			\draw [->] (rt1n) to (cn');
			\draw [->] (usrt) to [bend right=20] (cn');
			\end{tikzpicture}
		\caption{Reductions and non-reductions concerning the first-order parts of $\srt^2_k$ and $\rt^2_k$ for any $k\ge 2$. The arrows indicate strict Weihrauch reduction in the direction of the arrow. Every reduction which is not in the transitive closure of the diagram is known to not hold.}
		\label{fig:summary}
	\end{center}
\end{figure}

\begin{definition}
	$\usrt^2_k$ is the restriction of $\srt^2_k$ to \emph{unbalanced colorings}, i.e.\ to colorings $c\function{[\Nb]^n}{k}$ such that there exists an $i<k$ such that every infinite $c$-homogeneous set has color $i$.
\end{definition}
		
As already mentioned, it is known that $\CNatural\weireducible\srt^2_2$. The next result is an improvement in this reduction.
		
\begin{proposition}\thlabel{thm:tcn-usrt}
	$\tcn\weireducible\usrt^2_2$.
\end{proposition}
\begin{proof}
	Let $p\in\Baire$ be an instance of $\tcn$. We define an instance $c\function{[\Nb]^2}{2}$ of $\usrt^2_2$ in stages, in such a way that, at every stage $s$, the color of pairs $(x,y)$ with $x\leq s$ and $y\leq s+1$ is decided.

	Consider the computable function $g\pfunction{\mathbb{N}}{\mathbb{N}}$ that maps $s\in\mathbb{N}$ to the smallest $n\in\mathbb{N}$ s.t.\ $(\forall i<s)(p(i)\neq n+1)$, i.e.\ $n$ is the smallest number not enumerated by $p$ by stage $s$. Observe that $g$ is non-decreasing.
	
	At stage $0$, we set $c(0,1):=0$. At step $s>0$, we check if $g(s)=g(s-1)$. If yes then for every $x<s$ we set $c(x,s+1):=c(x,s)$, and $c(s,s+1):=0$. Otherwise, if $g(s)>g(s-1)$ then, for every $x\leq s$, we set $c(x,s+1):=1$. This concludes the construction.
	
	It is easy to see that the coloring is stable: for every $x$, there exists at most one $y$ such that $c(x,y)\neq c(x,y+1)$. To show that $c$ is unbalanced, we consider two cases: if, for infinitely many steps $s$, it happens that $g(s)>g(s-1)$ then, for every $x$, $\lim_y c(x,y)=1$, which implies that $c$ is unbalanced. If instead there are only finitely many $s$ as above, then for cofinitely many $x$ it holds that, for every $y>x$, $c(x,y)=0$, which again yields that $c$ is unbalanced.
	
	Let $H$ be an infinite $c$-homogeneous set: if it is $c$-homogeneous for color $1$, then, from the discussion above, it follows that $p$ is an enumeration of all of $\Nb$, hence we can output any number. If instead $H$ is $c$-homogeneous for $0$, let $h$ be the minimum element of $H$, and let $m:=g(h)$. By the construction, we know that at no future stage will $m$ be enumerated by $p$, which means that $m$ is a valid solution to $p$, thus proving the reduction. 
\end{proof}
		
Let $\sort\function{\Cantor}{\Cantor}$ be the function that, on input $p$, outputs $0^n\concat 1^\Nb$ if there are exactly $n$ positions $i$ such that $p(i)=1$, and $0^\Nb$ if there are infinitely many such positions $i$. The problem $\sort$ was introduced in \cite[Sec.\ 5]{NP18}, and was studied in \cite{DGHPP18} to calibrate the Weihrauch-degree of the problem $\mathsf{CFI}_{\Delta^0_2}$ of finding an infinite subset of a cofinite set given by a $\Delta^0_2$ approximation. It is known that, for every $k\in\mathbb{N}$, $\sort \not \weireducible \Choice{k}''$: this follows from the fact that $\CNatural \strictlyweireducible \sort$ (see \cite[Prop.\ 24]{NP18}), so $\sort\weireducible\Choice{k}''$ would imply $\CNatural\weireducible\Choice{k}''$, which is impossible since $\CNatural$ reduces to no multi-valued function with finite range (\thref{thm:cn_finite_codomain}).
		
\begin{proposition}\thlabel{lem:sort-usrt}
	$\sort\weireducible\usrt^2_2$.
\end{proposition}
\begin{proof}
	Let an instance $p$ of $\sort$ be given. The proof follows closely the one of \thref{thm:tcn-usrt}, except that we define the function $g$ as the map that sends $s$ to the maximum $i<s$ s.t.\ $p(i)=1$. The definition of the coloring $c$ is otherwise identical. 
		
	By the proof above, we know that $c$ is an instance of $\usrt^2_2$. Let $H$ be an infinite $c$-homogeneous set. Again, if $H$ if $c$-homogeneous for color $1$, it means that there are infinitely many positions $i$ such that $p(i)=1$, and hence we can output $0^\Nb$. If instead $H$ is $c$-homogeneous for color $0$, we just have to notice that for every $x\in\Nb$, if $\lim_y c(x,y)=0$, then for every $y>x$ it holds that $c(x,y)=0$. In particular, this means that if $\lim_y c(x,y)=0$, then every position $i\in\Nb$ such that $p(i)=1$ is such that $i<x$. Hence, let $h$ be the minimal element of $H$: since $\lim_yc(h,y)=0$ (as $c$ is stable), we can count the number $n$ of positions $i<h$ such that $p(i)=1$, and we can thus output $0^n\concat 1^\Nb$, proving the reduction.
\end{proof}
		
\begin{proposition}\thlabel{cor:lpo'-usrt}
	$\lpo'\weireducible \usrt^2_2$.
\end{proposition}
\begin{proof}
	It is easy to see that $\lpo'$ is Weihrauch equivalent to the problem $\isinf\function{\Cantor}{2}$ which, given an infinite binary string, outputs $1$ if and only if there are infinitely many $n\in \Nb$ such that $p(n)=1$. Observe that, in the proof of \thref{lem:sort-usrt}, the solution $H$ is homogeneous for color $1$ iff there are infinitely many $i$ s.t.\ $p(i)=1$. In other words, the color of any homogeneous solution for the coloring $c$ defined as in the proof of \thref{lem:sort-usrt} is the correct solution for $\isinf(p)$.
\end{proof}

We remark that $\tcn$, $\lpo'$ and $\sort$ are pairwise incomparable, as proved in \cite{NP18}. From this, it follows that so are $\tcn$, $\Choice k''$ and $\sort$: we have already mentioned that $\tcn\not\weireducible\Choice{k}''$ and $\sort\not\weireducible\Choice{k}''$. The fact that $\Choice{k}''\not\weireducible\tcn$ and $\Choice{k}''\not\weireducible\sort$ follows from the known fact that $\lpo'\weireducible\Choice{k}''$ (see e.g. \cite[Thm.\ 7.55]{BGP17}).

It is natural to ask whether $\firstOrderPart{(\usrt^2_k)}\weiequiv\firstOrderPart{(\srt^2_k)}$. We now show that this is not the case, by proving that $\Choice{k}''\not\weireducible\usrt^2_k$. In order to do this, we recall the concept of Erdős-Rad\'o tree associated to a coloring of pairs.

\begin{definition}
    Let $c\function{[\Nb]^2}{k}$ be a coloring. The \emph{Erdős-Rad\'o tree} associated to $c$ is the tree $T_c\subseteq \Nb^{<\Nb}$ defined in stages as follows: at stage $0$, we enumerate $\str{}$ and $\str{0}$ in $T_c$. At stage $s+1$, we look for the longest string $\sigma$ enumerated in $T_c$ up to stage $s$ such that, for every $i<j<|\sigma|$, $c(\sigma(i),\sigma(j))=c(\sigma(i),s+1)$. We then enumerate $\sigma\concat \str{s+1}$ in $T_c$, and move to the next stage.
  \end{definition}

 We remark that, although not immediately apparent, the definition above is well-posed: at every stage $s+1$, only one string $\sigma$ is selected to be extended to $\sigma\concat(s+1)$. Indeed, suppose for a contradiction that there are two maximal strings $\sigma_0$ and $\sigma_1$ eligible to be extended to $\sigma_0\concat(s+1)$ and $\sigma_1\concat(s+1)$, and let $\sigma$ be their longest common initial segment: we know that $\sigma\neq ()$, since $(0)$ is an initial segment of $\sigma$. But then,
  \[ c(\sigma(|\sigma|-1),s+1)=c(|\sigma|-1,\sigma_0(|\sigma|))\neq c(\sigma(|\sigma|-1),\sigma_1(|\sigma|))= c(\sigma(|\sigma|-1),s+1),\]
  which is a contradiction.

Although the definition is somewhat different, the construction above yields the same tree obtained by the usual description of Erdős-Rad\'o tree (see e.g.\ \cite[Lem.\ III.7.4]{Simpson09}): we prefer the construction above because it is, in a sense, simpler to parse than the original, more general definition. For every coloring of pairs $c$, $T_c$ is an infinite finitely-branching tree and it is computable in the coloring. We remark that every number $n\in\Nb$ appears in $T_c$ exactly once. This allows us to prove the following result.

\begin{lemma}\thlabel{thm:erdosradotree}
    Let $c\function{[\Nb]^2}{k}$ be an instance of $\srt^2_k$. Then, $T_c$ has exactly one infinite path.
\end{lemma}
\begin{proof}
    Suppose for a contradiction that $T_c$ had two infinite paths $p_0$ and $p_1$, and let $\sigma$ be their longest common segment (by the definition of $T_c$, $\str{0}$ is an initial segment of $\sigma$, so $\sigma$ is non empty). Then, by definition of $T_c$, for every $n\in\Nb$,
	\[ c(\sigma(|\sigma|-1),p_0(|\sigma|+n))\neq c(\sigma(|\sigma|-1),p_1(|\sigma|+n)). \] 
	But this implies that $\lim_i c(\sigma(|\sigma|-1), i)$ does not exist, which contradicts the stability of $c$. 
\end{proof}

The result above allows us to show that $\usrt^2_2$ is $\Sigma^0_3$-measurable.
\begin{proposition}\thlabel{lem:usrtlim}
    $\usrt^2_k\weireducible\mflim'$, and hence $\firstOrderPart{\usrt^2_2}\weireducible \CNatural'$.
\end{proposition}
\begin{proof}
    We prove a slightly stronger statement, namely $\usrt^2_k\weireducible \mflim_k\compproduct\mflim$.
	
	Fix an unbalanced stable coloring $c$. We first claim that, given a path $p$ through $T_c$, we can use $\mflim_k$ to obtain an homogeneous solution for $c$.     For $i<k$, let $L_i:=\{n\in\Nb\st \lim_j c(p(n),j)=i \}$. Notice that exactly one of the $L_i$ is infinite, since otherwise there would be infinite $c$-homogeneous sets for more than one color, contradicting the assumption that $c$ is unbalanced.
    By the definition of $T_c$, for every $n\in\Nb$, $\lim_i c(p(n),i)=c(p(n),p(n+1))$, so we can use $\mflim_k$ to find the $i$ such that $L_i$ is infinite. This shows that $L_i$ is a solution to $c$, proving our claim.
	
	Observe that, even if $T_c$ is finitely branching (in fact, every string has at most $k$ immediate successors), the set $[T_c]$ is not computably compact (i.e.\ the map $n \mapsto \max\{m \st (\forall \sigma \in T)(\sigma(n)\le m)\}$ is not computable). Indeed, if it were computably-compact then, by \cite[Thm.\ 7.23]{BGP17} and \thref{thm:erdosradotree}, we could uniformly compute a path through $T_c$ (as $\UCCantor \weiequiv \id$, see e.g.\ \cite[Cor.\ 7.26]{BGP17}). However, since an homogeneous solution for $c$ can be obtained from a path through $T_c$ using $\mflim_k$, this would imply e.g.\ that $\totalization{\CNatural} \weireducible \mflim_k$, which is clearly a contradiction.

	To conclude the proof, it is enough to notice that $\mflim$ suffices to compute a bound $q$ for a path through $T_c$ ($q$ is a bound if there is $p\in [T_c]$ s.t.\ for every $n$, $p(n)\le q(n)$). Given such a bound $q$, we can compute the subtree $\{ \sigma \in T_c \st (\forall i<\length{\sigma})(\sigma(i)\le q(i)) \}$. Since this tree is infinite and $q$-computably compact, we can now apply $\UCCantor$ to obtain a path, and $\mflim_k$ to obtain a homogeneous solution for $c$. 

	The fact that $\firstOrderPart{\usrt^2_2}\weireducible \CNatural'$ follows from \thref{thm:fop_lim}.
\end{proof}

\begin{corollary}
    For every $\ell,k\geq 2$, $(\Choice{\ell})''\not\weireducible\usrt^2_k$, and hence $\firstOrderPart{\usrt^2_k} \strictlyweireducible \firstOrderPart{\SRT{2}{k}}$.
\end{corollary}
\begin{proof}
    By the fact that $\Choice{\ell}$ is $\mflim$-computable but not computable, it follows that $\Choice{\ell}''$ is $\mflim''$-computable but not $\mflim'$-computable (see also the proof of \cite[Prop.\ 9.1]{prob-comput-brattka-gherardi-holzl}). Hence, by \thref{lem:usrtlim}, the claim follows. The separation $\firstOrderPart{\usrt^2_k} \strictlyweireducible \firstOrderPart{\SRT{2}{k}}$ then follows from \thref{thm:fop_rtnk}.
\end{proof}

For completeness, we conclude this section by studying the relationship between $\rt^1_\Nb$ and the benchmark problems $\lpo'$, $\sort$, and $\tcn$. As a preliminary observation, we notice that the results we gave above already imply that none of $\lpo'$, $\sort$, and $\tcn$ computes $\rt^1_\Nb$, since $\rt^1_\Nb$ is not solved by $\srt^2_k$ for any $k\in\Nb$. We will now show that $\rt^1_\Nb$ is not powerful enough to compute any of the other three problems.

\begin{lemma}\thlabel{lem:lpo'-rt1n}
    $\rt^1_\Nb \weiincomparable \lpo'$.
\end{lemma}
\begin{proof}
   We already said that $\rt^1_\Nb\not\weireducible\lpo'$. Then, suppose for a contradiction that $\lpo'\weireducible\rt^1_\Nb$. This would then imply that $\mflim'\weiequiv\parallelization{\lpo'} \weireducible \parallelization{\rt^1_\Nb}\weiequiv\WKL'$, which is a contradiction.
\end{proof}

To show that $\rt^1_\Nb$ does not compute neither $\tcn$ nor $\sort$, we will need the following preliminary lemma.

\begin{lemma}\thlabel{lem:lport1n-rt1n}
    $\lpo\compproduct\rt^1_\Nb\weireducible\rt^1_\Nb$.
\end{lemma}
\begin{proof}
    Let us consider $(p,c)\in \dom (\lpo\compproduct\rt^1_\Nb)$, where $c\function{\Nb}{\Nb}$ is a coloring of the integers such that its range is bounded by some $k\in\Nb$. We build a coloring $f\function{\Nb}{\Nb}$, the range of which will be bounded by $2k$, as follows. 

    For every $s$, we examine the coloring $c$ until we find a $c$-homogeneous set of size $s$, call it $H_s$, and let $i_s:=c(H_s)$. We then run $\Phi_p(H_s)$ for $s$ steps, thus obtaining a finite string $\sigma_s$. There are now two cases: if $\sigma_s(n)= 0$ for every $n<|\sigma_s|$, we set $f(s):=2i_s$ (notice that this includes the case that $\sigma_s=\str{}$), otherwise, we put $f(s):=2i_s+1$. 
    
    Notice that $f$ is indeed an instance of $\rt^1_\Nb$, provided $c$ is: indeed, since for every $s$ there exists a $c$-homogeneous set of size $s$, $f$ is total, and its range is $2k$, where $k$ is the range of $c$.
    
    Let $\Phi$ be the Turing functional mapping a pair $(p,c)\in \Baire\times\Baire$, where $c\in\dom(\rt^1_\Nb)$, to $f$ defined as above. Moreover, let $\Psi$ be the Turing functional defined as follows: on every input $(p,c,H)$, where $(p,c)$ is as before and $H\subseteq \Nb$, let $i,b$ be s.t.\ $\Phi(p,c)(H)= 2i+b$, 
	and let $S_i:= \{x\in \Nb \st c(x)=i\}$. Then, $\Psi(p,c,H)$ outputs the pair $(S_i, b)$. We claim that the functionals $\Phi,\Psi$ as above witnesses the reduction $\lpo\compproduct\rt^1_\Nb\weireducible\rt^1_\Nb$.
    
    We have already shown that $f=\Phi(p,c)$ is an instance of $\rt^1_\Nb$. Let $H$ be an infinite $f$-homogeneous set. We start by showing that $S_i$ is an infinite $c$-homogeneous set. Notice that by definition, $S_i$ is $c$-homogeneous, so we only have to show that it is infinite. This follows easily from the observation that if $H$ is $f$-homogeneous for $2i+b$, by construction there are infinitely many $s$ such that $S_i$ has size at least $s$, which proves that $S_i$ is infinite. 
    
    Hence, we only have to show that $b$ is a valid answer to $\LPO(\Phi_p(S_i))$. We claim that  $H$ is $f$-homogeneous for an even color if and only if $\Phi_p(S_i)(n)= 0$ for every $n\in\mathbb{N}$: notice that, if we do this, then indeed $b$ is as we want. Suppose at first that $H$ is homogeneous for $2i$, for some $i<k$. Then, for every $s\in H$, by the definition of $f$ we have that there is a set $H_s$ of size $s$ $c$-homogeneous for $i$ and such that $\Phi_p(H_s)(n)= 0$ for every $n$ such that $\Phi_p(H_s)(n)$ converges. Notice moreover that, for every $s\in H$, $H_s$ is an initial segment of $S_i$: it follows that $\Phi_p(S_i)(n)= \lim_{s\in H, s\to\infty} \Phi_p(H_s)(n)= 0$ for every $n$ ($\Phi_p(S_i)$ is a total function, since we have already shown that $S_i$ is a $c$-homogeneous set). Hence, indeed, $0$ is a correct answer to the $\lpo$-instance $\Phi_p(S_i)$. Suppose next that $H$ is $f$-homogeneous for $2i+1$, for some $i$. Then, for every $s\in H$, the definition of $f$ gives us that there is a $c$-homogeneous $H_s$ such that, for some $n\in \Nb$, $\Phi_p(H_s)(n)=1$. It then follows from an analysis similar to the previous case that $1$ is a correct answer to the $\lpo$-instance $\Phi_p(S_i)$, which concludes the proof of the claim.
\end{proof}    

We now introduce two problems that were originally defined in \cite{NP18}.

\begin{definition} We define the following functions: 
    \begin{itemize}
        \item $\isfins \function{\Cantor}{\mathbb{S}}$ is the function such that, on input $p\in \Cantor$, outputs $\top$ if there are only finitely many $n$ such that $p(n)=1$, and $\bot$ otherwise.
        \item $\isinfs\function{\Cantor}{\mathbb{S}}$ is the function such that, on input $p\in \Cantor$, outputs $\top$ if there are infinitely many $n$ such that $p(n)=1$, and $\bot$ otherwise.
    \end{itemize}
\end{definition}

Notice that the two problems above are weakenings of $\isinf$ (which we have introduced in \thref{cor:lpo'-usrt}) and of $\isfin$, which is the problem with domain $\Cantor$ and codomain $2$ such that $\isfin(p)=1$ if and only if there are finitely many $n$ such that $p(n)=1$, which are both equivalent to $\lpo'$. It is clear that the difference between the $\isfins$ and $\isfin$ (and between $\isinfs$ and $\isinf$) is in the representation of the solution. Recalling that, for every $p\in \Cantor$, $\repmap{\mathbb{S}}(p)=0$ if and only if $p=0^\Nb$, the following result is immediate.

\begin{lemma}\thlabel{lem:obvisinffin}
    $\isfin\weireducible\lpo\compproduct\isfins$ and $\isinf\weireducible\lpo\compproduct\isinfs$.
\end{lemma}

We are now ready to prove that $\rt^1_\Nb$ does not compute neither $\tcn$ nor $\sort$.

\begin{proposition}
   $\sort\not\weireducible\rt^1_\Nb$ and $\tcn\not\weireducible\rt^1_\Nb$.
\end{proposition}
\begin{proof}
    By \cite[Prop.\ 24]{NP18}, $\isfins\weireducible\sort$ and $\isinfs\weireducible\tcn$. Now, suppose for a contradiction that $\sort\weireducible\rt^1_\Nb$: since $\isfin\weiequiv\lpo'$, by \thref{lem:lport1n-rt1n} and \thref{lem:obvisinffin} we would have that 
    \[
        \lpo'\weireducible\lpo\compproduct\isfins\weireducible\lpo\compproduct\rt^1_\Nb\weireducible\rt^1_\Nb,
    \]
    which contradicts \thref{lem:lpo'-rt1n}, thus proving that $\sort\not\weireducible\rt^1_\Nb$.
    
    The other non-reduction is proved analogously. 
\end{proof}

\section{Conclusions}
In this paper, we explored the algebraic relations between the first-order part operator and several other operators on multi-valued functions already introduced in the literature. We now recap the results and highlight a few questions that remained open. 

In Section~\ref{sec:alg_fop}, we study the relations between $\firstOrderPart{(\cdot)}$ and the operators $\sqcup,\sqcap, \times, \compproduct$, and $'$. Moreover, to better understand the connections between the first-order part and the parallelization operator, in Section~\ref{sec:u*} we introduced the unbounded finite parallelization $\ustar{(\cdot)}$ and explored its connections with the operators $\sqcup,\sqcap, \times, \parallelization{\phantom{a}}, \compproduct$, and $'$. 

In particular, \thref{thm:FOP(parallelization)=u*(FOP)} and its generalization \thref{thm:FOP(parallelization)=u*(FOP)_strong} characterize the first-order part of the parallelization of a first-order problem. As already mentioned, these theorems can be rephrased as the commutativity, for first-order problems between the first-order part and the unbounded finite parallelization. We can then update the Open Question \ref{q:commutativity} as follows:

\begin{open}
	Under which hypotheses on $f$ we have $\firstOrderPart{(\parallelization{f})}\weiequiv\ustar{(\firstOrderPart{f})}$? How about $\firstOrderPart{((\parallelization{f})^{(n)})}\strongweiequiv (\ustar{f})^{(n)}$, for $n\in\mathbb{N}$?
\end{open}

In Section~\ref{sec:fop_diamond}, we explored the connections between first-order part, unbounded finite parallelization, and the diamond operator. \thref{thm:complete_problems_u*_diamond} provides a sufficient condition under which $\ustar{(\cdot)}$ and $(\cdot)^\diamond$ coincide. Again, this is not a necessary condition, and therefore it is natural to ask:

\begin{open}
	Under which hypotheses on $f$ we have $\ustar{f}\weiequiv f^\diamond$?
\end{open}

Going in the opposite direction, it would be interesting to better understand when these operations are different. In \thref{thm:f*<fu*} and \thref{thm:*<u*_fractals} we provided some sufficient conditions under which the finite parallelization differs from its unbounded counterpart. The following question was raised by Arno Pauly:

\begin{open}
	Is there a first-order problem s.t.\ $f^*\strictlyweireducible \ustar{f}\strictlyweireducible f^\diamond$? How about $f\strictlyweireducible f^*\strictlyweireducible \ustar{f}\strictlyweireducible f^\diamond$?
\end{open}

In Section~\ref{sec:applications} we characterize the first-order part of several well-known principles, including $\mflim^{(n)}$ and $\WKL^{(n)}$. A possible candidate to answer positively the first part of the last question is $\PiBound$: indeed, we proved in \thref{thm:pibound*<u*} that $f^*\strictlyweireducible \ustar{f}$. The proof strategy differs from the ones used to prove \thref{thm:f*<fu*} and \thref{thm:*<u*_fractals}, and can certainly be applied to other bounding principles as well. It is not clear, however, whether $\PiBound^\diamond \weireducible \ustar{\PiBound}$. Besides, $\PiBound \weiequiv \PiBound^*$, hence it is not a viable example to answer positively the second part of the last question.

Finally, we provided some bounds on the first-order part of some principles related to Ramsey theorem. The following problem remains open:

\begin{open}
	For $n>2$, $\firstOrderPart{\RT{n}{k}} \weireducible \firstOrderPart{\SRT{n}{k}}$?
\end{open}

\vspace{-3ex}

\bibliographystyle{mbibstyle}
\bibliography{bibliography}

\providecommand{\bysame}{\leavevmode\hbox to3em{\hrulefill}\thinspace}
\providecommand{\MR}{\relax\ifhmode\unskip\space\fi MR }
\providecommand{\MRhref}[2]{%
  \href{http://www.ams.org/mathscinet-getitem?mr=#1}{#2}
}
\providecommand{\href}[2]{#2}
\begin{thebibliography}{10}

\bibitem{KiharaADauriacChoice}
Angl{\`e}s~d'Auriac,  Paul-Elliot and Kihara,  Takayuki, \emph{A Comparison Of
  Various Analytic Choice Principles}, The Journal of Symbolic Logic
  \textbf{86} (2021), no.~4, 1452--1485, \doi{10.1017/jsl.2021.37}.

\bibitem{BdBPLow12}
Brattka,  Vasco, de~Brecht,  Matthew, and Pauly,  Arno, \emph{Closed choice and
  a {U}niform {L}ow {B}asis {T}heorem}, Annals of Pure and Applied Logic
  \textbf{163} (2012), no.~8, 986--1008, \doi{10.1016/j.apal.2011.12.020}.

\bibitem{Dagstuhl18}
Brattka,  Vasco, Dzhafarov,  Damir~D., Marcone,  Alberto, and Pauly,  Arno,
  \emph{Measuring the Complexity of Computational Content: From Combinatorial
  Problems to Analysis (Dagstuhl Seminar 18361)}, Dagstuhl Reports \textbf{8}
  (2019), no.~9, 1--28, \doi{10.4230/DagRep.8.9.1}.

\bibitem{BG11}
Brattka,  Vasco and Gherardi,  Guido, \emph{Effective Choice and Boundedness
  Principles in Computable Analysis}, The Bulletin of Symbolic Logic
  \textbf{17} (2011), no.~1, 73--117.

\bibitem{BG09}
\bysame, \emph{{W}eihrauch degrees, omniscience principles and weak
  computability}, The Journal of Symbolic Logic \textbf{76} (2011), no.~1,
  143--176.

\bibitem{BGCompOfChoice19}
\bysame, \emph{Completion of Choice}, Annals of Pure and Applied Logic
  \textbf{172} (2021), no.~3, 102914, \doi{10.1016/j.apal.2020.102914}.

\bibitem{prob-comput-brattka-gherardi-holzl}
Brattka,  Vasco, Gherardi,  Guido, and Hölzl,  Rupert, \emph{Probabilistic
  computability and choice}, Information and Computation \textbf{242} (2015),
  249--286, \doi{10.1016/j.ic.2015.03.005}.

\bibitem{BolWei11}
Brattka,  Vasco, Gherardi,  Guido, and Marcone,  Alberto, \emph{The
  Bolzano-Weierstrass Theorem is the jump of Weak K{\"o}nig's Lemma}, Annals of
  Pure and Applied Logic \textbf{163} (2012), no.~6, 623--655,
  \doi{10.1016/j.apal.2011.10.006}.

\bibitem{BGP17}
Brattka,  Vasco, Gherardi,  Guido, and Pauly,  Arno, \emph{{W}eihrauch
  {C}omplexity in {C}omputable {A}nalysis}, pp.~367--417, Springer
  International Publishing, 2021, \doi{10.1007/978-3-030-59234-9_11}.

\bibitem{BHK2017}
Brattka,  Vasco, Hendtlass,  Matthew, and Kreuzer,  Alexander~P., \emph{On the
  Uniform Computational Content of Computability Theory}, Theory of Computing
  Systems \textbf{61} (2017), no.~4, 1376--1426,
  \doi{10.1007/s00224-017-9798-1}.

\bibitem{BHK2017MonteCarlo}
Brattka,  Vasco, H\"olzl,  Rupert, and Kuyper,  Rutger, \emph{Monte Carlo
  Computability}, 34th Symposium on Theoretical Aspects of Computer Science
  (STACS 2017) (Dagstuhl, Germany) (Vollmer,  Heribert and Vall\'ee,  Brigitte,
  eds.), Leibniz International Proceedings in Informatics (LIPIcs), vol.~66,
  Schloss Dagstuhl--Leibniz-Zentrum fuer Informatik, 2017,
  \doi{10.4230/LIPIcs.STACS.2017.17}, pp.~17:1--17:14.

\bibitem{BP16}
Brattka,  Vasco and Pauly,  Arno, \emph{On the algebraic structure of Weihrauch
  degrees}, Logical Methods in Computer Science \textbf{14} (2018), no.~4,
  1--36, \doi{10.23638/LMCS-14(4:4)2018}.

\bibitem{BRramsey17}
Brattka,  Vasco and Rakotoniaina,  Tahina, \emph{On the uniform computational
  content of Ramsey's theorem}, The Journal of Symbolic Logic \textbf{82}
  (2017), no.~4, 1278--1316, \doi{10.1017/jsl.2017.43}.

\bibitem{CMVCantorBendixson}
Cipriani,  Vittorio, Marcone,  Alberto, and Valenti,  Manlio, \emph{{T}he
  {W}eihrauch lattice at the level of $\boldsymbol{\Pi}^1_1\mathsf{-CA}_0$: the
  {C}antor-{B}endixson theorem}, submitted, available at
  \url{https://arxiv.org/abs/2210.15556}.

\bibitem{DDHMS16}
Dorais,  Fran{\c{c}}ois~G., Dzhafarov,  Damir~D., Hirst,  Jeffry~L., Mileti,
  Joseph~R., and Shafer,  Paul, \emph{On uniform relationships between
  combinatorial problems}, Transactions of the American Mathematical Society
  \textbf{368} (2016), 1321--1359, \doi{10.1090/tran/6465}.

\bibitem{DGHPP18}
Dzhafarov,  Damir~D., Goh,  Jun~Le, Hirschfeldt,  Denis~R., Patey,  Ludovic,
  and Pauly,  Arno, \emph{Ramsey's theorem and products in the Weihrauch
  degrees}, Computability \textbf{9} (2020), no.~2, 85--110,
  \doi{10.3233/COM-180203}.

\bibitem{DSYFirstOrder}
Dzhafarov,  Damir~D., Solomon,  Reed, and Yokoyama,  Keita, \emph{On the
  first-order parts of problems in the Weihrauch degrees}, preprint, 2023,
  available at \url{https://arxiv.org/abs/2301.12733}.

\bibitem{GPVDescSeq}
Goh,  Jun~Le, Pauly,  Arno, and Valenti,  Manlio, \emph{Finding descending
  sequences through ill-founded linear orders}, The Journal of Symbolic Logic
  \textbf{86} (2021), no.~2, 817--854, \doi{10.1017/jsl.2021.15}.

\bibitem{HJ16}
Hirschfeldt,  Denis~R. and Jockusch,  Carl~G., \emph{On notions of
  computability-theoretic reduction between {$\Pi^1_2$} principles}, Journal of
  Mathematical Logic \textbf{16} (2016), no.~1, 1650002(1--59),
  \doi{10.1142/s0219061316500021}.

\bibitem{HirstLM}
Hirst,  Jeffry~L., \emph{Leaf management}, Computability \textbf{9} (2020),
  no.~3-4, 309--314, \doi{10.3233/COM-180243}.

\bibitem{KMP20}
Kihara,  Takayuki, Marcone,  Alberto, and Pauly,  Arno, \emph{Searching for an
  analogue of {$ATR_0$} in the Weihrauch lattice}, The Journal of Symbolic
  Logic \textbf{85} (2020), no.~3, 1006--1043, \doi{10.1017/jsl.2020.12}.

\bibitem{MVRamsey}
Marcone,  Alberto and Valenti,  Manlio, \emph{The open and clopen Ramsey
  theorems in the Weihrauch lattice}, The Journal of Symbolic Logic \textbf{86}
  (2021), no.~1, 316--351, \doi{10.1017/jsl.2021.10}.

\bibitem{NP18}
Neumann,  Eike and Pauly,  Arno, \emph{A topological view on algebraic
  computation models}, Journal of Complexity \textbf{44} (2018), 1--22,
  \doi{10.1016/j.jco.2017.08.003}.

\bibitem{Patey2016}
Patey,  Ludovic, \emph{The weakness of being cohesive, thin or free in reverse
  mathematics}, Israel Journal of Mathematics volume \textbf{216} (2016),
  905--955, \doi{10.1007/s11856-016-1433-3}.

\bibitem{Pauly16}
Pauly,  Arno, \emph{On the topological aspects of the theory of represented
  spaces}, Computability \textbf{5} (2016), no.~2, 159--180,
  \doi{10.3233/COM-150049}.

\bibitem{Simpson09}
Simpson,  Stephen~G., \emph{Subsystems of Second Order Arithmetic}, 2 ed.,
  Cambridge University Press, Cambridge, 2009.

\bibitem{Soare16}
Soare,  Robert~I., \emph{{T}uring {C}omputability: {T}heory and
  {A}pplications}, 1 ed., Springer, 2016, \doi{10.1007/978-3-642-31933-4}.

\bibitem{Weihrauch00}
Weihrauch,  Klaus, \emph{Computable Analysis: An Introduction}, 1 ed.,
  Springer-Verlag, Berlin, 2000.

\bibitem{Westrick20diamond}
Westrick,  Linda~Brown, \emph{A note on the diamond operator}, Computability
  \textbf{10} (2021), no.~2, 107--110, \doi{10.3233/COM-20029}.

\end{thebibliography}

\printauthor

\end{document}